\documentclass[reqno,11pt]{amsart}
 \usepackage[dvipsnames]{xcolor}
\usepackage{setspace,tikz,mathrsfs,listings,multicol,amssymb}
\usepackage{rotating}
\usepackage[vcentermath]{youngtab}
\usepackage[margin=0.9in,includefoot,footskip=30pt]{geometry}
\usepackage{enumerate}
\usepackage{booktabs}
\usepackage[centertableaux]{ytableau}
\usepackage[all,cmtip]{xy}
\usetikzlibrary{arrows,matrix}
\usepackage[colorlinks=true, pdfstartview=FitV, linkcolor=black, citecolor=black, urlcolor=black]{hyperref}
\usepackage{cleveref}
\usepackage{relsize}
\usepackage{fix-cm}
\usepackage{mathrsfs}
\usepackage{tikz}\usetikzlibrary{graphs,quotes,fit,positioning,matrix,calc,decorations.markings,angles,decorations.pathmorphing,decorations.pathreplacing}
\tikzset{math mode/.style = {execute at begin node=$, execute at end node=$}}
\tikzset{arrow/.style={postaction={decorate,thick,decoration={markings,mark = at position #1 with {\arrow{>}}}}},arrow/.default=0.5}
\tikzset{invarrow/.style={postaction={decorate,thick,decoration={markings,mark = at position #1 with {\arrow{<}}}}},invarrow/.default=0.5}
\tikzset{bosonic/.style={ultra thick}}
\tikzset{tab/.style={matrix of math nodes,column sep=-.35, row sep=-.35,text height=7pt,text width=7pt,align=center,inner sep=2,font=\footnotesize}}
\newcommand\tikzif[2][]{
\tikzifinpicture{#2}{\begin{tikzpicture}[#1]#2\end{tikzpicture}}
}
\tikzset{math mode/.style = {execute at begin node=$, execute at end node=$}}
\tikzset{arrow/.style={postaction={decorate,thick,decoration={markings,mark = at position #1 with {\arrow{>}}}}},arrow/.default=0.5}
\tikzset{invarrow/.style={postaction={decorate,thick,decoration={markings,mark = at position #1 with {\arrow{<}}}}},invarrow/.default=0.5}
\renewcommand\ss{\scriptstyle}
\newcommand{\doi}[1]{\href{https://doi.org/#1}{\texttt{doi:#1}}}
\newcommand{\arxiv}[1]{\href{https://arxiv.org/abs/#1}{\texttt{arXiv:#1}}}
\theoremstyle{plain}
\newtheorem{thm}{Theorem}[section]
\newtheorem{lemma}[thm]{Lemma}

\newtheorem{prop}[thm]{Proposition}
\newtheorem{cor}[thm]{Corollary}
\theoremstyle{definition}
\newtheorem{dfn}[thm]{Definition}
\newtheorem{ex}[thm]{Example}
\newtheorem{remark}[thm]{Remark}

\numberwithin{equation}{section}
\long\def\junk#1{}
\newcommand{\bra}[1]{\left\langle #1\right|}
\newcommand{\ket}[1]{\left|#1\right\rangle}
\newcommand{\braket}[2]{\left< #1 \vphantom{#2}\right|
\! \left. #2 \vphantom{#1} \right>}%
\newcommand\ebull[3]{\filldraw[fill=white, draw=black] (#1,#2) circle (#3cm);}
\newcommand\bbull[3]{\filldraw[fill=black, draw=black] (#1,#2) circle (#3cm);}
\newcommand\rbull[3]{\filldraw[fill=red,draw=black] (#1,#2) circle (#3cm);}
\newcommand\cbull[3]{\filldraw[fill=cyan,draw=black] (#1,#2) circle (#3cm);}
\newcommand\sbull[3]{\filldraw[fill=cyan,draw=black] (#1:#2) circle (#3cm);}
\tikzset{part1/.style={draw=red,line join=round}}
\tikzset{part2/.style={draw=cyan,line join=round}}
\newcommand{\bm}{\mathbf{m}}
\newcommand{\bll}{\mathbf{l}}
\newcommand{\bk}{\mathbf{k}}
\newcommand{\tl}{\mathtt{L}}
\newcommand{\tp}{\mathtt{P}}
\newcommand{\tm}{\mathtt{M}}
\newcommand{\tn}{\mathtt{N}}
\newcommand{\afl}{r}
\newcommand{\ra}{\mathit{A}}
\newcommand{\rb}{\mathit{B}}
\newcommand{\rc}{\mathit{C}}
\newcommand{\rd}{\mathit{D}}
\newcommand{\rp}{\mathrm{P}}
\newcommand{\rg}{\mathrm{G}}
\newcommand{\rh}{\mathrm{H}}
\newcommand{\rf}{\mathrm{F}}
\newcommand\bl[1]{\textrm{\color{blue}#1}}
\newcommand{\ba}{\mathbf{A}}
\newcommand{\bb}{\mathbf{B}}
\newcommand{\bc}{\mathbf{C}}
\newcommand{\bd}{\mathbf{D}}
\newcommand{\bp}{\mathbf{P}}
\newcommand{\cz}{\mathcal{Z}}
\DeclareMathOperator{\wt}{wt}
\usepackage{wrapfig}
\newcommand\ajrem[1]{{\color{blue}[\textbf{AJ:} #1]}}
\newcommand\mwrem[1]{{\color{blue}[\textbf{MW:} #1]}}
\newcommand\pzjrem[1]{{\color{purple}[\textbf{PZJ:} #1]}}
\title{Structure constants for spin Hall--Littlewood functions}
\author{Ajeeth Gunna, Michael Wheeler and Paul Zinn-Justin}
\address{Ajeeth Gunna, Michael Wheeler, Paul Zinn-Justin, School of Mathematics and Statistics, University of Melbourne, Parkville, Victoria 3010, Australia.}
\email{ajeeth.gunna@unimelb.edu.au,wheelerm@unimelb.edu.au,pzinn@unimelb.edu.au}
\begin{document}
\begin{abstract}
We provide a combinatorial formula for the structure constants of spin Hall--Littlewood functions. This is achieved by representing these functions and the structure constants as the partition function of a lattice model and applying the underlying Yang--Baxter equation. Our combinatorial expression is in terms of generalised honeycombs; the latter were introduced by Knutson and Tao for ordinary Littlewood--Richardson coefficients and applied to the computation of Hall polynomials by Zinn–Justin.
\end{abstract}

\maketitle


\section{Introduction}

\emph{Hall--Littlewood polynomials} are a one-parameter generalisation of the well-known \emph{Schur polynomials}. The structure constants of these polynomials, commonly called \emph{Hall polynomials}, exhibit intriguing properties: they are polynomials and they enumerate short exact sequences of finite abelian $p$-groups. A combinatorial formula for these coefficients was presented in~\cite{macdonald1998symmetric} using classical Littlewood--Richardson tableaux. More recently, an alternative formula for Hall polynomials in terms of honeycombs was introduced by Zinn-Justin \cite{Puzzles:ZinnJustin2019HoneycombsFH}.

In \cite{lattice:spinborodin2014}, Borodin extended the classical Hall--Littlewood polynomials by incorporating an additional parameter ($s$): the resulting functions are commonly called \emph{spin Hall--Littlewood functions} ($\mathrm{F}_{\bm}$). In the same work, Borodin explored various identities associated with these functions, including Cauchy identities, Pieri rules, and the branching formula, orthogonality, among others; also see ~\cite{spin:Corwin2015StochasticHS,spin:Borodin_Corwin_Petrov_Sasamoto_2015,spin:BorodinPetrov2016HigherSS}.

\begin{equation}\label{eq:intro_spin_product}
\mathrm{F}_{\bm}(x_1,\dots,x_n;s) \hspace{1mm}\mathrm{F}_{\mathbf{l}}(x_1,\dots,x_{n};s)=\sum_{\bk}
\mathcal{C}^{\bk}_{\bll,\bm}(q,s)\hspace{1mm}
\mathrm{F}_{\mathbf{k}}(x_1,\dots,x_{n};s)  
\end{equation}

In this paper, we derive a puzzle formulation for $\mathcal{C}^{\bk}_{\bll,\bm}$. Through this formulation, we observe that although $\mathrm{F}_{\bm}$ are rational functions, expansion of their product in themselves remains finite.

\subsection{Lattice models and Littlewood--Richardson coefficients}

Over the past two decades, exactly solvable lattice models have been extensively used to study various families of symmetric functions. In this framework, symmetric functions are represented as partition functions of lattice models, and the \emph{Yang--Baxter equation} (YBE) is leveraged to derive identities such as \emph{Cauchy identities}, \emph{Littlewood identities}, \emph{Pieri rules}, and \emph{branching formulas}. The following families of symmetric polynomials have been explored using lattice models:
\begin{itemize}
     \item \emph{Schur polynomials}: \cite{puzzles:2008paul,schur:Brubaker2009SchurPA,schur:Aggarwal2021FreeFS,schur:Naprienko_2024}.
    \item \emph{Grothendieck polynomials}: \cite{Ms-bosfer-ktheory,MS13,Puzzles:WZJ16Grothendieck,Buciumas2020DoubleGP,lattice:Gunna2020VertexMF,brubakergroth}.
    \item \emph{Hall--Littlewood polynomials}: \cite{hall:Tsilevich2005QuantumIS,Puzzles:WZJHallPI,Puzzles:ZinnJustin2019HoneycombsFH}.
    \item \emph{Schubert polynomials}: \cite{puzzles:Iva_selfdual,PAintegrability1,PAintegrability2,PAintegrability3}.
    \item \emph{Macdonald polynomials}: \cite{macd:CantiniGW,macd:GarbaliGW,macd:GarbaliW,macd:Borodin2019NonsymmetricMP,macd:Garbali2024ShuffleAL}.
    \item \emph{LLT polynomials}: \cite{llt:Corteel2020AVM,llt:ABW}.
    \item \emph{{\textit{q}-Whittaker polynomials}}: \cite{whit:Mucciconi2020SpinQP,whit:Korotkikh2022RepresentationTI,whit:Borodin2021InhomogeneousS}.
\end{itemize}

In~\cite{puzzles:2008paul}, Zinn-Justin introduced a method for computing Littlewood--Richardson coefficients of Schur polynomials using exactly solvable lattice models. Building on this, Wheeler and Zinn-Justin applied the technique to compute the structure constants of Grothendieck polynomials~\cite{Puzzles:WZJ16Grothendieck}. The essence of this approach lies in constructing a lattice and evaluating its partition function in two different ways using YBE. In one evaluation, the model produces the product of the target functions. In the other, the same partition function is interpreted as a sum over certain combinatorial object multiplied by the target function.

In this paper, a similar technique is used; however, the argument presented here is much simpler than in previous cases and is reminiscent of the now-standard approach used to prove \emph{Cauchy identities} using YBE. Additionally, the lattice model used in previous cases allows at most one particle per site, and it is unclear if that model can be extended to cases where each site can hold multiple particles. We outline the structure of our technique.

\begin{equation}
\label{eq:intro_yangbaxterlatticemodel}
\def\k{3}\def\n{5}\def\spc{0.5}\pgfmathtruncatemacro\kk{\n-\k}
 \begin{tikzpicture}[scale=0.4,baseline=(current bounding box.center)]
\draw (-\n*0.25-\k*0.5,-\n*1.299) coordinate (A) -- ++(120:\n) coordinate (B)-- ++(60:\n) coordinate (C) -- ++(0:\n) coordinate (D) -- ++(-60:\n) coordinate (E) -- ++(-120:\n) coordinate (F) -- cycle;
\draw (C) -- ++(-60:\n) coordinate (G) -- (A); \draw (G)--(E);
\path (B)++(180:0.5)--++(60:2.5) node {$\ss \bm $};
\path (C)++(90:0.5)--++(0:2.5) node {$\ss \bll $};
\draw[fill=lightgray] (A)--++(60:5)--++(0:5)--++(-120:5)--++(180:5);
\path (B) ++(0:2.5) node {\bl A};
\path (F) ++(120:2.5) node {\bl B};
\path (A) ++(60:7.5) node {\bl C};
\end{tikzpicture}=\sum_{\bk}\begin{tikzpicture}[scale=0.4,baseline=(current bounding box.center)]
\draw (-\n*0.25-\k*0.5,-\n*1.299) coordinate (A) -- ++(120:\n) coordinate (B)-- ++(60:\n) coordinate (C) -- ++(0:\n) coordinate (D) -- ++(-60:\n) coordinate (E) -- ++(-120:\n) coordinate (F) -- cycle;
\draw (B) -- ++(0:\n) coordinate (G) -- (D); \draw (F) -- (G);
\path (D)++(180:0.5)--++(-120:2.5) node {$\ss \bk $};
\path (B)++(180:0.5)--++(60:2.5) node {$\ss \bm $};
\path (C)++(90:0.5)--++(0:2.5) node {$\ss \bll $};
\draw[fill=lightgray] (A)--++(0:5)--++(120:5)--++(180:5)--++(-60:5);
\path (A) ++(60:2.5) node {\bl D};
\path (G) ++(120:2.5) node {\bl E};
\path (G) ++(0:2.5) node {\bl F};
\end{tikzpicture}
\end{equation}

In equation~\eqref{eq:intro_yangbaxterlatticemodel}, the partition function of regions $A$ and $C$ are spin Hall--Littlewood functions while region $B$ gives a trivial factor. On the RHS of ~\eqref{eq:intro_yangbaxterlatticemodel},   region $F$ produces the spin Hall--Littlewood function, and region D contributes a trivial factor while region $E$ gives the desired puzzles.

 \subsection{Layout of the paper}
 In ~\cref{sec:preliminaries} we recall necessary prerequisites and establish our notation. We then recall the most general solution of the YBE in the case of models of type $\mathcal{U}_{q}(\widehat{\mathfrak{sl}}_{\afl+1})$. We then specialise those weights to introduce the six vertex model and its wave functions, higher spin vertex models, and spin Hall--Littlewood functions.

 In \cref{sec:maintheorems} we introduce \emph{six vertex puzzles} and \emph{higher spin puzzles} and state our main theorems. Sections ~\ref{sec:proof_6v} and ~\ref{sec:proof_spinhall} are dedicated for the proofs of our main theorems.

\section{Preliminaries}
\label{sec:preliminaries}
\subsection{Lattice model}
We consider a lattice model consisting of horizontal lines oriented from left to right and vertical lines oriented from bottom to top, with their intersections forming vertices. Variables are attached to both vertical and horizontal lines. We decorate the edges with labels, typically a non-negative integer or a tuple of non-negative integers, depending on the context.
\begin{equation*}
\begin{tikzpicture}[scale=0.8,baseline=(current bounding box.center)]
\draw[ultra thick, arrow=1] (0,0.5)--(1,0.5);
\draw[ultra thick, arrow=1] (0.5,0)--(0.5,1);
\node at (-0.5,0.5) {$\bb$};
\node at (1.5,0.5) {$\bd$};
\node at (0.5,-0.5) {$\ba$};
\node at (0.5,1.5) {$\bc$};
\draw[->] (0.5,-1.5) node[below] {$y$}--(0.5,-1);
\draw[->] (-1.5,0.5) node[left] {$x$}--(-1,0.5);
\end{tikzpicture}=W\left(\dfrac{x}{y};q;\ba,\bb,\bc,\bd\right)
\end{equation*}
We refer to the labelling of the edges of all vertices in a lattice as a configuration. 
We attach a weight to each vertex depending on its respective row and column. These weights are non-zero only when the \emph{conservation} is satisfied i.e., the sum of the labels entering from the bottom and left is equal to the sum of the labels on the top and right. A configuration is valid only when conservation is satisfied at each vertex.  Below is an example of a valid configuration where edges are labelled with $0$ and $1$.
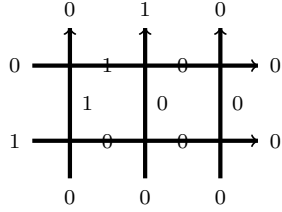
\[
\begin{tikzpicture}[scale=1]
\draw[ultra thick,arrow=1] (0.5,0) node[left,black]{$\ss 1$}--(1.5,0) node[black]{$\ss 0$} --(2.5,0) node[black]{$\ss 0$}--(3.5,0) node[right,black]{$\ss 0$};
\draw[ultra thick,arrow=1] (0.5,1) node[left,black]{$\ss 0$}--(1.5,1) node[black]{$\ss 1$}--(2.5,1)node[black]{$\ss 0$}--(3.5,1)node[right,black]{$\ss 0$};
\draw[ultra thick,arrow=1]  (1,-0.5) node[below] {$\ss 0$}--(1,0.5) node [right] {$\ss 1$}--(1,1.5) node[above] {$\ss 0$};
\draw[ultra thick,arrow=1]  (2,-0.5)node[below] {$\ss 0$}--(2,0.5) node[right]{$\ss 0$}--(2,1.5) node[above] {$\ss 1$};
\draw[ultra thick,arrow=1]  (3,-0.5)node[below] {$\ss 0$}--(3,0.5) node[right]{$\ss 0$}--(3,1.5)node[above] {$\ss 0$};
\end{tikzpicture}
\]
For each configuration, the total weight is determined by multiplying the weights of the individual vertices. The partition function of a lattice model with a fixed boundary is defined as the sum of the weights of all possible configurations.

Drawing a vertex as the intersection of two lines is convenient when the edges are labelled with non-negative integers. However, this representation becomes cumbersome when the edges are decorated by tuples. Hence, we often depict a vertex as a tile, with particles entering from the bottom and left, and exiting through the top and right. Below is an example of a vertex with edges decorated by ordered pairs, and its corresponding tile representation.
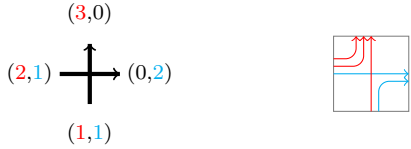
\begin{center}
\begin{tabular}{c@{\hspace{2cm}}c}
\begin{tikzpicture}[scale=0.8,baseline=(current bounding box.center)]
\draw[ultra thick, arrow=1] (0,0.5)--(1,0.5);
\draw[ultra thick, arrow=1] (0.5,0)--(0.5,1);
\node at (-0.5,0.5) {$\ss (\textcolor{red}{2},\textcolor{cyan}{1})$};
\node at (1.5,0.5) {$\ss (0,\textcolor{cyan}{2})$};
\node at (0.5,-0.5) {$\ss (\textcolor{red}{1},\textcolor{cyan}{1})$};
\node at (0.5,1.5) {$\ss (\textcolor{red}{3},0)$};
\end{tikzpicture}
&
\begin{tikzpicture}[scale=1,baseline=(current bounding box.center)]
 \draw[gray] (0,0) rectangle (1,1); 
\draw[red,->,rounded corners] (0,0.7)--(0.3,0.7)--(0.3,1);
\draw[red,->,rounded corners] (0,0.6)--(0.4,0.6)--(0.4,1);
\draw[cyan,->,rounded corners] (0,0.5)--(1,0.5);
\draw[red,->,rounded corners] (0.5,0)--(0.5,1);
\draw[cyan,->,rounded corners] (0.6,0)--(0.6,0.4)--(1,0.4);
\end{tikzpicture}\\
\end{tabular}
\end{center}

\subsection{Yang--Baxter equation}
\label{subsec:Yangbaxterequation}
We recall the Bosnjak--Mangazeev solution of the Yang--Baxter equation~\cite{bosnjak2016construction,lattice:borodin2018coloured}. Here, $\afl$ denotes the rank of the quantized affine algebra $\mathcal{U}_{q}(\widehat{\mathfrak{sl}}_{\afl+1})$, and $\tl, \tm, \tn$ are positive numbers that denote the capacity \footnote{This refers to the total number of particles that a line can occupy.} of a line. Variables $x,y,z$ are attached to lines. For an $\afl$-tuple of non-negative integers $\ba$, we denote $|\ba|$ as the sum of all the coordinates of $\ba$. 

Let $\ba,\bb,\bc,\bd$ be $\afl$-tuples of non-negative integers, and associate a weight to vertices as follows:

\begin{equation*}
\begin{tikzpicture}[scale=0.8,baseline=(current bounding box.center)]
\draw[gray] (0,0) rectangle (1,1);
\node at (-0.5,0.5) {$\bb$};
\node at (1.5,0.5) {$\bd$};
\node at (0.5,-0.5) {$\ba$};
\node at (0.5,1.5) {$\bc$};
\draw[->] (0.5,-1.5) node[below] {$(y,\tm)$}--(0.5,-1);
\draw[->] (-1.5,0.5) node[left] {$(x,\tl)$}--(-1,0.5);
\end{tikzpicture} =  W_{\tl,\tm}\left(\dfrac{x}{y};q;\ba,\bb,\bc,\bd\right)
\end{equation*}
For vertices where $|\ba|,|\bc|>\tm$ or $|\bb|,|\bd|>\tl$, the weights are identically zero. These weights are said to satisfy the Yang--Baxter equation when

\begin{multline}
\label{eq:yangbaxter}
    \sum_{\bc_1, \bc_2 ,\bc_3} W_{\tl,\tm}\left( \dfrac{x}{y};q;\ba_2,\ba_1,\bc_2,\bc_1\right)
    W_{\tl,\tn}\left( \dfrac{x}{z};q;\ba_3,\bc_1,\bc_3,\bb_1\right)
    W_{\tm,\tn}\left( \dfrac{y}{z};q;\bc_3,\bc_2,\bb_3,\bb_2\right)\\
    =\sum_{\bc_1, \bc_2 ,\bc_3} W_{\tm,\tn}\left( \dfrac{y}{z};q;\ba_3,\ba_2,\bc_3,\bc_2\right) W_{\tl,\tn}\left( \dfrac{x}{z};q;\bc_3,\ba_1,\bb_3,\bc_1\right)
    W_{\tl,\tm}\left( \dfrac{x}{y};q;\bc_2,\bc_1,\bb_2,\bb_1\right)
\end{multline}
 where $\ba_1,\ba_2,\ba_3,\bb_1,\bb_2,\bb_3$ are fixed $\afl$-tuples of non-negative integers, and the summation on both sides of the equation is over triples of $\afl$-tuples $\bc_1,\bc_2,\bc_3$ of non-negative integers. Graphically \ref{eq:yangbaxter} is represented as below:
\

\begin{align}
\label{graph-master}
\sum_{\bc_1,\bc_2,\bc_3}
\begin{tikzpicture}[scale=1.5,baseline=(current bounding box.center)]
\draw[gray] (0:0)--(60:1)--++(-60:1)--++(-120:1)--++(120:1);
\draw[gray] (0:0)++(60:1)++(-60:1)--++(0:1)--++(120:1)--++(180:1);
\draw[gray] (0:0)++(60:1)++(-60:1)++(0:1)--++(-120:1)--++(180:1);
\node[left] at (-0.2,0.5) {$(x,\tl) \rightarrow$};
\node[left] at (-0.2,-0.5) {$(y,\tm) \rightarrow$};
\node[below] at (1,-1.4) {$\uparrow$};
\node[below] at (1,-1.9) {$(z,\tn)$};
\path (0:0)++(60:0.5) node {$\ss \ba_{1}$};
\path (0:0)++(-60:0.5) node {$\ss \ba_{2}$};
\path (0:0)++(-60:1)++(0:0.5) node {$\ss \ba_{3}$};
\path (0:0)++(-60:1)++(0:1)++(60:0.5) node {$\ss \bb_{1}$};
\path (0:0)++(-60:1)++(0:1)++(60:1)++(120:0.5) node {$\ss \bb_{2}$};
\path (0:0)++(-60:1)++(0:1)++(60:1)++(120:1)++(180:0.5) node {$\ss \bb_{3}$};
\path (0:0)++(-60:1)++(60:0.5) node {$\ss \bc_{1}$};
\path (0:0)++(-60:1)++(60:1)++(120:0.5) node {$\ss \bc_{2}$};
\path (0:0)++(-60:1)++(60:1)++(0:0.5) node {$\ss \bc_{3}$};
\end{tikzpicture}
=
\quad
\sum_{\bc_1,\bc_2,\bc_3}
\begin{tikzpicture}[scale=1.5,baseline=(current bounding box.center)]
\draw[gray] (0:0)--++(60:1)--++(0:1)--++(-60:1)--++(-120:1)--++(180:1)--++(120:1);
\draw[gray] (0:0)--++(0:1)--++(60:1);
\draw[gray] (0:0)++(0:1)--++(-60:1);
\node[left] at (-0.2,0.5) {$(x,\tl) \rightarrow$};
\node[left] at (-0.2,-0.5) {$(y,\tm) \rightarrow$};
\node[below] at (1,-1.4) {$\uparrow$};
\node[below] at (1,-1.9) {$(z,\tn)$};
\path (0:0)++(60:0.5) node {$\ss \ba_{1}$};
\path (0:0)++(-60:0.5) node {$\ss \ba_{2}$};
\path (0:0)++(-60:1)++(0:0.5) node {$\ss \ba_{3}$};
\path (0:0)++(-60:1)++(0:1)++(60:0.5) node {$\ss \bb_{1}$};
\path (0:0)++(-60:1)++(0:1)++(60:1)++(120:0.5) node {$\ss \bb_{2}$};
\path (0:0)++(-60:1)++(0:1)++(60:1)++(120:1)++(180:0.5) node {$\ss \bb_{3}$};
\path (0:0)++(0:0.5) node {$\ss \bc_{3}$};
\path (0:0)++(0:1)++(-60:0.5) node {$\ss \bc_{2}$};
\path (0:0)++(0:1)++(60:0.5) node {$\ss \bc_{1}$};
\end{tikzpicture}
\end{align}

\

We recall the standard definitions of the \emph{$q$-Pochhammer symbol} and the \emph{$q$-binomial coefficient}:  

The $q$-Pochhammer symbol is defined as  
\[
(a;q)_{k} := (1 - a)(1 - aq)\cdots(1 - aq^{k-1}), \quad \text{with} \quad (a;q)_{0} = 1. 
\]  

The $q$-binomial coefficient is given by  
\[
\binom{a}{b}_{q} = \dfrac{(q;q)_{a}}{(q;q)_{a-b}(q;q)_{b}}.  
\]

\

\begin{thm}[~\cite{bosnjak2016construction},\cite{lattice:borodin2018coloured}]
\label{theorem:weightsoftheYBE}
For any two $\afl$-tuples $\lambda$, $\mu$ such that $\lambda_i \leq \mu_i$ for all $1\leq i \leq \afl$, we define the function:
\begin{equation*}
\Phi(\lambda,\mu;x,y):=\dfrac{(x;q)_{|\lambda|} (y/x ; q)_{|\mu-\lambda|}}{(y;q)_{|\mu|}}(y/x)^{|\lambda|}\left( q^{\sum\limits_{i<j} (\mu_i -\lambda_i)\lambda_j}\right)\prod^{n}_{i=1}{\binom{\mu_i}{\lambda_i}}_{q}
\end{equation*}
Then the weights defined as 

\begin{multline}
\label{generalweights}
W_{\tl,\tm}(x;q;\ba,\bb,\bc,\bd)=\mathbf{1}_{\ba+\bb=\bc+\bd} 
\hspace{5mm}x^{|\bd-\bb|} q^{|\ba \tl-\bd \tm|}\\[0.5 em]
\times \sum_{\mathbf{P}}\Phi(\bc-\mathbf{P},\bc+\bd-\mathbf{P};q^{\tl-\tm}x,q^{-\tm}x)\Phi(\mathbf{P},\bb;q^{-\tl}/x,q^{-\tl})
\end{multline} where the summation is over $\afl$-tuples $\mathbf{P}=({P}_1,\dots,{P}_{\afl})$ such that $0\leq {P}_{i}\leq \min(B_i,C_i)$,
satisfy the Yang--Baxter equation ~\eqref{eq:yangbaxter}.
\end{thm}

\subsection{Six vertex model}
In this subsection we consider the specialisation $\afl=1,\tl=1,\tm=1$ of the general weights~\eqref{generalweights}. In this setting, the edge states of a vertex are non-negative integers constrained to be at most $1$. Therefore, each edge label can be either $0$ or $1$. Graphically, we use (\begin{tikzpicture}
    \bbull{0}{0}{0.1};
\end{tikzpicture}) to indicate an edge label $1$ and (\begin{tikzpicture}
    \ebull{0}{0}{0.1};
\end{tikzpicture}) to indicate an edge label $0$. There are six vertices where the conservation is satisfied (we simplify the notation here by omitting the spin labels):

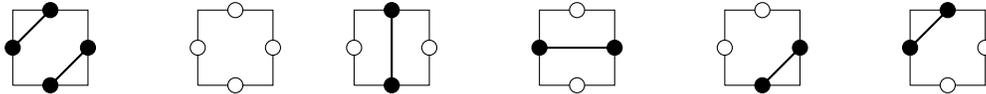
\begin{equation*}
\begin{tabular}{cccccc}
\begin{tikzpicture}[baseline={(0,0)}]
\draw (0:0)--++(0:1)--++(90:1)--++(180:1)--++(-90:1); 
\draw[fill=black] (0,0)++(90:0.5) circle (0.1);
\draw[fill=black] (0,0)++(0:0.5) circle (0.1);
\draw[fill=black] (0,0)++(0:1)++(90:0.5) circle (0.1);
\draw[fill=black] (0,0)++(0:0.5)++(90:1) circle (0.1);
\draw[black,thick] (0,0.5)--(0.5,1);
\draw[black,thick] (0.5,0)--(1,0.5);
\end{tikzpicture}
\qquad&\qquad
\begin{tikzpicture}[baseline={(0,0)}]
\draw (0:0)--++(0:1)--++(90:1)--++(180:1)--++(-90:1); 
\draw[fill=white] (0,0)++(90:0.5) circle (0.1);
\draw[fill=white] (0,0)++(0:0.5) circle (0.1);
\draw[fill=white] (0,0)++(0:1)++(90:0.5) circle (0.1);
\draw[fill=white] (0,0)++(0:0.5)++(90:1) circle (0.1);
\end{tikzpicture}
\qquad&\quad
\begin{tikzpicture}[baseline={(0,0)}]
\draw (0:0)--++(0:1)--++(90:1)--++(180:1)--++(-90:1); 
\draw[fill=white] (0,0)++(90:0.5) circle (0.1);
\draw[fill=black] (0,0)++(0:0.5) circle (0.1);
\draw[fill=white] (0,0)++(0:1)++(90:0.5) circle (0.1);
\draw[fill=black] (0,0)++(0:0.5)++(90:1) circle (0.1);
\draw[black,thick] (0:0)++(0:0.5)--++(90:1);
\end{tikzpicture}
\qquad&\qquad
\begin{tikzpicture}[baseline={(0,0)}]
\draw (0:0)--++(0:1)--++(90:1)--++(180:1)--++(-90:1); 
\draw[fill=black] (0,0)++(90:0.5) circle (0.1);
\draw[fill=white] (0,0)++(0:0.5) circle (0.1);
\draw[fill=black] (0,0)++(0:1)++(90:0.5) circle (0.1);
\draw[fill=white] (0,0)++(0:0.5)++(90:1) circle (0.1);
\draw[black,thick] (0:0)++(90:0.5)--++(0:1);
\end{tikzpicture}
\qquad & \qquad
\begin{tikzpicture}[baseline={(0,0)}]
\draw (0:0)--++(0:1)--++(90:1)--++(180:1)--++(-90:1); 
\draw[fill=white] (0,0)++(90:0.5) circle (0.1);
\draw[fill=black] (0,0)++(0:0.5) circle (0.1);
\draw[fill=black] (0,0)++(0:1)++(90:0.5) circle (0.1);
\draw[fill=white] (0,0)++(0:0.5)++(90:1) circle (0.1);
\draw[black,thick] (0.5,0)--(1,0.5);
\end{tikzpicture}
\qquad & \qquad
\begin{tikzpicture}[baseline={(0,0)}]
\draw (0:0)--++(0:1)--++(90:1)--++(180:1)--++(-90:1); 
\draw[fill=black] (0,0)++(90:0.5) circle (0.1);
\draw[fill=white] (0,0)++(0:0.5) circle (0.1);
\draw[fill=white] (0,0)++(0:1)++(90:0.5) circle (0.1);
\draw[fill=black] (0,0)++(0:0.5)++(90:1) circle (0.1);
\draw[black,thick] (0,0.5)--(0.5,1);
\end{tikzpicture}\\
\end{tabular}
\end{equation*}

\begin{dfn}
\label{def:6v-wavefunction}
For two binary strings $m = (m_{1}, m_{2}, \dots)$ and $w = (w_{1}, w_{2}, \dots)$, where $m_{i}, w_{i} \in \{0,1\}$ and $m$ contains $n$ more $1$'s than $w$, we define the rational function $\rh_{m/w}(x_1, \dots, x_n;y_{1},\dots)$ as follows:
\begin{equation}
\label{eq:definitionofH}
\rh_{m/w}(x_1,\dots,x_n;y_{1},\dots):=
\begin{tikzpicture}[scale=0.6,baseline=(current bounding box.center)]
\draw (0:0)--++(0:4)--++(90:5)--++(180:4)--++(-90:5);
\foreach \x in {1,2,3,4}{
\draw[gray] (0:\x)--++(90:5);
\draw[gray] (90:\x)--++(0:4);
};
\foreach \x in {1,2,3,4}{
\draw[fill=black] (0:0)++(90:5)++(0:\x-0.5) circle (0.1);
\draw[fill=white] (0:0)++(0:\x-0.5) circle (0.1);
};
\path (180:0.5)++(90:4.5) node {$\ss m_1$};
\path (180:0.5)++(90:3.5) node {$\ss m_2$};
\path (180:0.5)++(90:2) node[rotate=90] {$\ss \dots$};
\path (0:4.5)++(90:4.5) node {$\ss w_1$};
\path (0:4.5)++(90:3.5) node {$\ss w_2$};
\path (0:4.5)++(90:2) node[rotate=90] {$\ss \dots$};
\draw[->] (0.5,-1) node[below] {$\ss x_1 $}--(0.5,-0.5);
\node at (2,-1) {$\ldots$} ;
\draw[->] (3.5,-1) node[below] {$\ss x_n $}--(3.5,-0.5);
 \path (0:0)++(180:3)++(90:2) node[rotate=90] {$\dots$};
 \path (0:0)++(180:3)++(90:3.5) node {$\ss y_{2}$};
 \path (0:0)++(180:3)++(90:4.5) node {$\ss y_{1}$};
  \foreach \x in {1,2,3,4,5}{
 \draw[->] (0:0)++(180:2.5)++(90:\x-0.5)--++(0:0.5);
 };
\end{tikzpicture}
\end{equation}
where the above lattice is to be interpreted as a partition function, with vertex weights:
\

\begin{equation}
\label{weights:defining-H}
\begin{tabular}{cccc}
\begin{tikzpicture}[baseline={(0,0)},rotate=0]
\draw (0:0)--++(0:1)--++(90:1)--++(180:1)--++(-90:1); 
\draw[->] (0:0)++(-90:0.3)++(0:0.5) node[below]{$x$}--++(90:0.6);
\draw[->] (0:0)++(90:0.5)++(180:0.3) node[left] {$y$}--++(0:0.6);
\end{tikzpicture}
&\qquad
\begin{tikzpicture}[baseline={(0,0)}]
\draw (0:0)--++(0:1)--++(90:1)--++(180:1)--++(-90:1); 
\draw[fill=black] (0,0)++(90:0.5) circle (0.1);
\draw[fill=black] (0,0)++(0:0.5) circle (0.1);
\draw[fill=black] (0,0)++(0:1)++(90:0.5) circle (0.1);
\draw[fill=black] (0,0)++(0:0.5)++(90:1) circle (0.1);
\path (0:0)++(0:0.5)++(-90:1) node {$1$};
\draw[black,thick] (0,0.5)--(0.5,1);
\draw[black,thick] (0.5,0)--(1,0.5);
\end{tikzpicture}
&\qquad
\begin{tikzpicture}[baseline={(0,0)}]
\draw (0:0)--++(0:1)--++(90:1)--++(180:1)--++(-90:1); 
\draw[fill=white] (0,0)++(90:0.5) circle (0.1);
\draw[fill=white] (0,0)++(0:0.5) circle (0.1);
\draw[fill=white] (0,0)++(0:1)++(90:0.5) circle (0.1);
\draw[fill=white] (0,0)++(0:0.5)++(90:1) circle (0.1);
\path (0:0)++(0:0.5)++(-90:1) node {$1$};
\end{tikzpicture}
&\quad
\begin{tikzpicture}[baseline={(0,0)}]
\draw (0:0)--++(0:1)--++(90:1)--++(180:1)--++(-90:1); 
\draw[fill=white] (0,0)++(90:0.5) circle (0.1);
\draw[fill=black] (0,0)++(0:0.5) circle (0.1);
\draw[fill=white] (0,0)++(0:1)++(90:0.5) circle (0.1);
\draw[fill=black] (0,0)++(0:0.5)++(90:1) circle (0.1);
\path (0:0)++(0:0.5)++(-90:1) node {$\ss \dfrac{1-xy^{-1}q^{-1}}{1-q x y^{-1}}$};
\draw[black,thick] (0:0)++(0:0.5)--++(90:1);
\end{tikzpicture}\\[6em]
&\quad
\begin{tikzpicture}[baseline={(0,0)}]
\draw (0:0)--++(0:1)--++(90:1)--++(180:1)--++(-90:1); 
\draw[fill=black] (0,0)++(90:0.5) circle (0.1);
\draw[fill=white] (0,0)++(0:0.5) circle (0.1);
\draw[fill=black] (0,0)++(0:1)++(90:0.5) circle (0.1);
\draw[fill=white] (0,0)++(0:0.5)++(90:1) circle (0.1);
\path (0:0)++(0:0.5)++(-90:1) node {$\ss {\dfrac{q\left(1-xy^{-1}q^{-1}\right)}{1- xy^{-1}}}$};
\draw[black,thick] (0:0)++(90:0.5)--++(0:1);
\end{tikzpicture}
&\quad
\begin{tikzpicture}[baseline={(0,0)}]
\draw (0:0)--++(0:1)--++(90:1)--++(180:1)--++(-90:1); 
\draw[fill=white] (0,0)++(90:0.5) circle (0.1);
\draw[fill=black] (0,0)++(0:0.5) circle (0.1);
\draw[fill=black] (0,0)++(0:1)++(90:0.5) circle (0.1);
\draw[fill=white] (0,0)++(0:0.5)++(90:1) circle (0.1);
\path (0:0)++(0:0.5)++(-90:1) node {$\ss \dfrac{(1-q)xy^{-1}q^{-1}}{1-xy^{-1}}$};
\draw[black,thick] (0.5,0)--(1,0.5);
\end{tikzpicture}
&\quad
\begin{tikzpicture}[baseline={(0,0)}]
\draw (0:0)--++(0:1)--++(90:1)--++(180:1)--++(-90:1); 
\draw[fill=black] (0,0)++(90:0.5) circle (0.1);
\draw[fill=white] (0,0)++(0:0.5) circle (0.1);
\draw[fill=white] (0,0)++(0:1)++(90:0.5) circle (0.1);
\draw[fill=black] (0,0)++(0:0.5)++(90:1) circle (0.1);
\path (0:0)++(0:0.5)++(-90:1) node {$\ss \dfrac{(1-q)}{1- x y^{-1}}$};
\draw[black,thick] (0,0.5)--(0.5,1);
\end{tikzpicture}\\
\end{tabular}
\end{equation} 

\

When the word $w$ consists entirely of holes, we omit the index $w$ from the notation and denote the function simply by $\rh_{m}(x_1, \dots, x_n; y_1, \dots)$. Moreover, we define
\[
\rh_{m/w}(x_1, \dots, x_n) := \rh_{m/w}(x_1, \dots, x_n; y_1, \dots)\big|_{y_i = q^{-1}}.
\]
\end{dfn}

\begin{remark}
The weights ~(\ref{weights:defining-H}) are obtained from the general weights~(\ref{generalweights}) via the following transformations:
\begin{equation*}
\begin{tikzpicture}[scale=0.8,baseline=(current bounding box.center)]
\draw[gray] (0,0) rectangle (1,1);
\node at (-0.5,0.5) {$b$};
\node at (1.5,0.5) {$d$};
\node at (0.5,-0.5) {$a$};
\node at (0.5,1.5) {$c$};
\draw[->] (0.5,-1.5) node[below] {$x$}--(0.5,-1);
\draw[->] (-1.5,0.5) node[left] {$y$}--(-1,0.5);
\end{tikzpicture}=W_{1,1}\left(\dfrac{q y}{x};q;1-a,1-b,1-c,1-d\right)
\end{equation*}
\end{remark}

\begin{prop}\cite{spin:BorodinPetrov2016HigherSS}  
The functions $\rh_{m/w}(x_1, \dots, x_n)$ are symmetric in the variables $x_1, \dots, x_n $.  
\end{prop}

\begin{remark}
At $q=0$, $H_{m}(x_{1},\dots,x_{n})$ are symmetric Grothendieck polynomials which arise in the study of the $K$-theory of Grassmannian. In the recent literature, many papers have explored the application of lattice models to the study of these polynomials ~\cite{MS13,Ms-bosfer-ktheory,PAintegrability1,Buciumas2020DoubleGP,brubakergroth,lattice:Gunna2020VertexMF}. In~\cite{Puzzles:WZJ16Grothendieck}, the second and third authors derived puzzles for the structure constants of these polynomials by using solvable lattice models.
\end{remark}
\subsection{Higher spin vertex model}

In this subsection we consider $\tl$ to be generic and setting $\afl=1,\tm=1$ specialisation of the general weights~\eqref{generalweights}. In this setting the edge states of a vertex are tuples with a single component. The values on the vertical edges are constrained to be no greater than $\tl$, while the values on the horizontal edges are constrained to be no greater than $1$. Hence, we label the edges of a vertex with integers $0\leq a,c \leq 1$ and $0\leq b,d \leq \tl$.
\begin{equation*}
\begin{tikzpicture}[scale=0.8,baseline=(current bounding box.center)]
\draw[gray] (0,0) rectangle (1,1);
\node at (-0.5,0.5) {$b$};
\node at (1.5,0.5) {$d$};
\node at (0.5,-0.5) {$a$};
\node at (0.5,1.5) {$c$};
\draw[->] (0.5,-1.5) node[below] {$(y,1)$}--(0.5,-1);
\draw[->] (-1.5,0.5) node[left] {$(x,\tl)$}--(-1,0.5);
\end{tikzpicture}=W_{\tl,1}\left(\dfrac{x}{y};q;a,b,c,d\right)
\end{equation*} 
There are four types of vertices where the conservation is satisfied (notation is simplified by omitting the spin labels):
\begin{equation}
\label{weights:def-spinHl}
\begin{tabular}{ccccc}
\begin{tikzpicture}[scale=1,baseline=-2pt]
 \draw[gray] (0,0) rectangle (1,1); 
\node at (0.5,-0.3) {$\ss A$};
\node at (-0.3,0.5) {$\ss B$};
\node at (0.5,1.3) {$\ss C$};
\node at (1.3,0.5) {$\ss D$};
\draw[->] (0.5,-1) node[below] {$\ss x$}--(0.5,-0.7);
\draw[->] (-1,0.5) node[left] {$\ss s$}--(-0.7,0.5);
\end{tikzpicture}
&
\begin{tikzpicture}[scale=1,baseline=-2pt]
\draw[gray] (0,0) rectangle (1,1); 
\draw[->,rounded corners] (0,0.5)--(1,0.5);
\draw[->,rounded corners] (0,0.6)--(0.5,0.6)--(0.5,1);
\draw[->,rounded corners] (0,0.4)--(1,0.4);
\draw[->,rounded corners] (0.5,0)--(0.5,0.3)--(1,0.3);
\node at (0.5,-0.3) {$\ss 1$};
\node at (-0.3,0.5) {$\ss m$};
\node at (0.5,1.3) {$\ss 1$};
\node at (1.3,0.5) {$\ss m$};
\end{tikzpicture}
&
\begin{tikzpicture}[scale=1,baseline=-2pt]
 \draw[gray] (0,0) rectangle (1,1); 
\draw[->,rounded corners] (0,0.4)--(1,0.4);
\draw[->,rounded corners] (0,0.5)--(1,0.5);
\draw[->,rounded corners] (0,0.6)--(0.5,0.6)--(0.5,1);
\node at (0.5,-0.3) {$\ss 0$};
\node at (-0.3,0.5) {$\ss m$};
\node at (0.5,1.3) {$\ss 1$};
\node at (1.5,0.5) {$\ss m-1$};
\end{tikzpicture}
&
\begin{tikzpicture}[scale=1,baseline=-2pt]
 \draw[gray] (0,0) rectangle (1,1); 
\draw[->] (0,0.7)--(1,0.7);
\draw[->,rounded corners] (0,0.5)--(1,0.5);
\draw[->,rounded corners] (0,0.6)--(1,0.6);
\draw[->,rounded corners] (0.5,0)--(0.5,0.4)--(1,0.4);
\node at (0.5,-0.3) {$\ss 1$};
\node at (-0.3,0.5) {$\ss m$};
\node at (0.5,1.3) {$\ss 0$};
\node at (1.5,0.5) {$\ss m+1$};
\end{tikzpicture}
&
\begin{tikzpicture}[scale=1,baseline=-2pt]
 \draw[gray] (0,0) rectangle (1,1); 
\draw[->] (0,0.4)--(1,0.4);
\draw[->] (0,0.5)--(1,0.5);
\draw[->] (0,0.6)--(1,0.6);
\node at (0.5,-0.3) {$\ss 0$};
\node at (-0.3,0.5) {$\ss m$};
\node at (0.5,1.3) {$\ss 0$};
\node at (1.5,0.5) {$\ss m$};
\end{tikzpicture}\\
& $\ss \dfrac{1-q^{m}x s}{x-s}$
& $\ss \dfrac{1-q^{m}}{x-s}$
& $\ss \dfrac{(1-q^{m}s^{2})x}{x-s}$
& $\ss \dfrac{x-q^{m}s}{x-s} $
\end{tabular}
\end{equation}  
\begin{remark}
The weights $W_{\tl,1}$ given in \eqref{generalweights} do not include the parameter $s$. Here, we briefly outline how to derive these weights from the general ones. They are dependent on $\tl$ via $q^{\tl}$. After the appropriate normalisation, the Yang--Baxter equation transforms into a polynomial in the variable $q^{\tl}$, which holds for infinitely many non-negative values of $\tl$. These weights are obtained by substituting $q^{-\tl}=s^{2}$. 

The exact derivation of these weights from (\ref{generalweights}) is given in (\ref{subsec:spin_regionA}). We also refer the reader to~\cite[Appendix C.2]{lattice:borodin2018coloured} for a detailed explanation of this derivation.
\end{remark}
\begin{dfn}
\label{def:spinHl}
Fix positive integers $P$ and $n$. Let $\mathbf{m}=(m_1 ,\dots,m_{P})$ be a $P$-tuple of non-negative integers such that $\sum^{P}_{i=1}m_{i}=n$. We define the \emph{spin Hall--Littlewood functions} $\mathrm{F}_{\bm}(x_1,\dots x_n;s)$ as the partition function of 
\begin{equation}
\label{lattice:definition_spinHL}
\begin{tikzpicture}[scale=0.6,baseline=(current bounding box.center)]
\draw (0:0) coordinate (A)--++(90:6) coordinate (B)--++(0:4) coordinate (C)-- ++(-90:6) coordinate (D)--cycle;
\path (A)++(0:-0.7)++(90:5.5) node[black] {$\ss m_{1}$};
\path (A)++(0:-0.7)++(90:4.5) node[black] {$\ss m_{2}$};
\path (A)++(0:-0.5)++(90:2) node[rotate=90] {$\ldots$};
\path (A)++(0:-0.5)++(90:3) node[rotate=90] {$\ldots$};
\path (A)++(0:-0.7)++(90:0.5) node[black] {$\ss m_{P}$};
\foreach\x in {1,2,3,4,5,6}{
\draw[fill=white] (0:0)++(0:4)++(90:\x-0.5) circle (0.1);};
\foreach\x in {1,2,3,4,5}{
\draw[lightgray] (0:0)++(90:\x)--++(0:4);};
\foreach\x in {1,2,3}{
\draw[lightgray] (0:0)++(0:\x)--++(90:6);};
\foreach\x in {1,2,3,4}{
\draw[fill=white] (0:0)++(0:\x-0.5) circle (0.1);
\draw[fill=black] (0:0)++(0:\x-0.5)++(90:6) circle (0.1);};
\draw[->] (0.5,-1) node[below] {$\ss x_1 $}--(0.5,-0.5);
\node at (2,-1) {$\ldots$} ;
\draw[->] (3.5,-1) node[below] {$\ss x_n $}--(3.5,-0.5);
\end{tikzpicture}
\end{equation}
where the above lattice is to be interpreted as a partition function, with vertex weights given in (\ref{weights:def-spinHl}).
\end{dfn}

These functions are symmetric in $(x_{1},\dots, x_{n})$. Like many other families of symmetric functions, they satisfy interesting identities such as a branching formula, Cauchy identity, and Pieri rules. They are also invariant under padding $\bm$ with extra zeroes.

By setting $s=0$, we recover the lattice model for Hall--Littlewood polynomials introduced in ~\cite{hall:Tsilevich2005QuantumIS}, with inverted variables. In this paper, we derive a combinatorial formula in terms of {\em higher spin puzzles} for the structure constants of spin Hall--Littlewood functions.

\subsection{Relation to Borodin's \emph{spin Hall Littlewood functions}}
 The functions that we consider in this paper are slight modifications of those considered in ~\cite{lattice:spinborodin2014}. We recall the definition of these functions; we shall denote them by $\mathcal{F}$.

\begin{dfn}
\label{def:borodin_spinHl}
Fix positive integers $P$ and $n$. Let $\textbf{m}=(m_1 ,\dots,m_{P})$ be a $P$-tuple of non-negative integers such that $\sum^{P}_{i=1}m_{i}=n$. We define \emph{Borodin's spin Hall--Littlewood functions} $\mathcal{F}_{\bm}(x_1,\dots x_n;s)$ as the partition function of 

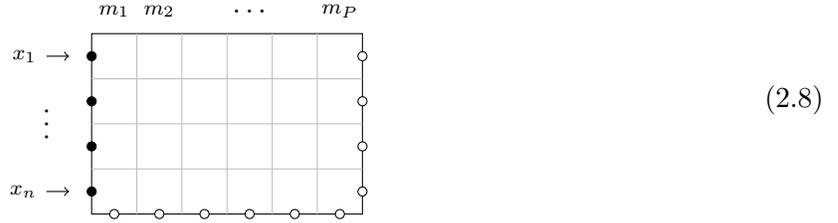
\begin{equation}
\label{lattice:definition-borodin_spinHL}
\begin{tikzpicture}[scale=0.6,baseline=(current bounding box.center)]
\draw (0:0) coordinate (A)--++(90:4) coordinate (B)--++(0:6) coordinate (C)-- ++(-90:4) coordinate (D)--cycle;
\path (A)++(90:4.5)++(0:0.5) node[black] {$\ss m_{1}$};
\path (A)++(90:4.5)++(0:1.5) node[black] {$\ss m_{2}$};
\path (A)++(90:4.5)++(0:3.5) node {$\ldots$};
\path (A)++(90:4.5)++(0:5.5) node[black] {$\ss m_{P}$};
\foreach\x in {1,2,3,4,5,6}{
\draw[fill=white] (0:0)++(0:\x-0.5) circle (0.1);};
\foreach\x in {1,2,3}{
\draw[lightgray] (0:0)++(90:\x)--++(0:6);};
\foreach\x in {1,2,3,4,5}{
\draw[lightgray] (0:0)++(0:\x)--++(90:4);};
\foreach\x in {1,2,3,4}{
\draw[fill=white] (0:0)++(0:6)++(90:\x-0.5) circle (0.1);
\draw[fill=black] (0:0)++(90:\x-0.5) circle (0.1);};
\draw[->] (-1,0.5) node[left] {$\ss x_n $}--(-0.5,0.5) ;
\node[rotate=90] at (-1,2) {$\ldots$} ;
\draw[->] (-1,3.5) node[left] {$\ss x_1 $}--(-0.5,3.5);
\end{tikzpicture}
\end{equation}

\

where the above lattice is to be interpreted as a partition function, with vertex weights given below:
\begin{align}
\label{weights:def-borodin_spinHL}
\begin{tabular}{ccccc}
\begin{tikzpicture}[scale=1,baseline=-2pt]
 \draw[gray] (0,0) rectangle (1,1); 
\node at (0.5,-0.3) {$\ss A$};
\node at (-0.3,0.5) {$\ss B$};
\node at (0.5,1.3) {$\ss C$};
\node at (1.3,0.5) {$\ss D$};
\draw[->] (0.5,-1) node[below] {$\ss s$}--(0.5,-0.7);
\draw[->] (-1,0.5) node[left] {$\ss x$}--(-0.7,0.5);
\end{tikzpicture}
&
\begin{tikzpicture}[scale=1,baseline=-2pt]
\draw[gray] (0,0) rectangle (1,1); 
\draw[->,rounded corners] (0.4,0)--(0.4,1);
\draw[->,rounded corners] (0.5,0)--(0.5,1);
\draw[->,rounded corners] (0.6,0)--(0.6,1);
\node at (-0.3,0.5) {$\ss 0$};
\node at (0.5,-0.3) {$\ss m$};
\node at (1.3,0.5) {$\ss 0$};
\node at (0.5,1.3) {$\ss m$};
\end{tikzpicture}
&
\begin{tikzpicture}[scale=1,baseline=-2pt]
\draw[gray] (0,0) rectangle (1,1); 
\draw[->,rounded corners] (0.3,0)--(0.3,1);
\draw[->,rounded corners] (0.4,0)--(0.4,1);
\draw[->,rounded corners] (0.5,0)--(0.5,1);
\draw[->,rounded corners] (0.6,0)--(0.6,0.5)--(1,0.5);
\node at (-0.3,0.5) {$\ss 0$};
\node at (0.5,-0.3) {$\ss m+1$};
\node at (1.3,0.5) {$\ss 1$};
\node at (0.5,1.3) {$\ss m$};
\end{tikzpicture}
&
\begin{tikzpicture}[scale=1,baseline=-2pt]
\draw[gray] (0,0) rectangle (1,1); 
\draw[->,rounded corners] (0,0.5)--(0.4,0.5)--(0.4,1);
\draw[->,rounded corners] (0.5,0)--(0.5,1);
\draw[->,rounded corners] (0.6,0)--(0.6,1);
\draw[->,rounded corners] (0.7,0)--(0.7,1);
\node at (-0.3,0.5) {$\ss 1$};
\node at (0.5,-0.3) {$\ss m$};
\node at (1.3,0.5) {$\ss 0$};
\node at (0.5,1.3) {$\ss m+1$};
\end{tikzpicture}
&
\begin{tikzpicture}[scale=1,baseline=-2pt]
\draw[gray] (0,0) rectangle (1,1); 
\draw[->,rounded corners] (0,0.5)--(0.4,0.5)--(0.4,1);
\draw[->,rounded corners] (0.5,0)--(0.5,1);
\draw[->,rounded corners] (0.6,0)--(0.6,1);
\draw[->,rounded corners] (0.7,0)--(0.7,0.5)--(1,0.5);
\node at (-0.3,0.5) {$\ss 1$};
\node at (0.5,-0.3) {$\ss m$};
\node at (1.3,0.5) {$\ss 1$};
\node at (0.5,1.3) {$\ss m$};
\end{tikzpicture}\\
& $\ss \dfrac{1-s x q^{m}}{1-sx}$
& $\ss \dfrac{(1-s^{2} q^{m})x}{1-sx}$
& $\ss \dfrac{(1-q^{m+1})}{1-sx}$
& $\ss \dfrac{x-s q^{m}}{1-sx} $
\end{tabular}
\end{align} 
\end{dfn}

Furthermore, they have the following symmetrization formula which we recall here:

\

\begin{thm}~\cite{lattice:spinborodin2014}
For a tuple of non negative integers $\bm=(m_{1},m_{2},\dots)$ with $\sum_{i}m_{i}=n$, we attach a partition $\mu=(\mu_{1}\geq\mu_{2}\cdots\geq \mu_{n}\geq 0)$ such that $m_{i}=|\{j|\mu_{j}=i-1\}|$. We have the following formula:
\begin{equation}
\label{eq:symmetrisationformula_borodinspin}
\mathcal{F}_{\bm} (x_{1},\dots,x_{n};s)=\dfrac{(1-q)^{n}}{\prod^{n}_{i=1}(1-s x_{i})} \sum_{\sigma\in S_{n}}\sigma  \left(\prod_{1\leq i<j\leq n}
\dfrac{x_{i}-q x_{j}}{x_{i}-x_{j}} \cdot \prod^{n}_{i=1} \left(\dfrac{x_{i}-s}{1-s x_{i}}\right)^{\mu_{i}}\right)
\end{equation}
where $S_{n}$ is the symmetric group on $n$ letters.
\end{thm}

\

We now proceed to derive a relationship between $\mathcal{F}_{\mathbf{m}}$ and $\mathrm{F}_{\mathbf{m}}$ and thereby deriving a symmetrization formula for $\mathrm{F}_{\mathbf{m}}$.

\begin{prop}
\label{prop:relationbetween the two f's}
For a tuple of non negative integers $\bm=(m_{1},m_{2},\dots)$ with $\sum_{i}m_{i}=n$, the following equality holds:
\begin{equation}\label{eq:borodinf_relation}
 \mathrm{F}_{\bm}(x_{1},\dots,x_{n})=\left(\prod^{n}_{i=1} x_{i}^{-1}\right)\mathcal{F}_{\bm}(x_{1}^{-1},\dots,x_{n}^{-1};s).
\end{equation}
\end{prop}
\begin{proof}
We derive this equality by comparing the weights used to define $\mathrm{F}$ to those of $\mathcal{F}$. We begin with the observation that the lattice ~(\ref{lattice:definition_spinHL}) used to define $\mathrm{F}$, when reflected along NW-to-SE diagonal, is same as the lattice ~(\ref{lattice:definition-borodin_spinHL}) used to define $\mathcal{F}$. For convenience, we recall the weights ~(\ref{weights:def-spinHl}) used to define $\mathrm{F}$:
\begin{equation}
\begin{tabular}{ccccc}
\begin{tikzpicture}[scale=1,baseline=-2pt]
 \draw[gray] (0,0) rectangle (1,1); 
\node at (0.5,-0.3) {$\ss A$};
\node at (-0.3,0.5) {$\ss B$};
\node at (0.5,1.3) {$\ss C$};
\node at (1.3,0.5) {$\ss D$};
\draw[->] (0.5,-1) node[below] {$\ss x$}--(0.5,-0.7);
\draw[->] (-1,0.5) node[left] {$\ss s$}--(-0.7,0.5);
\end{tikzpicture}
&
\begin{tikzpicture}[scale=1,baseline=-2pt]
\draw[gray] (0,0) rectangle (1,1); 
\draw[->] (0,0.4)--(1,0.4);
\draw[->,rounded corners] (0,0.5)--(1,0.5);
\draw[->,rounded corners] (0,0.6)--(0.5,0.6)--(0.5,1);
\draw[->,rounded corners] (0.5,0)--(0.5,0.3)--(1,0.3);
\node at (0.5,-0.3) {$\ss 1$};
\node at (-0.3,0.5) {$\ss m$};
\node at (0.5,1.3) {$\ss 1$};
\node at (1.3,0.5) {$\ss m$};
\end{tikzpicture}
&
\begin{tikzpicture}[scale=1,baseline=-2pt]
 \draw[gray] (0,0) rectangle (1,1); 
\draw[->,rounded corners] (0,0.4)--(1,0.4);
\draw[->,rounded corners] (0,0.5)--(1,0.5);
\draw[->,rounded corners] (0,0.6)--(0.5,0.6)--(0.5,1);
\node at (0.5,-0.3) {$\ss 0$};
\node at (-0.3,0.5) {$\ss m$};
\node at (0.5,1.3) {$\ss 1$};
\node at (1.5,0.5) {$\ss m-1$};
\end{tikzpicture}
&
\begin{tikzpicture}[scale=1,baseline=-2pt]
 \draw[gray] (0,0) rectangle (1,1); 
\draw[->,rounded corners] (0,0.5)--(1,0.5);
\draw[->,rounded corners] (0,0.6)--(1,0.6);
\draw[->,rounded corners] (0,0.7)--(1,0.7);
\draw[->,rounded corners] (0.5,0)--(0.5,0.4)--(1,0.4);
\node at (0.5,-0.3) {$\ss 1$};
\node at (-0.3,0.5) {$\ss m$};
\node at (0.5,1.3) {$\ss 0$};
\node at (1.5,0.5) {$\ss m+1$};
\end{tikzpicture}
&
\begin{tikzpicture}[scale=1,baseline=-2pt]
 \draw[gray] (0,0) rectangle (1,1); 
\draw[->] (0,0.4)--(1,0.4);
\draw[->] (0,0.5)--(1,0.5);
\draw[->] (0,0.6)--(1,0.6);
\node at (0.5,-0.3) {$\ss 0$};
\node at (-0.3,0.5) {$\ss m$};
\node at (0.5,1.3) {$\ss 0$};
\node at (1.5,0.5) {$\ss m$};
\end{tikzpicture}\\
& $\ss \dfrac{1-q^{m}x s}{x-s}$
& $\ss \dfrac{1-q^{m}}{x-s}$
& $\ss \dfrac{(1-q^{m}s^{2})x}{x-s}$
& $\ss \dfrac{x-q^{m}s}{x-s} $
\end{tabular}
\end{equation} 

When the above vertices are reflected along the main diagonal they take the following form:

\begin{equation}
\label{weights:matching with borodin_flipped}
\begin{tabular}{ccccc}
\begin{tikzpicture}[scale=1,baseline=-2pt]
 \draw[gray] (0,0) rectangle (1,1); 
\node at (0.5,-0.3) {$\ss A$};
\node at (-0.3,0.5) {$\ss B$};
\node at (0.5,1.3) {$\ss C$};
\node at (1.3,0.5) {$\ss D$};
\draw[->] (0.5,-1) node[below] {$\ss s$}--(0.5,-0.7);
\draw[->] (-1,0.5) node[left] {$\ss x$}--(-0.7,0.5);
\end{tikzpicture}
&
\begin{tikzpicture}[scale=1,baseline=-2pt]
\draw[gray] (0,0) rectangle (1,1); 
\draw[->,rounded corners] (0,0.5)--(0.3,0.5)--(0.3,1);
\draw[->,rounded corners] (0.4,0)--(0.4,1);
\draw[->,rounded corners] (0.5,0)--(0.5,1);
\draw[->,rounded corners] (0.6,0)--(0.6,0.5)--(1,0.5);
\node at (-0.3,0.5) {$\ss 1$};
\node at (0.5,-0.3) {$\ss m$};
\node at (1.3,0.5) {$\ss 1$};
\node at (0.5,1.3) {$\ss m$};
\end{tikzpicture}
&
\begin{tikzpicture}[scale=1,baseline=-2pt]
\draw[gray] (0,0) rectangle (1,1); 
\draw[->,rounded corners] (0,0.5)--(0.4,0.5)--(0.4,1);
\draw[->,rounded corners] (0.5,0)--(0.5,1);
\draw[->,rounded corners] (0.6,0)--(0.6,1);
\node at (-0.3,0.5) {$\ss 1$};
\node at (0.5,-0.3) {$\ss m-1$};
\node at (1.3,0.5) {$\ss 0$};
\node at (0.5,1.3) {$\ss m$};
\end{tikzpicture}
&
\begin{tikzpicture}[scale=1,baseline=-2pt]
\draw[gray] (0,0) rectangle (1,1); 
\draw[->,rounded corners] (0.3,0)--(0.3,1);
\draw[->,rounded corners] (0.4,0)--(0.4,1);
\draw[->,rounded corners] (0.5,0)--(0.5,1);
\draw[->,rounded corners] (0.6,0)--(0.6,0.5)--(1,0.5);
\node at (-0.3,0.5) {$\ss 0$};
\node at (0.5,-0.3) {$\ss m+1$};
\node at (1.3,0.5) {$\ss 1$};
\node at (0.5,1.3) {$\ss m$};
\end{tikzpicture}
&
\begin{tikzpicture}[scale=1,baseline=-2pt]
\draw[gray] (0,0) rectangle (1,1); 
\draw[->,rounded corners] (0.4,0)--(0.4,1);
\draw[->,rounded corners] (0.5,0)--(0.5,1);
\draw[->,rounded corners] (0.6,0)--(0.6,1);
\node at (-0.3,0.5) {$\ss 0$};
\node at (0.5,-0.3) {$\ss m$};
\node at (1.3,0.5) {$\ss 0$};
\node at (0.5,1.3) {$\ss m$};
\end{tikzpicture}\\
& $\ss \dfrac{x^{-1}-q^{m} s}{1-s x^{-1}}$
& $\ss \dfrac{(1-q^{m})x^{-1}}{1-sx^{-1}}$
& $\ss \dfrac{(1-q^{m}s^{2})}{1-sx^{-1}}$
& $\ss \dfrac{1-s x^{-1}q^{m}}{1-sx^{-1}} $
\end{tabular}
\end{equation} 

\

Observe that the weights in~(\ref{weights:matching with borodin_flipped}) closely match those in~(\ref{weights:def-borodin_spinHL}) under the substitution  $x \mapsto x^{-1}$. However, the match is not exact due to a discrepancy involving an extra factor of $x^{-1}$ between the second and third vertex types.

In each configuration of the lattice~(\ref{lattice:definition-borodin_spinHL}), every row contains exactly one more vertex of the second type than of the third. As a result, the discrepancy contributes an overall factor of $x_1^{-1} \cdots x_n^{-1}$ to the partition function of the lattice~(\ref{lattice:definition-borodin_spinHL}) with weights~(\ref{weights:matching with borodin_flipped}), when compared to the partition function of the same lattice with weights~(\ref{weights:def-borodin_spinHL}) (after applying the substitution $x\mapsto x^{-1}$).
\end{proof}

\begin{cor}
Fix a tuple of non-negative integers $\bm$ with the corresponding partition $\mu$. The following symmetrization formula holds:
\begin{equation}\label{eq:symmformulafor_spinHL}
   \mathrm{F}_{\bm}(x_{1},\dots,x_{n};s)= \dfrac{(1-q)^{n}}{\prod^{n}_{i=1}(x_{i}-s)} \sum_{\sigma\in S_{n}}\sigma  \left(\prod_{1\leq i<j\leq n}
\dfrac{x_{j}-q x_{i}}{x_{j}-x_{i}} \cdot \prod^{n}_{i=1} \left(\dfrac{1-sx_{i}}{x_{i}-s}\right)^{\mu_{i}}\right)
\end{equation}
\end{cor}
\section{Main Theorems}
 \label{sec:maintheorems}
\subsection{Six-vertex puzzles} Consider the tiles below. Each tile has an associated weight, as indicated underneath its graphical depiction. 
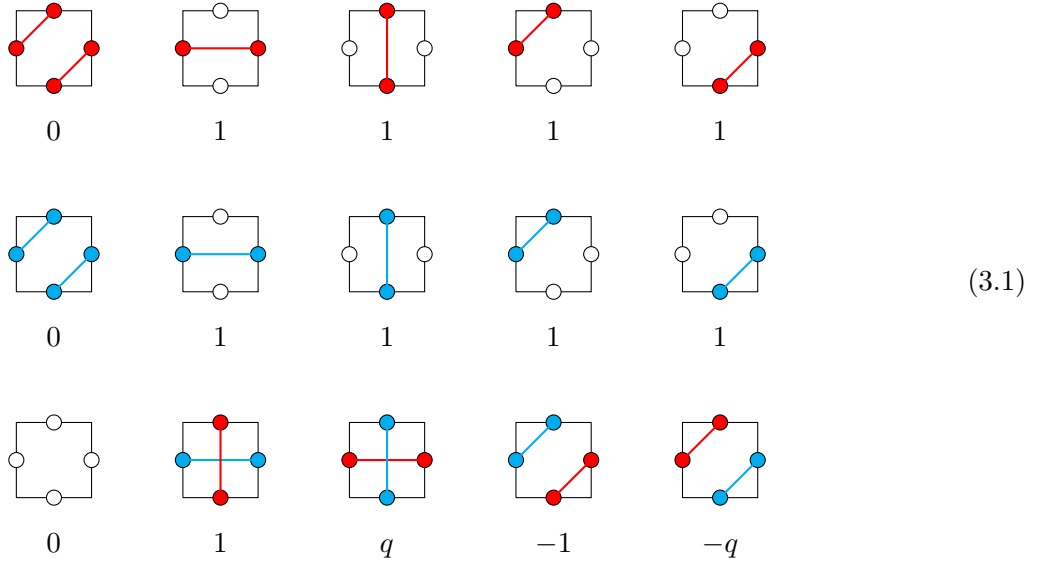
\begin{equation}
\label{weights:wavefunctions-puzzleweights}
\begin{tabular}{c@{\hskip 1cm}c@{\hskip 1cm}c@{\hskip 1cm}c@{\hskip 1cm}c}
\begin{tikzpicture}[baseline=(current bounding box.center)]
    \draw (0:0)--++(0:1)--++(90:1)--++(180:1)--++(-90:1);
    \draw[fill=red] (0:0)++(90:0.5) circle (0.1);
    \draw[fill=red] (0:0)++(0:0.5) circle (0.1); 
    \draw[fill=red] (0:0)++(0:1)++(90:0.5) circle (0.1);
    \draw[fill=red] (0:0)++(0:0.5)++(90:1) circle (0.1);
    \draw[red,thick] (0,0.5)--(0.5,1);
    \draw[red,thick] (0.5,0)--(1,0.5);
\end{tikzpicture}
&
\begin{tikzpicture}[baseline=(current bounding box.center)]
    \draw (0:0)--++(0:1)--++(90:1)--++(180:1)--++(-90:1);
    \draw[fill=red] (0:0)++(90:0.5) circle (0.1);
    \draw[fill=white] (0:0)++(0:0.5) circle (0.1); 
    \draw[fill=red] (0:0)++(0:1)++(90:0.5) circle (0.1);
    \draw[fill=white] (0:0)++(0:0.5)++(90:1) circle (0.1);
    \draw[red,thick] (0,0.5)--(1,0.5);
\end{tikzpicture}
&
\begin{tikzpicture}[baseline=(current bounding box.center)]
    \draw (0:0)--++(0:1)--++(90:1)--++(180:1)--++(-90:1);
    \draw[fill=white] (0:0)++(90:0.5) circle (0.1);
    \draw[fill=red] (0:0)++(0:0.5) circle (0.1); 
    \draw[fill=white] (0:0)++(0:1)++(90:0.5) circle (0.1);
    \draw[fill=red] (0:0)++(0:0.5)++(90:1) circle (0.1);
    \draw[red,thick] (0.5,0)--(0.5,1);
\end{tikzpicture}
&
\begin{tikzpicture}[baseline=(current bounding box.center)]
    \draw (0:0)--++(0:1)--++(90:1)--++(180:1)--++(-90:1);
    \draw[fill=red] (0:0)++(90:0.5) circle (0.1);
    \draw[fill=white] (0:0)++(0:0.5) circle (0.1); 
    \draw[fill=white] (0:0)++(0:1)++(90:0.5) circle (0.1);
    \draw[fill=red] (0:0)++(0:0.5)++(90:1) circle (0.1);
    \draw[red,thick] (0,0.5)--(0.5,1);
\end{tikzpicture}
&
\begin{tikzpicture}[baseline=(current bounding box.center)]
    \draw (0:0)--++(0:1)--++(90:1)--++(180:1)--++(-90:1);
    \draw[fill=white] (0:0)++(90:0.5) circle (0.1);
    \draw[fill=red] (0:0)++(0:0.5) circle (0.1); 
    \draw[fill=red] (0:0)++(0:1)++(90:0.5) circle (0.1);
    \draw[fill=white] (0:0)++(0:0.5)++(90:1) circle (0.1);
    \draw[red,thick] (0.5,0)--(1,0.5);
\end{tikzpicture}
\\[2 em]
$0$ & $1$ & $1$ & $1$ & $1$\\[2 em]
  \begin{tikzpicture}[baseline=(current bounding box.center)]
    \draw (0:0)--++(0:1)--++(90:1)--++(180:1)--++(-90:1);
    \draw[fill=cyan] (0:0)++(90:0.5) circle (0.1);
    \draw[fill=cyan] (0:0)++(0:0.5) circle (0.1); 
    \draw[fill=cyan] (0:0)++(0:1)++(90:0.5) circle (0.1);
    \draw[fill=cyan] (0:0)++(0:0.5)++(90:1) circle (0.1);
    \draw[cyan,thick] (0,0.5)--(0.5,1);
    \draw[cyan,thick] (0.5,0)--(1,0.5);
\end{tikzpicture}
&
\begin{tikzpicture}[baseline=(current bounding box.center)]
    \draw (0:0)--++(0:1)--++(90:1)--++(180:1)--++(-90:1);
    \draw[fill=cyan] (0:0)++(90:0.5) circle (0.1);
    \draw[fill=white] (0:0)++(0:0.5) circle (0.1); 
    \draw[fill=cyan] (0:0)++(0:1)++(90:0.5) circle (0.1);
    \draw[fill=white] (0:0)++(0:0.5)++(90:1) circle (0.1);
    \draw[cyan,thick] (0,0.5)--(1,0.5);
\end{tikzpicture}
&
\begin{tikzpicture}[baseline=(current bounding box.center)]
    \draw (0:0)--++(0:1)--++(90:1)--++(180:1)--++(-90:1);
    \draw[fill=white] (0:0)++(90:0.5) circle (0.1);
    \draw[fill=cyan] (0:0)++(0:0.5) circle (0.1); 
    \draw[fill=white] (0:0)++(0:1)++(90:0.5) circle (0.1);
    \draw[fill=cyan] (0:0)++(0:0.5)++(90:1) circle (0.1);
    \draw[cyan,thick] (0.5,0)--(0.5,1);
\end{tikzpicture}
&
\begin{tikzpicture}[baseline=(current bounding box.center)]
    \draw (0:0)--++(0:1)--++(90:1)--++(180:1)--++(-90:1);
    \draw[fill=cyan] (0:0)++(90:0.5) circle (0.1);
    \draw[fill=white] (0:0)++(0:0.5) circle (0.1); 
    \draw[fill=white] (0:0)++(0:1)++(90:0.5) circle (0.1);
    \draw[fill=cyan] (0:0)++(0:0.5)++(90:1) circle (0.1);
    \draw[cyan,thick] (0,0.5)--(0.5,1);
\end{tikzpicture}
&
\begin{tikzpicture}[baseline=(current bounding box.center)]
    \draw (0:0)--++(0:1)--++(90:1)--++(180:1)--++(-90:1);
    \draw[fill=white] (0:0)++(90:0.5) circle (0.1);
    \draw[fill=cyan] (0:0)++(0:0.5) circle (0.1); 
    \draw[fill=cyan] (0:0)++(0:1)++(90:0.5) circle (0.1);
    \draw[fill=white] (0:0)++(0:0.5)++(90:1) circle (0.1);
    \draw[cyan,thick] (0.5,0)--(1,0.5);
\end{tikzpicture}
\\[2 em] 
$0$ & $1$ & $1$ & $1$ & $1$\\[2 em]
  \begin{tikzpicture}[baseline=(current bounding box.center)]
    \draw (0:0)--++(0:1)--++(90:1)--++(180:1)--++(-90:1);
    \draw[fill=white] (0:0)++(90:0.5) circle (0.1);
    \draw[fill=white] (0:0)++(0:0.5) circle (0.1); 
    \draw[fill=white] (0:0)++(0:1)++(90:0.5) circle (0.1);
    \draw[fill=white] (0:0)++(0:0.5)++(90:1) circle (0.1);
\end{tikzpicture}
&
\begin{tikzpicture}[baseline=(current bounding box.center)]
    \draw (0:0)--++(0:1)--++(90:1)--++(180:1)--++(-90:1);
    \draw[fill=cyan] (0:0)++(90:0.5) circle (0.1);
    \draw[fill=red] (0:0)++(0:0.5) circle (0.1); 
    \draw[fill=cyan] (0:0)++(0:1)++(90:0.5) circle (0.1);
    \draw[fill=red] (0:0)++(0:0.5)++(90:1) circle (0.1);
    \draw[cyan,thick] (0,0.5)--(1,0.5);
    \draw[red,thick] (0.5,0)--(0.5,1);
\end{tikzpicture}
&
\begin{tikzpicture}[baseline=(current bounding box.center)]
    \draw (0:0)--++(0:1)--++(90:1)--++(180:1)--++(-90:1);
    \draw[fill=red] (0:0)++(90:0.5) circle (0.1);
    \draw[fill=cyan] (0:0)++(0:0.5) circle (0.1); 
    \draw[fill=red] (0:0)++(0:1)++(90:0.5) circle (0.1);
    \draw[fill=cyan] (0:0)++(0:0.5)++(90:1) circle (0.1);
    \draw[red,thick] (0,0.5)--(1,0.5);
    \draw[cyan,thick] (0.5,0)--(0.5,1);
\end{tikzpicture}
&
\begin{tikzpicture}[baseline=(current bounding box.center)]
    \draw (0:0)--++(0:1)--++(90:1)--++(180:1)--++(-90:1);
    \draw[fill=cyan] (0:0)++(90:0.5) circle (0.1);
    \draw[fill=red] (0:0)++(0:0.5) circle (0.1); 
    \draw[fill=red] (0:0)++(0:1)++(90:0.5) circle (0.1);
    \draw[fill=cyan] (0:0)++(0:0.5)++(90:1) circle (0.1);
    \draw[cyan,thick] (0,0.5)--(0.5,1);
    \draw[red,thick] (0.5,0)--(1,0.5);
\end{tikzpicture}
&
\begin{tikzpicture}[baseline=(current bounding box.center)]
    \draw (0:0)--++(0:1)--++(90:1)--++(180:1)--++(-90:1);
    \draw[fill=red] (0:0)++(90:0.5) circle (0.1);
    \draw[fill=cyan] (0:0)++(0:0.5) circle (0.1); 
    \draw[fill=cyan] (0:0)++(0:1)++(90:0.5) circle (0.1);
    \draw[fill=red] (0:0)++(0:0.5)++(90:1) circle (0.1);
    \draw[red,thick] (0,0.5)--(0.5,1);
    \draw[cyan,thick] (0.5,0)--(1,0.5);
\end{tikzpicture}\\[2 em] 
$0$ & $1$ & $q$ & $-1$ & $-q$\\
\end{tabular}
\end{equation}

\begin{dfn}  
Let $P$ and $N$ be two positive integers. A \emph{six-vertex puzzle} is a tiling of a rectangular region of length $P$ and height $N$ using the tiles described above. We denote by $\mathcal{C}^{k,w}_{l,m}$ the set of all six-vertex puzzles where:  
\begin{itemize}  
    \item The top boundary is determined by a binary string $l$ of length $N$, consisting of red and white particles.  
    \item The bottom boundary is determined by a binary string $w$ of length $N$, consisting of red and white particles.  
    \item The left boundary is determined by a binary string $m$ of length $P$, consisting of red and blue particles.  
    \item The right boundary is determined by a binary string $k$ of length $P$, consisting of blue and white particles.  
\end{itemize}  
\end{dfn}

\begin{ex}
We give three examples. For the first two, we have $P=5$ and $N=3$, while for the last example, we have $P=N=5$.
 \begin{center}
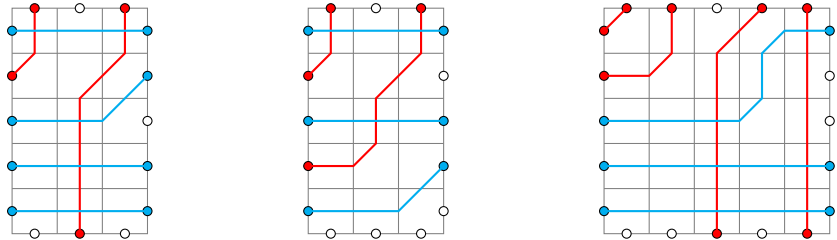
    \begin{tabular}{c@{\hskip 2cm}c@{\hskip 2cm}c}
    \begin{tikzpicture}[scale=0.6,baseline=(current bounding box.center)]
\draw[gray](0.5,0.5) rectangle (3.5,5.5);  
\foreach \x in {1,2,3,4}{
\draw[gray] (0.5,\x+0.5)--(3.5,\x+0.5);
};
\foreach \x in {1,2}{
\draw[gray] (\x+0.5,0.5)--(\x+0.5,5.5);
};
\ebull{1}{0.5}{0.1};
\rbull{2}{0.5}{0.1};
\ebull{3}{0.5}{0.1};
\rbull{1}{5.5}{0.1};
\ebull{2}{5.5}{0.1};
\rbull{3}{5.5}{0.1};
\cbull{0.5}{5}{0.1};
\rbull{0.5}{4}{0.1};
\cbull{0.5}{3}{0.1};
\cbull{0.5}{2}{0.1};
\cbull{0.5}{1}{0.1};
\cbull{3.5}{5}{0.1};
\cbull{3.5}{4}{0.1};
\ebull{3.5}{3}{0.1};
\cbull{3.5}{2}{0.1};
\cbull{3.5}{1}{0.1};
\draw[red,thick](0.5,4)--(1,4.5)--(1,5.5);
\draw[red,thick](2,0.5)--(2,3.5)--(3,4.5)--(3,5.5);
\draw[cyan,thick] (0.5,1)--(3.5,1);
\draw[cyan,thick] (0.5,2)--(3.5,2);
\draw[cyan,thick] (0.5,3)--(2.5,3)--(3.5,4);
\draw[cyan,thick] (0.5,5)--(3.5,5);
\end{tikzpicture}
&
\begin{tikzpicture}[scale=0.6,baseline=(current bounding box.center)]
\draw[gray](0.5,0.5) rectangle (3.5,5.5);  
\foreach \x in {1,2,3,4}{
\draw[gray] (0.5,\x+0.5)--(3.5,\x+0.5);
};
\foreach \x in {1,2}{
\draw[gray] (\x+0.5,0.5)--(\x+0.5,5.5);
};
\ebull{1}{0.5}{0.1};
\ebull{2}{0.5}{0.1};
\ebull{3}{0.5}{0.1};
\rbull{1}{5.5}{0.1};
\ebull{2}{5.5}{0.1};
\rbull{3}{5.5}{0.1};
\cbull{0.5}{5}{0.1};
\rbull{0.5}{4}{0.1};
\cbull{0.5}{3}{0.1};
\rbull{0.5}{2}{0.1};
\cbull{0.5}{1}{0.1};
\cbull{3.5}{5}{0.1};
\ebull{3.5}{4}{0.1};
\cbull{3.5}{3}{0.1};
\cbull{3.5}{2}{0.1};
\ebull{3.5}{1}{0.1};
\draw[red,thick](0.5,4)--(1,4.5)--(1,5.5);
\draw[red,thick] (0.5,2)--(1.5,2)--(2,2.5)--(2,3.5)--(3,4.5)--(3,5.5);
\draw[cyan,thick] (0.5,1)--(2.5,1)--(3.5,2);
\draw[cyan,thick] (0.5,3)--(3.5,3);
\draw[cyan,thick] (0.5,5)--(3.5,5);
\end{tikzpicture}
&
    \begin{tikzpicture}[scale=0.6,baseline=(current bounding box.center)]
\draw[gray](0.5,0.5) rectangle (5.5,5.5);  
\foreach \x in {1,2,3,4}{
\draw[gray] (0.5,\x+0.5)--(5.5,\x+0.5);
};
\foreach \x in {1,2,3,4}{
\draw[gray] (\x+0.5,0.5)--(\x+0.5,5.5);
};
\ebull{1}{0.5}{0.1};
\ebull{2}{0.5}{0.1};
\rbull{3}{0.5}{0.1};
\ebull{4}{0.5}{0.1};
\rbull{5}{0.5}{0.1};
\rbull{1}{5.5}{0.1};
\rbull{2}{5.5}{0.1};
\ebull{3}{5.5}{0.1};
\rbull{4}{5.5}{0.1};
\rbull{5}{5.5}{0.1};
\rbull{0.5}{5}{0.1};
\rbull{0.5}{4}{0.1};
\cbull{0.5}{3}{0.1};
\cbull{0.5}{2}{0.1};
\cbull{0.5}{1}{0.1};
\cbull{5.5}{5}{0.1};
\ebull{5.5}{4}{0.1};
\ebull{5.5}{3}{0.1};
\cbull{5.5}{2}{0.1};
\cbull{5.5}{1}{0.1};
\draw[red,thick] (0.5,5)--(1,5.5);
\draw[red,thick](0.5,4)--(1.5,4)--(2,4.5)--(2,5.5);
\draw[red,thick](3,0.5)--(3,4.5)--(4,5.5);
\draw[red,thick](5,0.5)--(5,5.5);
\draw[cyan,thick] (0.5,1)--(5.5,1);
\draw[cyan,thick] (0.5,2)--(5.5,2);
\draw[cyan,thick] (0.5,3)--(3.5,3)--(4,3.5)--(4,4.5)--(4.5,5)--(5.5,5);
\end{tikzpicture}\\
\end{tabular}
\end{center}   
\end{ex}

 \begin{thm}
 \label{theorem:puzzlesforwavefunctions}
For positive integers $P$, $N$ and $n$, consider binary strings $m,k\in \{ \begin{tikzpicture}
    \ebull{0}{0}{0.1}
\end{tikzpicture},
\begin{tikzpicture}
    \bbull{0}{0}{0.1}
\end{tikzpicture}\}^P$ and $w,l\in \{ \begin{tikzpicture}
    \ebull{0}{0}{0.1}
\end{tikzpicture},
\begin{tikzpicture}
    \bbull{0}{0}{0.1}
\end{tikzpicture}\}^N$ where $l$ has exactly $n$ more particles than $w$. Then we have: 
\begin{equation}
    \rh_{m}(x_1,\dots,x_n)\hspace{2mm} \rh_{l/w}(x_1,\dots,x_{n})=\sum_{k}C^{k,w}_{l,m}\hspace{2mm} \rh_{k}(x_1,\dots,x_n)
\end{equation}
where 
\begin{equation}    
\label{puzzledef:maintheorem}
C^{k,w}_{l,m}= \text{  }
 \begin{tikzpicture}[scale=0.6,baseline=(current bounding box.center)]
\draw (0:0)--++(0:5)--++(90:5)--++(180:5)--++(-90:5);
\foreach \x in {1,2,3,4}{
\draw[gray] (0:\x)--++(90:5);
\draw[gray] (90:\x)--++(0:5);
};
\foreach \x/\y in {5/1,4/2,3/3}{
\path (0:0)++(0:\x-0.5)++(-90:0.5) node {$\ss \omega_{\y}$};
\path (0:0)++(0:\x-0.5)++(90:5.5) node {$\ss \lambda_{\y}$};
\path (0:0)++(180:0.5)++(90:\x-0.5) node {$\ss \mu_{\y}$};
\path (0:0)++(0:5.5)++(90:\x-0.5) node {$\ss \kappa_{\y}$};
\path (0:0)++(180:0.5)++(90:1.5) node[rotate=90] {$\ss \dots$};
\path (0:0)++(0:5.5)++(90:1.5) node[rotate=90] {$\ss \dots$};
\path (0:0)++(0:1.5)++(-90:0.5) node {$\ss \dots$};
\path (0:0)++(0:1.5)++(90:5.5) node {$\ss \dots$};
};
\path (0:0)++(180:0.5)++(90:0.5) node {$\ss \mu_{P}$};
\path (0:0)++(0:5.5)++(90:0.5) node {$\ss \kappa_{P}$};
\path (0:0)++(0:0.5)++(-90:0.5) node {$\ss \omega_{N}$};
\path (0:0)++(0:0.5)++(90:5.5) node {$\ss \lambda_{N}$};
\end{tikzpicture}
\end{equation}
\begin{align*}
\lambda_{i}=\begin{cases}
    \begin{tikzpicture}
    \rbull{0}{0}{0.1}
\end{tikzpicture} & \text{ if } l_{i}=\begin{tikzpicture}
    \bbull{0}{0}{0.1}
\end{tikzpicture} \\
\begin{tikzpicture}
    \ebull{0}{0}{0.1}
\end{tikzpicture}& \text{ if } l_{i}=\begin{tikzpicture}
    \ebull{0}{0}{0.1}
\end{tikzpicture}
\end{cases}
\qquad
\omega_{i}=\begin{cases}
    \begin{tikzpicture}
    \rbull{0}{0}{0.1}
\end{tikzpicture} & \text{ if } w_{i}=\begin{tikzpicture}
    \bbull{0}{0}{0.1}
\end{tikzpicture}\\
\begin{tikzpicture}
    \ebull{0}{0}{0.1}
\end{tikzpicture}& \text{ if } w_{i}=\begin{tikzpicture}
    \ebull{0}{0}{0.1}
\end{tikzpicture}
\end{cases}
\qquad
\mu_{i}=\begin{cases}
    \begin{tikzpicture}
    \rbull{0}{0}{0.1}
\end{tikzpicture} & \text{ if } m_{i}=\begin{tikzpicture}
    \bbull{0}{0}{0.1}
\end{tikzpicture}\\
\begin{tikzpicture}
    \sbull{0}{0}{0.1}
\end{tikzpicture}& \text{ if } m_{i}=\begin{tikzpicture}
    \ebull{0}{0}{0.1}
\end{tikzpicture}
\end{cases}
\qquad
\kappa_{i}=\begin{cases}
    \begin{tikzpicture}
    \ebull{0}{0}{0.1}
\end{tikzpicture} & \text{ if } k_{i}=\begin{tikzpicture}
    \bbull{0}{0}{0.1}
\end{tikzpicture}\\
\begin{tikzpicture}
    \sbull{0}{0}{0.1}
\end{tikzpicture}& \text{ if  } k_{i}=\begin{tikzpicture}
    \ebull{0}{0}{0.1}
\end{tikzpicture}
\end{cases}
\end{align*}
Furthermore, $C^{k,w}_{l,m}$ are polynomials in $q$ with positive coefficients, up to an overall multiplicative sign.
\end{thm}
Below, we provide two examples to illustrate the theorem mentioned above. 
\begin{ex}
Fix $P=5$, $N=3$ and $n=1$. The product $\rh_{(0,1,0,0,0)}\rh_{(1,0,1))/(0,1,0)}$ has the following expansion:
\begin{align*}
    \rh_{(0,1,0,0,0)}& \rh_{(1,0,1) / (0,1,0)}=\rh_{ (0,1,0,0,0)} -(1+q+q)\rh_{ (0,0,1,0,0)} +(q+q+q^{2})\rh_{ (0,0,0,1,0) }- q^{2} \rh_{ (0,0,0,0,1) }\\[1 em]
&=
\begin{tikzpicture}[scale=0.5,baseline=(current bounding box.center)]
\draw[gray](0.5,0.5) rectangle (3.5,5.5);  
\foreach \x in {1,2,3,4}{
\draw[gray] (0.5,\x+0.5)--(3.5,\x+0.5);
};
\foreach \x in {1,2}{
\draw[gray] (\x+0.5,0.5)--(\x+0.5,5.5);
};
\ebull{1}{0.5}{0.1};
\rbull{2}{0.5}{0.1};
\ebull{3}{0.5}{0.1};
\rbull{1}{5.5}{0.1};
\ebull{2}{5.5}{0.1};
\rbull{3}{5.5}{0.1};
\cbull{0.5}{5}{0.1};
\rbull{0.5}{4}{0.1};
\cbull{0.5}{3}{0.1};
\cbull{0.5}{2}{0.1};
\cbull{0.5}{1}{0.1};
\cbull{3.5}{5}{0.1};
\ebull{3.5}{4}{0.1};
\cbull{3.5}{3}{0.1};
\cbull{3.5}{2}{0.1};
\cbull{3.5}{1}{0.1};
\draw[red,thick](0.5,4)--(1,4.5)--(1,5.5);
\draw[red,thick](2,0.5)--(2,3.5)--(3,4.5)--(3,5.5);
\draw[cyan,thick] (0.5,1)--(3.5,1);
\draw[cyan,thick] (0.5,2)--(3.5,2);
\draw[cyan,thick] (0.5,3)--(3.5,3);
\draw[cyan,thick] (0.5,5)--(3.5,5);
\end{tikzpicture} \hspace{2mm}\rh_{ (0,1,0,0,0)}
+\left(
\begin{tikzpicture}[scale=0.5,baseline=(current bounding box.center)]
\draw[gray](0.5,0.5) rectangle (3.5,5.5);  
\foreach \x in {1,2,3,4}{
\draw[gray] (0.5,\x+0.5)--(3.5,\x+0.5);
};
\foreach \x in {1,2}{
\draw[gray] (\x+0.5,0.5)--(\x+0.5,5.5);
};
\ebull{1}{0.5}{0.1};
\rbull{2}{0.5}{0.1};
\ebull{3}{0.5}{0.1};
\rbull{1}{5.5}{0.1};
\ebull{2}{5.5}{0.1};
\rbull{3}{5.5}{0.1};
\cbull{0.5}{5}{0.1};
\rbull{0.5}{4}{0.1};
\cbull{0.5}{3}{0.1};
\cbull{0.5}{2}{0.1};
\cbull{0.5}{1}{0.1};
\cbull{3.5}{5}{0.1};
\cbull{3.5}{4}{0.1};
\ebull{3.5}{3}{0.1};
\cbull{3.5}{2}{0.1};
\cbull{3.5}{1}{0.1};
\draw[red,thick](0.5,4)--(1,4.5)--(1,5.5);
\draw[red,thick](2,0.5)--(2,3.5)--(3,4.5)--(3,5.5);
\draw[cyan,thick] (0.5,1)--(3.5,1);
\draw[cyan,thick] (0.5,2)--(3.5,2);
\draw[cyan,thick] (0.5,3)--(2.5,3)--(3.5,4);
\draw[cyan,thick] (0.5,5)--(3.5,5);
\end{tikzpicture}
+
\begin{tikzpicture}[scale=0.5,baseline=(current bounding box.center)]
\draw[gray](0.5,0.5) rectangle (3.5,5.5);  
\foreach \x in {1,2,3,4}{
\draw[gray] (0.5,\x+0.5)--(3.5,\x+0.5);
};
\foreach \x in {1,2}{
\draw[gray] (\x+0.5,0.5)--(\x+0.5,5.5);
};
\ebull{1}{0.5}{0.1};
\rbull{2}{0.5}{0.1};
\ebull{3}{0.5}{0.1};
\rbull{1}{5.5}{0.1};
\ebull{2}{5.5}{0.1};
\rbull{3}{5.5}{0.1};
\cbull{0.5}{5}{0.1};
\rbull{0.5}{4}{0.1};
\cbull{0.5}{3}{0.1};
\cbull{0.5}{2}{0.1};
\cbull{0.5}{1}{0.1};
\cbull{3.5}{5}{0.1};
\cbull{3.5}{4}{0.1};
\ebull{3.5}{3}{0.1};
\cbull{3.5}{2}{0.1};
\cbull{3.5}{1}{0.1};
\draw[red,thick](0.5,4)--(1,4.5)--(1,5.5);
\draw[red,thick](2,0.5)--(2,2.5)--(3,3.5)--(3,5.5);
\draw[cyan,thick] (0.5,1)--(3.5,1);
\draw[cyan,thick] (0.5,2)--(3.5,2);
\draw[cyan,thick] (0.5,3)--(1.5,4)--(3.5,4);
\draw[cyan,thick] (0.5,5)--(3.5,5);
\end{tikzpicture}
+
\begin{tikzpicture}[scale=0.5,baseline=(current bounding box.center)]
\draw[gray](0.5,0.5) rectangle (3.5,5.5);  
\foreach \x in {1,2,3,4}{
\draw[gray] (0.5,\x+0.5)--(3.5,\x+0.5);
};
\foreach \x in {1,2}{
\draw[gray] (\x+0.5,0.5)--(\x+0.5,5.5);
};
\ebull{1}{0.5}{0.1};
\rbull{2}{0.5}{0.1};
\ebull{3}{0.5}{0.1};
\rbull{1}{5.5}{0.1};
\ebull{2}{5.5}{0.1};
\rbull{3}{5.5}{0.1};
\cbull{0.5}{5}{0.1};
\rbull{0.5}{4}{0.1};
\cbull{0.5}{3}{0.1};
\cbull{0.5}{2}{0.1};
\cbull{0.5}{1}{0.1};
\cbull{3.5}{5}{0.1};
\cbull{3.5}{4}{0.1};
\ebull{3.5}{3}{0.1};
\cbull{3.5}{2}{0.1};
\cbull{3.5}{1}{0.1};
\draw[red,thick](0.5,4)--(1,4.5)--(1,5.5);
\draw[red,thick](2,0.5)--(2,2.5)--(3,3.5)--(3,5.5);
\draw[cyan,thick] (0.5,1)--(3.5,1);
\draw[cyan,thick] (0.5,2)--(3.5,2);
\draw[cyan,thick] (0.5,3)--(1.5,3)--(2.5,4)--(3.5,4);
\draw[cyan,thick] (0.5,5)--(3.5,5);
\end{tikzpicture}
\right) \hspace{2mm}\rh_{ (0,0,1,0,0)} \\[1 em]
&+\left(\begin{tikzpicture}[scale=0.5,baseline=(current bounding box.center)]
\draw[gray](0.5,0.5) rectangle (3.5,5.5);  
\foreach \x in {1,2,3,4}{
\draw[gray] (0.5,\x+0.5)--(3.5,\x+0.5);
};
\foreach \x in {1,2}{
\draw[gray] (\x+0.5,0.5)--(\x+0.5,5.5);
};
\ebull{1}{0.5}{0.1};
\rbull{2}{0.5}{0.1};
\ebull{3}{0.5}{0.1};
\rbull{1}{5.5}{0.1};
\ebull{2}{5.5}{0.1};
\rbull{3}{5.5}{0.1};
\cbull{0.5}{5}{0.1};
\rbull{0.5}{4}{0.1};
\cbull{0.5}{3}{0.1};
\cbull{0.5}{2}{0.1};
\cbull{0.5}{1}{0.1};
\cbull{3.5}{5}{0.1};
\cbull{3.5}{4}{0.1};
\cbull{3.5}{3}{0.1};
\ebull{3.5}{2}{0.1};
\cbull{3.5}{1}{0.1};
\draw[red,thick](0.5,4)--(1,4.5)--(1,5.5);
\draw[red,thick](2,0.5)--(2,2.5)--(3,3.5)--(3,5.5);
\draw[cyan,thick] (0.5,1)--(3.5,1);
\draw[cyan,thick] (0.5,2)--(2.5,2)--(3.5,3);
\draw[cyan,thick] (0.5,3)--(1.5,3)--(2.5,4)--(3.5,4);
\draw[cyan,thick] (0.5,5)--(3.5,5);
\end{tikzpicture}
+
\begin{tikzpicture}[scale=0.5,baseline=(current bounding box.center)]
\draw[gray](0.5,0.5) rectangle (3.5,5.5);  
\foreach \x in {1,2,3,4}{
\draw[gray] (0.5,\x+0.5)--(3.5,\x+0.5);
};
\foreach \x in {1,2}{
\draw[gray] (\x+0.5,0.5)--(\x+0.5,5.5);
};
\ebull{1}{0.5}{0.1};
\rbull{2}{0.5}{0.1};
\ebull{3}{0.5}{0.1};
\rbull{1}{5.5}{0.1};
\ebull{2}{5.5}{0.1};
\rbull{3}{5.5}{0.1};
\cbull{0.5}{5}{0.1};
\rbull{0.5}{4}{0.1};
\cbull{0.5}{3}{0.1};
\cbull{0.5}{2}{0.1};
\cbull{0.5}{1}{0.1};
\cbull{3.5}{5}{0.1};
\cbull{3.5}{4}{0.1};
\cbull{3.5}{3}{0.1};
\ebull{3.5}{2}{0.1};
\cbull{3.5}{1}{0.1};
\draw[red,thick](0.5,4)--(1,4.5)--(1,5.5);
\draw[red,thick](2,0.5)--(2,2.5)--(3,3.5)--(3,5.5);
\draw[cyan,thick] (0.5,1)--(3.5,1);
\draw[cyan,thick] (0.5,2)--(2.5,2)--(3.5,3);
\draw[cyan,thick] (0.5,3)--(1.5,4)--(3.5,4);
\draw[cyan,thick] (0.5,5)--(3.5,5);
\end{tikzpicture}
+
\begin{tikzpicture}[scale=0.5,baseline=(current bounding box.center)]
\draw[gray](0.5,0.5) rectangle (3.5,5.5);  
\foreach \x in {1,2,3,4}{
\draw[gray] (0.5,\x+0.5)--(3.5,\x+0.5);
};
\foreach \x in {1,2}{
\draw[gray] (\x+0.5,0.5)--(\x+0.5,5.5);
};
\ebull{1}{0.5}{0.1};
\rbull{2}{0.5}{0.1};
\ebull{3}{0.5}{0.1};
\rbull{1}{5.5}{0.1};
\ebull{2}{5.5}{0.1};
\rbull{3}{5.5}{0.1};
\cbull{0.5}{5}{0.1};
\rbull{0.5}{4}{0.1};
\cbull{0.5}{3}{0.1};
\cbull{0.5}{2}{0.1};
\cbull{0.5}{1}{0.1};
\cbull{3.5}{5}{0.1};
\cbull{3.5}{4}{0.1};
\cbull{3.5}{3}{0.1};
\ebull{3.5}{2}{0.1};
\cbull{3.5}{1}{0.1};
\draw[red,thick](0.5,4)--(1,4.5)--(1,5.5);
\draw[red,thick](2,0.5)--(2,1.5)--(3,2.5)--(3,5.5);
\draw[cyan,thick] (0.5,1)--(3.5,1);
\draw[cyan,thick] (0.5,2)--(1.5,2)--(2.5,3)--(3.5,3);
\draw[cyan,thick] (0.5,3)--(1.5,4)--(3.5,4);
\draw[cyan,thick] (0.5,5)--(3.5,5);
\end{tikzpicture}\right)\rh_{(0,0,0,1,0)}+
\begin{tikzpicture}[scale=0.5,baseline=(current bounding box.center)]
\draw[gray](0.5,0.5) rectangle (3.5,5.5);  
\foreach \x in {1,2,3,4}{
\draw[gray] (0.5,\x+0.5)--(3.5,\x+0.5);
};
\foreach \x in {1,2}{
\draw[gray] (\x+0.5,0.5)--(\x+0.5,5.5);
};
\ebull{1}{0.5}{0.1};
\rbull{2}{0.5}{0.1};
\ebull{3}{0.5}{0.1};
\rbull{1}{5.5}{0.1};
\ebull{2}{5.5}{0.1};
\rbull{3}{5.5}{0.1};
\cbull{0.5}{5}{0.1};
\rbull{0.5}{4}{0.1};
\cbull{0.5}{3}{0.1};
\cbull{0.5}{2}{0.1};
\cbull{0.5}{1}{0.1};
\cbull{3.5}{5}{0.1};
\cbull{3.5}{4}{0.1};
\cbull{3.5}{3}{0.1};
\cbull{3.5}{2}{0.1};
\ebull{3.5}{1}{0.1};
\draw[red,thick](0.5,4)--(1,4.5)--(1,5.5);
\draw[red,thick](2,0.5)--(2,1.5)--(3,2.5)--(3,5.5);
\draw[cyan,thick] (0.5,1)--(2.5,1)--(3.5,2);
\draw[cyan,thick] (0.5,2)--(1.5,2)--(2.5,3)--(3.5,3);
\draw[cyan,thick] (0.5,3)--(1.5,4)--(3.5,4);
\draw[cyan,thick] (0.5,5)--(3.5,5);
\end{tikzpicture}\hspace{2mm}\rh_{(0,0,0,0,1)}
\end{align*}
\end{ex} 

\begin{ex} Fix $P,N=4$ and $n=2$, we compute
\begin{align*}
    \rh_{(1,1,0,0)} \rh_{(1,1,0,0)}&=\rh_{(1,1,0,0)}-q\rh_{(1,0,1,0)}-q^{2}\rh_{(0,1,0,1)}+q^{3}\rh_{(0,0,1,1)}\\[1 em]
&=
\begin{tikzpicture}[scale=0.5,baseline=(current bounding box.center)]
\draw[gray] (0.5,0.5) rectangle (4.5,4.5);
\foreach\x in {1,2,3}{
\draw[gray] (\x+0.5,0.5)--(\x+0.5,4.5);
\draw[gray] (0.5,\x+0.5)--(4.5,\x+0.5);
};
\foreach \x  in {1,2,3,4}{
\ebull{\x}{0.5}{0.1};
};
\rbull{0.5}{4}{0.1};
\rbull{0.5}{3}{0.1};
\cbull{0.5}{2}{0.1};
\cbull{0.5}{1}{0.1};
\ebull{1}{4.5}{0.1};
\ebull{2}{4.5}{0.1};
\rbull{3}{4.5}{0.1};
\rbull{4}{4.5}{0.1};
\ebull{4.5}{4}{0.1};
\ebull{4.5}{3}{0.1};
\cbull{4.5}{2}{0.1};
\cbull{4.5}{1}{0.1};
\draw[cyan,thick] (0.5,1)--(4.5,1);
\draw[cyan,thick] (0.5,2)--(4.5,2);
\draw[red,thick] (0.5,3)--(3.5,3)--(4,3.5)--(4,4.5);
\draw[red,thick] (0.5,4)--(2.5,4)--(3,4.5);
\end{tikzpicture} \hspace{2mm}\rh_{(0,0,1,1)}
+
\begin{tikzpicture}[scale=0.5,baseline=(current bounding box.center)]
\draw[gray] (0.5,0.5) rectangle (4.5,4.5);
\foreach\x in {1,2,3}{
\draw[gray] (\x+0.5,0.5)--(\x+0.5,4.5);
\draw[gray] (0.5,\x+0.5)--(4.5,\x+0.5);
};
\foreach \x  in {1,2,3,4}{
\ebull{\x}{0.5}{0.1};
};
\rbull{0.5}{4}{0.1};
\rbull{0.5}{3}{0.1};
\cbull{0.5}{2}{0.1};
\cbull{0.5}{1}{0.1};
\ebull{1}{4.5}{0.1};
\ebull{2}{4.5}{0.1};
\rbull{3}{4.5}{0.1};
\rbull{4}{4.5}{0.1};
\ebull{4.5}{4}{0.1};
\cbull{4.5}{3}{0.1};
\ebull{4.5}{2}{0.1};
\cbull{4.5}{1}{0.1};
\draw[cyan,thick] (0.5,1)--(4.5,1);
\draw[cyan,thick] (0.5,2)--(3.5,2)--(4.5,3);
\draw[red,thick] (0.5,3)--(3.5,3)--(4,3.5)--(4,4.5);
\draw[red,thick] (0.5,4)--(2.5,4)--(3,4.5);
\end{tikzpicture}\hspace{2mm}\rh_{(0,1,0,1)}\\[1em]
&+
\begin{tikzpicture}[scale=0.5,baseline=(current bounding box.center)]
\draw[gray] (0.5,0.5) rectangle (4.5,4.5);
\foreach\x in {1,2,3}{
\draw[gray] (\x+0.5,0.5)--(\x+0.5,4.5);
\draw[gray] (0.5,\x+0.5)--(4.5,\x+0.5);
};
\foreach \x  in {1,2,3,4}{
\ebull{\x}{0.5}{0.1};
};
\rbull{0.5}{4}{0.1};
\rbull{0.5}{3}{0.1};
\cbull{0.5}{2}{0.1};
\cbull{0.5}{1}{0.1};
\ebull{1}{4.5}{0.1};
\ebull{2}{4.5}{0.1};
\rbull{3}{4.5}{0.1};
\rbull{4}{4.5}{0.1};
\cbull{4.5}{4}{0.1};
\ebull{4.5}{3}{0.1};
\cbull{4.5}{2}{0.1};
\ebull{4.5}{1}{0.1};
\draw[cyan,thick] (0.5,1)--(3.5,1)--(4.5,2);
\draw[cyan,thick] (0.5,2)--(2.5,2)--(3,2.5)--(3,3.5)--(3.5,4)--(4.5,4);
\draw[red,thick] (0.5,3)--(3.5,3)--(4,3.5)--(4,4.5);
\draw[red,thick] (0.5,4)--(2.5,4)--(3,4.5);
\end{tikzpicture}\hspace{2mm}\rh_{(1,0,1,0)}
+
\begin{tikzpicture}[scale=0.5,baseline=(current bounding box.center)]
\draw[gray] (0.5,0.5) rectangle (4.5,4.5);
\foreach\x in {1,2,3}{
\draw[gray] (\x+0.5,0.5)--(\x+0.5,4.5);
\draw[gray] (0.5,\x+0.5)--(4.5,\x+0.5);
};
\foreach \x  in {1,2,3,4}{
\ebull{\x}{0.5}{0.1};
};
\rbull{0.5}{4}{0.1};
\rbull{0.5}{3}{0.1};
\cbull{0.5}{2}{0.1};
\cbull{0.5}{1}{0.1};
\ebull{1}{4.5}{0.1};
\ebull{2}{4.5}{0.1};
\rbull{3}{4.5}{0.1};
\rbull{4}{4.5}{0.1};
\cbull{4.5}{4}{0.1};
\cbull{4.5}{3}{0.1};
\ebull{4.5}{2}{0.1};
\ebull{4.5}{1}{0.1};
\draw[cyan,thick] (0.5,1)--(3.5,1)--(4,1.5)--(4,2.5)--(4.5,3);
\draw[cyan,thick] (0.5,2)--(2.5,2)--(3,2.5)--(3,3.5)--(3.5,4)--(4.5,4);
\draw[red,thick] (0.5,3)--(3.5,3)--(4,3.5)--(4,4.5);
\draw[red,thick] (0.5,4)--(2.5,4)--(3,4.5);
\end{tikzpicture} \hspace{2mm}\rh_{(1,1,0,0)}
\end{align*}
\end{ex}
 
\subsection{Higher spin puzzles}
To give a formula for the structure constants for spin Hall--Littlewood functions, we need to consider a slightly modified version of the puzzles presented in the previous subsection. Let us begin with the description of the tiles necessary to construct these puzzles. Every edge of a tile is decorated with an ordered pair of non-negative integers.
  
\begin{equation}
\begin{tikzpicture}[scale=1,baseline=-2pt]
 \draw[gray] (0,0) rectangle (1,1); 
\draw[red,->,rounded corners] (0,0.7)--(0.3,0.7)--(0.3,1);
\draw[red,->,rounded corners] (0,0.6)--(0.4,0.6)--(0.4,1);
\draw[red,->,rounded corners] (0.5,0)--(0.5,1);
\draw[red,->,rounded corners] (0.6,0)--(0.6,0.4)--(1,0.4);
\draw[cyan,<-] (0,0.5)--(1,0.5);
\draw[cyan,<-,rounded corners] (0,0.4)--(0.4,0.4)--(0.4,0);
\draw[cyan,<-,rounded corners] (0.6,1)--(0.6,0.6)--(1,0.6);
\node at (0.5,-0.3) {$\ss (a_1,a_2)$};
\node at (-0.7,0.5) {$\ss (b_1,b_2)$};
\node at (0.5,1.3) {$\ss (c_1,c_2)$};
\node at (1.7,0.5) {$\ss (d_1,d_2)$};
\end{tikzpicture}
\end{equation}

The weight of these tiles is zero whenever $a_1+b_1 \neq c_1+d_1$ and/or $a_2+d_2 \neq b_2+c_2$. In other words, the first particles (red particles) enter from the left and bottom, exiting through the top and the right edge, while the second particles (blue particles) enter from the bottom and the right, exiting through the left and the top edge. Non-zero weights for these tiles are specified within the theorem.

\begin{dfn}
 For three vectors of non negative integers $\bll=(l_1,\dots,l_{N})$, $\bm=(m_1,\dots,m_{P})$, and $\bk=(k_1,\dots,k_{P})$ such that $\sum_{i}l_{i}=\sum_{i}m_{i}=\sum_{i}k_{i}=n$, the \emph{higher spin puzzles} are a tiling of a rectangular region with the above mentioned tiles with the boundary conditions shown below: 
\begin{align}
\label{puzzles:latticemodel-general}
 \begin{tikzpicture}[scale=0.6,baseline=(current bounding box.center)]
 \draw (0,0)--++(0:5)--++(90:6)--++(180:5)--++(-90:6);
  \foreach\x in {1,2,3,4,5}{
 \draw[gray] (0:\x)--++(90:6);
  \draw[gray] (90:\x)--++(0:5);
 };
 \path (0:0)++(180:1)++(90:0.5) node {$\ss \textcolor{cyan} {m_{P}} + \textcolor{red} {m_{P}}$};
 \path (0:0)++(180:0.7)++(90:2) node[rotate=90] {$\dots$};
 \path (0:0)++(180:1)++(90:3.5) node {$\ss \textcolor{cyan} {m_{3}} + \textcolor{red} {m_{3}}$};
 \path (0:0)++(180:1)++(90:4.5) node {$\ss \textcolor{cyan} {m_{2}} + \textcolor{red} {m_{2}}$};
\path (0:0)++(180:1)++(90:5.5) node {$\ss \textcolor{cyan} {m_{1}} + \textcolor{red} {m_{1}}$};
 \path (0:0)++(0:5.5)++(90:0.5) node[cyan] {$\ss k_{P}$};
 \path (0:0)++(0:5.5)++(90:2) node[rotate=90] {$\dots$};
 \path (0:0)++(0:5.5)++(90:3.5) node[cyan] {$\ss k_{3}$};
 \path (0:0)++(0:5.5)++(90:4.5) node[cyan] {$\ss k_{2}$};
  \path (0:0)++(0:5.5)++(90:5.5) node[cyan] {$\ss k_1$};
 \draw[fill=white] (0:0)++(0:0.5) circle (0.1);
\draw[fill=white] (0:0)++(0:1.5) circle (0.1);
\draw[fill=white] (0:0)++(0:2.5) circle (0.1);
\draw[fill=white] (0:0)++(0:3.5) circle (0.1);
\draw[fill=white] (0:0)++(0:4.5) circle (0.1);
  \path (0:0)++(90:6.5)++(0:0.6) node[red] {$\ss l_{N}$};
  \path (0:0)++(90:6.5)++(0:2) node[red] {$\dots$};
  \path (0:0)++(90:6.5)++(0:3.6) node[red] {$\ss l_2$};
  \path (0:0)++(90:6.5)++(0:4.6) node[red] {$\ss l_1$};
 \end{tikzpicture}
\end{align}
 where for a non-negative integer $k$, we write  $\textcolor{red}{k}$ and  $\textcolor{cyan}{k}$ to denote $(k,0)$ and $(0,k)$.
\end{dfn}

Graphically, we can interpret the boundary conditions of the puzzle as follows: the blue particles enter from the right and exit through the left boundary while moving upwards and to the left, while the red particles enter from the left and exit through the top while moving upwards and to the right. 

\begin{thm}
\label{theorem:Puzzles-Spin-Hall} For a vector $\bll=(l_1,\dots,l_{N})$ of non negative integers we write $|\bll|=\sum_{i}il_{i}$. For three vectors of non negative integers $\bll=(l_1,\dots,l_{N})$, $\bm=(m_1,\dots,m_{P})$, and $\bk=(k_1,\dots,k_{P})$ such that $\sum_{i}l_{i}=\sum_{i}m_{i}=\sum_{i}k_{i}=n$,
we have the following product rule for spin Hall--Littlewood functions:
\begin{equation}
\label{eq:spin-theorem}
\mathrm{F}_{\bm}(x_1,\dots,x_n;s) \hspace{3mm}\mathrm{F}_{\mathbf{l}}(x_1,\dots,x_{n};s)=\sum_{\bk}
\mathcal{C}^{\bk}_{\bll,\bm}(q,s)\hspace{2mm}
\mathrm{F}_{\mathbf{k}}(x_1,\dots,x_{n};s)  
\end{equation}
where 
\begin{align}
\label{puzzles:spin-theorem}
\mathcal{C}^{\bk}_{\bll,\bm}(q,s) =s^{|\bll|+|\bm|-|\bk|}(-1)^{|\bm|-|\bk|}
 \left(
 \begin{tikzpicture}[scale=0.6,baseline=(current bounding box.center)]
 \draw (0,0)--++(0:5)--++(90:6)--++(180:5)--++(-90:6);
  \foreach\x in {1,2,3,4,5}{
 \draw[gray] (0:\x)--++(90:6);
  \draw[gray] (90:\x)--++(0:5);
 };
 \path (0:0)++(180:1)++(90:0.5) node {$\ss \textcolor{cyan} {m_{P}} + \textcolor{red} {m_{P}}$};
 \path (0:0)++(180:0.7)++(90:2) node[rotate=90] {$\dots$};
 \path (0:0)++(180:1)++(90:3.5) node {$\ss \textcolor{cyan} {m_{3}} + \textcolor{red} {m_{3}}$};
 \path (0:0)++(180:1)++(90:4.5) node {$\ss \textcolor{cyan} {m_{2}} + \textcolor{red} {m_{2}}$};
\path (0:0)++(180:1)++(90:5.5) node {$\ss \textcolor{cyan} {m_{1}} + \textcolor{red} {m_{1}}$};
 \path (0:0)++(0:5.5)++(90:0.5) node[cyan] {$\ss k_{P}$};
 \path (0:0)++(0:5.5)++(90:2) node[rotate=90] {$\dots$};
 \path (0:0)++(0:5.5)++(90:3.5) node[cyan] {$\ss k_{3}$};
 \path (0:0)++(0:5.5)++(90:4.5) node[cyan] {$\ss k_{2}$};
  \path (0:0)++(0:5.5)++(90:5.5) node[cyan] {$\ss k_1$};
 \draw[fill=white] (0:0)++(0:0.5) circle (0.1);
\draw[fill=white] (0:0)++(0:1.5) circle (0.1);
\draw[fill=white] (0:0)++(0:2.5) circle (0.1);
\draw[fill=white] (0:0)++(0:3.5) circle (0.1);
\draw[fill=white] (0:0)++(0:4.5) circle (0.1);
  \path (0:0)++(90:6.5)++(0:0.6) node[red] {$\ss l_{N}$};
  \path (0:0)++(90:6.5)++(0:2) node[red] {$\dots$};
  \path (0:0)++(90:6.5)++(0:3.6) node[red] {$\ss l_2$};
  \path (0:0)++(90:6.5)++(0:4.6) node[red] {$\ss l_1$};
 \end{tikzpicture}
 \right) \prod_{i\geq 1}\dfrac{(q;q)_{l_{i}}}{(q;q)_{k_{i}}}\dfrac{1}{(s^{2};q)_{l_{i}}}
\end{align}
with weights:
\begin{multline}
\label{weights:finalweights_puzzle}
\begin{tikzpicture}[scale=1,baseline=(current bounding box.center)]
 \draw (0,0) rectangle (1,1); 
\node at (0.5,-0.3) {$\ss (a_1 ,a_2)$};
\node at (-0.5,0.5) {$\ss (b_1 ,b_2)$};
\node at (0.5,1.3) {$\ss (c_1 , c_2)$};
\node at (1.5,0.5) {$\ss (d_1 , d_2)$};
\end{tikzpicture}= \mathbf{1}_{a_1+b_1=c_1+d_1}\hspace{2mm} \mathbf{1}_{a_2+d_2=b_2+c_2} \hspace{2mm}\mathbf{1}_{b_1 \leq b_2}\hspace{2mm} \mathbf{1}_{d_1 \leq d_2}\\[1em]
\sum_{p_{1},p_{2}}
(-1)^{p_{1}}\hspace{1mm}q^{(d_{1}-d_{2})(c_{2}-p_{2})+\binom{b_{2}-b_{1}+p_{1}}{2}}\hspace{1mm}\binom{\text{\small $c_1 +d_1 -p_1$}}{\text{\small{$c_1 -p_1$}}}_{q} \hspace{1mm} \binom{\text{\small $b_1$}}{\text{\small $p_1$}}_{q}\\[1em]
\displaystyle
\dfrac{ \displaystyle \prod^{p_{1}+p_{2}}_{i=1} (q-s^{\text{\tiny $2$}}q^{i-1})}{(s^{\text{\tiny $2$}};q)_{p_{1}+p_{2}+b_{2}-b_{1}}}
\hspace{1mm}
\dfrac{ \displaystyle \prod^{c_{2}-p_{2}}_{i=1} (s^{\text{\tiny $2$}}q^{d_{2}}-q^{i})}{(q;q)_{c_{2}-p_{2}}}
\hspace{1mm}
\dfrac{ \displaystyle (q;q)_{c_{1}+c_{2}-p_{1}-p_{2}}}{(q;q)_{c_{1}+c_{2}-p_{1}-p_{2}-d_{2}+d_{1}}}
\hspace{1mm}
\dfrac{ \displaystyle (s^{\text{\tiny $2$}} q^{b_{2}};q)_{p_{2}}}{(q;q)_{p_{2}}}
\end{multline}
\end{thm}

\

\begin{ex}
Fix $P = 4$ and $N = 3$, given $\bll = (1,2,0)$, $\bm = (2,1,0,0)$, and $\bk = (0,1,2,0)$, there are three puzzles with non-trivial weights. We list the puzzles with their weights given below.

\

\begin{center}
\begin{tabular}{c@{\hskip 1cm}c@{\hskip 1cm}c}
\begin{tikzpicture}[scale=0.8,baseline=(current bounding box.center)]
\draw[lightgray](0,-1) rectangle (3,3);  
\foreach \x in {0,1,2,3}{
\draw[lightgray] (0,\x)--(3,\x);
};
\foreach \x in {1,2,3}{
\draw[lightgray] (\x,-1)--(\x,3);
};
\draw[red,thick,-stealth,rounded corners] (0,2.6)--(1,2.6)--(1.5,3);
\draw[red,thick,-stealth,rounded corners] (0,2.7)--(1,2.7)--(1.35,3);
\draw[red,thick,-stealth,rounded corners] (0,1.5)--(1,1.5)--(2.5,3);
\draw[cyan,thick,stealth-,rounded corners] (0,2.5)--(2,2.5)--(3,1.5);
\draw[cyan,thick,stealth-,rounded corners] (0,2.4)--(1,2.4)--(3,0.6);
\draw[cyan,thick,stealth-,rounded corners] (0,1.4)--(2,1.4)--(3,0.5);
 \end{tikzpicture}
&
 \begin{tikzpicture}[scale=0.8,baseline=(current bounding box.center)]
\draw[lightgray](0,-1) rectangle (3,3);  
\foreach \x in {0,1,2,3}{
\draw[lightgray] (0,\x)--(3,\x);
};
\foreach \x in {1,2,3}{
\draw[lightgray] (\x,-1)--(\x,3);
};
\draw[red,thick,-stealth,rounded corners] (0,2.6)--(1,2.6)--(1.5,3);
\draw[red,thick,-stealth,rounded corners] (0,2.7)--(1,2.7)--(1.35,3);
\draw[red,thick,-stealth,rounded corners] (0,1.5)--(2,1.5)--(2.5,2)--(2.5,3);
\draw[cyan,thick,stealth-,rounded corners] (0,2.5)--(1,2.5)--(2,1.6)--(3,1.6);
\draw[cyan,thick,stealth-,rounded corners] (0,2.4)--(1,2.4)--(3,0.6);
\draw[cyan,thick,stealth-,rounded corners] (0,1.4)--(2,1.4)--(3,0.5);
 \end{tikzpicture}
&
 \begin{tikzpicture}[scale=0.8,baseline=(current bounding box.center)]
\draw[lightgray](0,-1) rectangle (3,3);  
\foreach \x in {0,1,2,3}{
\draw[lightgray] (0,\x)--(3,\x);
};
\foreach \x in {1,2,3}{
\draw[lightgray] (\x,-1)--(\x,3);
};
\draw[red,thick,-stealth,rounded corners] (0,2.6)--(1,2.6)--(1.5,3);
\draw[red,thick,-stealth,rounded corners] (0,2.7)--(1,2.7)--(1.35,3);
\draw[red,thick,-stealth,rounded corners] (0,1.5)--(2,1.5)--(2.5,2)--(2.5,3);
\draw[cyan,thick,stealth-,rounded corners] (0,2.5)--(2,2.5)--(3,1.5);
\draw[cyan,thick,stealth-,rounded corners] (0,2.4)--(1,2.4)--(3,0.6);
\draw[cyan,thick,stealth-,rounded corners] (0,1.4)--(2,1.4)--(3,0.5);
 \end{tikzpicture}\\[5 em]
$\dfrac{(1-q) (1-q^{2})^2 s (1-q^{2} s^{2})}{(1-s^{2})^3 \left(1-q s^2\right)}$ & $\dfrac{(1-q)^{2} (1-q^{2}) s (1-q^{2} s^{2})}{(1-s^{2})^3 \left(1-q s^2\right)}$& $\dfrac{(1-q)(1-q^{2})^{2}s(1-q s^{2})}{(1-s^{2})^{4}}$
\end{tabular}
\end{center}

\

Then the coefficient $\mathcal{C}^{\bk}_{\bll,\bm}$ is the sum of the weights of these three puzzles.
\end{ex}

\

\begin{remark}
    The summation on the right-hand side of~\eqref{eq:spin-theorem} is finite. This finiteness arises from the vanishing of the weights of certain tiles. 

Consider a blue particle that exits the puzzle of width $N$ through the left boundary at the $K$th row (counted from the top). If this particle had entered from a row with index strictly above $K + N$, then it would necessarily have a tile of the following type:
\begin{center}
\begin{tikzpicture}[scale=0.6,baseline=-2pt]
 \draw[gray] (0,0) rectangle (1,1); 
 \draw[cyan,<-] (0.5,1)--(0.5,0);
 \draw[cyan,<-] (0.6,1)--(0.6,0);
 \node at (0.5,-0.3) {$\ss (0,k)$};
 \node at (-0.7,0.5) {$\ss (0,0)$};
 \node at (0.5,1.3) {$\ss (0,k)$};
 \node at (1.7,0.5) {$\ss (0,0)$};
\end{tikzpicture}
\end{center}
This is illustrated in the pictures below. However, this tile has weight zero. This imposes a strict upper bound on the rows from which particles can enter, and thus ensures the finiteness of the summation in~\eqref{eq:spin-theorem}.

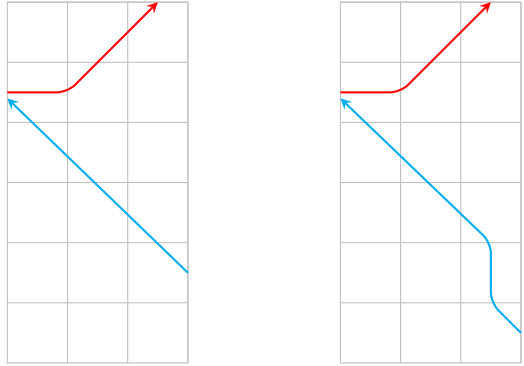
\[
 \begin{tikzpicture}[scale=0.8,baseline=(current bounding box.center)]
\draw[lightgray](0,-3) rectangle (3,3);  
\foreach \x in {-2,-1,0,1,2,3}{
\draw[lightgray] (0,\x)--(3,\x);
};
\foreach \x in {1,2,3}{
\draw[lightgray] (\x,-3)--(\x,3);
};
\draw[red,thick,-stealth,rounded corners] (0,1.5)--(1,1.5)--(2.5,3);
\draw[cyan,thick,stealth-,rounded corners] (0,1.4)--(3,-1.5);
\end{tikzpicture}
\hspace{2cm}
\begin{tikzpicture}[scale=0.8,baseline=(current bounding box.center)]
\draw[lightgray](0,-3) rectangle (3,3);  
\foreach \x in {-2,-1,0,1,2,3}{
\draw[lightgray] (0,\x)--(3,\x);
};
\foreach \x in {1,2,3}{
\draw[lightgray] (\x,-3)--(\x,3);
};
\draw[red,thick,-stealth,rounded corners] (0,1.5)--(1,1.5)--(2.5,3);
\draw[cyan,thick,stealth-,rounded corners] (0,1.4)--(2.5,-1)--(2.5,-2)--(3,-2.5);
 \end{tikzpicture}
\]
\end{remark}

\section{Six vertex Puzzles}
\label{sec:proof_6v}
To prove Theorem ~\ref{theorem:puzzlesforwavefunctions}, we return to vertex weights of Theorem \ref{theorem:weightsoftheYBE}, in the case $r=2$. We use two types of weights $W_{\tl,\tm}(x;q;\ba,\bb,\bc,\bd)$, as defined in \ref{generalweights}. The first has $\tl=\tm=1$, and the second has $\tl=1$ with 
$\tm$ generic. Graphically, tiles with $\tl=1$ and generic $\tm$ are shaded gray to distinguish them from the $\tl=\tm=1$ case. We begin by considering the Yang--Baxter equation illustrated below. We prove Theorem~\ref{theorem:puzzlesforwavefunctions} by computing the partition functions on both sides of Equation~\eqref{yangbaxterlatticemodel-sixvertex}.

\begin{multline}
\label{yangbaxterlatticemodel-sixvertex}
 \begin{tikzpicture}[scale=0.65,baseline=(current bounding box.center)]
 \draw[white,fill=lightgray] (0:0)++(-60:5)--++(0:1)--++(60:6)--++(120:4)--++(0:-1)--++(-60:4)--++(-120:6);
\draw (0:0)++(-60:1) coordinate (A)--++(60:2) coordinate (AA)--++(60:4) coordinate (B)--++(0:5) coordinate (C)--++(-60:4) coordinate (D)--++(-120:6) coordinate (E);
\draw (A)--++(-60:4) coordinate (F)--++(0:5);
\draw (F)--++(60:6) coordinate (G)--++(120:4);
\draw (G)--(D);
\draw[gray] (AA)++(60:-1)--++(-60:4)--++(0:5);
\foreach \x in {1,2,3,4}{
\draw[gray] (F)++(0:\x)--++(60:6)--++(120:4);
\draw[gray] (A)++(-60:\x)--++(60:6)--++(0:5);
\draw[gray] (A)++(60:\x)--++(-60:4)--++(0:5);
};
\foreach \x in {1,2,3,4,5}{
\draw[gray] (A)++(60:\x)--++(-60:4)--++(0:5);
};
\foreach \x in {1,2,3,4,5,6}
{
\draw[fill=cyan] (F)++(0:5)++(60:\x-0.5) circle (0.1);
};
\foreach \x in {1,2,3,4}
{\draw[fill=cyan] (A)++(-60:\x-0.5) circle (0.1);
\draw[fill=white] (F)++(0:5)++(60:6)++(120:\x-0.5) circle (0.1);};
\path (F)++(0:0.5)++(-90:0.4) node {$\ss \textcolor{red}{n}$};
\draw[fill=white] (A)++(0:-0.5)++(60:6)++(0:1) circle (0.1);
\path (F)++(-90:0.5)++(0:3) node {$\xleftarrow[w]{}$};
\path (F)++(60:6)++(0:3)++(120:4.5) node {$\xleftarrow[l]{}$};
\path (A)++(0:-0.5)++(60:3.5) node[rotate=60]{$\xleftarrow[m]{}$};
\path (0,0)++(-60:2.5)++(60:3) node {$\fontsize{20pt}{0}{\bl{A}}$};
\path (F)++(0:2.5)++(60:3) node {$\fontsize{20pt}{0}{\bl{B}}$};
\path (F)++(0:2.5)++(60:6)++(120:2.5) node {$\fontsize{20pt}{0}{\bl{C}}$};
\path (F)++(0:2.5)++(60:6)++(120:2.5) node {$\fontsize{20pt}{0}{\bl{C}}$};
\foreach \x in {1,2,3,4}{
\draw[->] (A)++(-60:\x-0.5)++(-120:1)--++(60:0.5);};
\path (A)++(-60:1-0.5)++(-120:1.5) node {$\ss (x_1,1) $};
\path (A)++(-60:2-0.5)++(-120:1.5) node {$\ss (x_2,1) $};
\path (A)++(-60:3-0.5)++(-120:1.5) node[rotate=-60] {$\ss \ldots$};
\path (A)++(-60:4-0.5)++(-120:1.5) node {$\ss (x_{n},1) $};
\path (F)++(0:0.4)++(-90:1.8) node {$\ss (z_0,\mathtt{M})$};
\draw[->] (F)++(0:0.5)++(-90:1.5)--++(90:0.5);
\draw[->] (F)++(0:1.5)++(-90:1.5)--++(90:0.5);
\draw[->] (F)++(0:2.5)++(-90:1.5)--++(90:0.5);
\draw[->] (F)++(0:3.5)++(-90:1.5)--++(90:0.5);
\draw[->] (F)++(0:4.5)++(-90:1.5)--++(90:0.5);
\path (F)++(0:1.5)++(-90:1.8)++(0:0.1) node {$\ss (z_{N},1) $};
\path (F)++(0:3.3)++(-90:1.8) node {$\ss \dots $};
\path (F)++(0:4.5)++(-90:1.8) node {$\ss (z_1,1) $};
\foreach\x in {1,2,3,4,5,6}{
\draw[->] (A)++(60:\x-0.5)++(120:1.5)--++(-60:0.5);
};
\path (A)++(60:0.5)++(120:2)++(0:-0.2) node {$\ss (q y_{P},1)$};
\path (A)++(60:1.5)++(120:2)++(0:-0.2) node {$\ss (q y_{P-1},1)$};
\path (A)++(60:2.5)++(120:2)++(90:0.5) node[rotate=60] {$ \dots$};
\path (A)++(60:4.5)++(120:2)++(0:-0.2) node {$\ss (q y_{2},1)$};
\path (A)++(60:5.5)++(120:2)++(0:-0.2) node {$\ss (q y_1,1)$};
\end{tikzpicture}
=
 \begin{tikzpicture}[scale=0.65,baseline=(current bounding box.center)]
 \draw[white,fill=lightgray] (F)--++(120:4)--++(60:6)--++(0:1)--++(-120:6)--++(-60:4)--++(0:-1);
\draw (0:0)++(-60:1) coordinate (A)--++(60:2) coordinate (AA)--++(60:4) coordinate (B)--++(0:5) coordinate (C)--++(-60:4) coordinate (D)--++(-120:4) coordinate (EE)--++(-120:2) coordinate (E);
\draw (A)--++(-60:4) coordinate (F)--++(0:5);
\draw (A)--++(0:5) coordinate (G);
\draw (G)--++(60:6);
\draw (G)--++(-60:4);
\path (F)++(0:0.5)++(-90:0.4) node {$\ss \textcolor{red}{n}$};
\foreach \x in {1,2,3,4}{
\draw[gray] (A)++(-60:\x)--++(0:5)--++(60:6);
};
\foreach \x in {1,2,3,4,5}{
\draw[gray] (A)++(-60:4)++(0:\x)--++(120:4)--++(60:6);
};
\foreach \x in {1,2,3,4,5}{
\draw[gray] (A)++(60:\x)--++(0:5)--++(-60:4);
};
\foreach \x in {1,2,3,4,5,6}
{
\draw[fill=cyan] (F)++(0:5)++(60:\x-0.5) circle (0.1);
};
\foreach \x in {1,2,3,4}
{\draw[fill=cyan] (A)++(-60:\x-0.5) circle (0.1);
\draw[fill=white] (F)++(0:5)++(60:6)++(120:\x-0.5) circle (0.1);};
\draw[fill=white] (A)++(0:-0.5)++(60:6)++(0:1) circle (0.1);
\path (F)++(-90:0.5)++(0:3) node {$\xleftarrow[w]{}$};
\path (F)++(60:6)++(0:3)++(120:4.5) node {$\xleftarrow[l]{}$};
\path (A)++(0:-0.5)++(60:3.5) node[rotate=60]{$\xleftarrow[m]{}$};
\path (0,0)++(0:3)++(60:3) node {$\fontsize{20pt}{0}{\bl{E}}$};
\path (F)++(0:2.5)++(120:2.5) node {$\fontsize{20pt}{0}{\bl{D}}$};
\path (0,0)++(0:5)++(-60:2.5)++(60:3) node {$\fontsize{20pt}{0}{\bl{F}}$};
\path (F)++(0:0.3)++(-60:1.8) node {$\ss (z_0,\mathtt{M})$};
\draw[->] (F)++(0:0.5)++(-60:1.5)--++(120:0.5);
\draw[->] (F)++(0:1.5)++(-60:1.5)--++(120:0.5);
\draw[->] (F)++(0:2.5)++(-60:1.5)--++(120:0.5);
\draw[->] (F)++(0:3.5)++(-60:1.5)--++(120:0.5);
\draw[->] (F)++(0:4.5)++(-60:1.5)--++(120:0.5);
\path (F)++(0:3.3)++(-60:1.8) node {$\ss \dots $};
\path (F)++(0:4.5)++(-60:1.8) node {$\ss (z_1,1) $};
\path (F)++(0:1.5)++(-60:1.8)++(0:0.1) node {$\ss (z_{N},1) $};
\foreach \x in {1,2,3,4}{
\draw[->] (A)++(-60:\x-0.5)++(0:-0.3)++(0:-1)--++(0:0.5);};
\path (A)++(-60:1-0.5)++(0:-2) node {$\ss (x_1,1) $};
\path (A)++(-60:2-0.5)++(0:-2) node {$\ss (x_2,1) $};
\path (A)++(-60:3-0.5)++(0:-2) node {$\ss \ldots$};
\path (A)++(-60:4-0.5)++(0:-2.2) node {$\ss (x_{n},1) $};
\foreach \x in {1,2,3,4,5,6}{
\draw[->] (A)++(60:\x-0.5)++(0:-0.5)++(0:-1)--++(0:0.5);};
\path (A)++(60:1-0.5)++(0:-2.2) node {$\ss (q y_{P},1)$};
\path (A)++(60:2-0.5)++(0:-2.4) node {$\ss (q y_{P-1},1)$};
\path (A)++(60:2.5)++(0:-2)++(60:0.5) node[rotate=60] {$ \dots$};
\path (A)++(60:5-0.5)++(0:-2.4) node {$\ss (q y_2,1)$};
\path (A)++(60:6-0.5)++(0:-2.4) node {$\ss (q y_1,1)$};
\end{tikzpicture}
\end{multline}

Before proceeding, we set up the necessary notation and conventions. We represent: – $(1,0)$ as a red particle \begin{tikzpicture} \rbull{0}{0}{0.1} \end{tikzpicture}, $(0,1)$ as a blue particle \begin{tikzpicture} \cbull{0}{0}{0.1} \end{tikzpicture} ,
– and $(0,0)$ as a white particle (also referred to as a hole) \begin{tikzpicture} \ebull{0}{0}{0.1} \end{tikzpicture}. We also use \textcolor{red}{$n$} to denote the vector $(n,0)$. We denote the partition function of a region by 
$\mathcal{Z}$, using the letter of the region as a subscript.

The binary string $m$ has length $P$, with 
$P-n$ blue particles and 
$n$ white particles. Consequently, only blue particles can enter region $\bl{A}$. At the bottom boundary of region $\bl{B}$, the leftmost edge is occupied by $(n,0)$, which may be interpreted as a reservoir of red particles. The strings 
$w$ and $l$ are binary strings consisting of red particles and holes, with $l$ containing exactly $n$ more red particles than 
$w$.

With this notation in place, we now compute the partition functions of each region.

\subsection{Region \texorpdfstring{$A$}{A}}
We begin by analysing the boundary conditions of $\bl{A}$. Only blue particles enter $\bl{A}$, and they can either proceed into $\bl{B}$ or $\bl{C}$. However, once a blue particle enters 
$\bl{C}$, it cannot exit, since the top and right boundaries allow no blue particles to exit. Therefore, all the blue particles enter $\bl{B}$. Consequently, the weight of $\bl{A}$ is the partition function of the lattice below:
\begin{equation}
\mathcal{Z}_{\bl{A}}(x_{1},\dots,x_{n};y_{1},\dots,y_{P}):=
\begin{tikzpicture}[scale=0.6,baseline=(current bounding box.center)]
\draw (0:0)--++(180:4);
\draw (0:0)--++(90:2);
\draw (0:0)++(180:4)--++(90:6)--++(0:4)--++(-90:4);
\foreach\x in {1,2,3,4,5,6}{
 \draw[fill=cyan] (0:0)++(90:\x-0.5) circle (0.1);};
 \draw[gray] (0:0)++(90:1)--++(180:4);
 \draw[gray] (0:0)++(90:2)--++(180:4);
\foreach\x in {1,2,3,4}{
 \draw[gray] (0:0)++(180:\x)--++(90:2);
 \draw[gray] (0:0)++(90:2)++(90:\x)--++(180:4);
\draw[gray] (0:0)++(90:2)++(180:\x)--++(90:4);
 }l
 \foreach\x in {1,2,3,4}{
 \draw[fill=cyan] (0:0)++(180:\x-0.5) circle (0.1);
 \draw[fill=white] (0:0)++(90:6)++(180:\x-0.5) circle (0.1);
 };
 \path (0:0)++(180:4.5)++(90:2.5) node[rotate=90] {$\xleftarrow[m]{}$};
 \path (0:0)++(180:6.5)++(90:0.5) node {$\ss (qy_{P},1)$};
 \path (0:0)++(180:6.5)++(90:2) node[rotate=90] {$\dots$};
 \path (0:0)++(180:6.5)++(90:3.5) node {$\ss (qy_{3},2)$};
 \path (0:0)++(180:6.5)++(90:4.5) node {$\ss (qy_{2},1)$};
  \path (0:0)++(180:6.5)++(90:5.5) node {$\ss (qy_{1},1)$};
  \foreach \x in {1,2,3,4,5,6}{
 \draw[->] (0:0)++(180:5.5)++(90:\x-0.5)--++(0:0.5);
 };
 \foreach \x in {1,2,3,4}{
 \draw[->] (0:0)++(180:4)++(-90:1.5)++(0:\x-0.5)--++(90:0.5);
 };
  \path (0:0)++(180:4)++(-90:2)++(0:0.5) node {$\ss (x_{1},1)$};
  \path (0:0)++(180:4)++(-90:2)++(0:2) node {$\dots$};
  \path (0:0)++(180:4)++(-90:2)++(0:3.5) node {$\ss (x_{n},1)$};
\end{tikzpicture}
\end{equation}
with the tiles given below:
\begin{equation}
\label{weights:6v-regionA}
\begin{tabular}{cccccc}
\begin{tikzpicture}[baseline={(0,0)},scale=0.8]
\draw (0:0)--++(180:1)--++(90:1)--++(0:1)--++(-90:1); 
\draw[fill=white] (0,0)++(180:1)++(90:0.5) circle (0.1);
\draw[fill=white] (0,0)++(180:0.5) circle (0.1);
\draw[fill=white] (0,0)++(90:0.5) circle (0.1);
\draw[fill=white] (0,0)++(90:1)++(180:0.5) circle (0.1);
\end{tikzpicture} 
&
\qquad
\begin{tikzpicture}[baseline={(0,0)},scale=0.8]
\draw (0:0)--++(180:1)--++(90:1)--++(0:1)--++(-90:1); 
\draw[fill=cyan] (0,0)++(180:1)++(90:0.5) circle (0.1);
\draw[fill=cyan] (0,0)++(180:0.5) circle (0.1);
\draw[fill=cyan] (0,0)++(90:0.5) circle (0.1);
\draw[fill=cyan] (0,0)++(90:1)++(180:0.5) circle (0.1);
\draw[cyan,thick] (0:0)++(180:0.5)--++(45:0.6);
\draw[cyan,thick] (0:0)++(180:1)++(90:0.5)--++(45:0.6);
\end{tikzpicture}
&
\quad
\begin{tikzpicture}[baseline={(0,0)},scale=0.8]
\draw (0:0)--++(180:1)--++(90:1)--++(0:1)--++(-90:1); 
\draw[fill=cyan] (0,0)++(180:1)++(90:0.5) circle (0.1);
\draw[fill=white] (0,0)++(180:0.5) circle (0.1);
\draw[fill=cyan] (0,0)++(90:0.5) circle (0.1);
\draw[fill=white] (0,0)++(90:1)++(180:0.5) circle (0.1);
\draw[cyan,thick] (0:0)++(90:0.5)--++(180:1);
\end{tikzpicture}
&
\begin{tikzpicture}[baseline={(0,0)},scale=0.8]
\draw (0:0)--++(180:1)--++(90:1)--++(0:1)--++(-90:1); 
\draw[fill=white] (0,0)++(180:1)++(90:0.5) circle (0.1);
\draw[fill=cyan] (0,0)++(180:0.5) circle (0.1);
\draw[fill=white] (0,0)++(90:0.5) circle (0.1);
\draw[fill=cyan] (0,0)++(90:1)++(180:0.5) circle (0.1);
\draw[cyan,thick] (0:0)++(180:0.5)--++(90:1);
\end{tikzpicture}
&
\begin{tikzpicture}[baseline={(0,0)},scale=0.8]
\draw (0:0)--++(180:1)--++(90:1)--++(0:1)--++(-90:1); 
\draw[fill=cyan] (0,0)++(180:1)++(90:0.5) circle (0.1);
\draw[fill=white] (0,0)++(180:0.5) circle (0.1);
\draw[fill=white] (0,0)++(90:0.5) circle (0.1);
\draw[fill=cyan] (0,0)++(90:1)++(180:0.5) circle (0.1);
\draw[cyan,thick] (0:0)++(180:1)++(90:0.5)--++(45:0.6);
\end{tikzpicture}
&
\begin{tikzpicture}[baseline={(0,0)},scale=0.8]
\draw (0:0)--++(180:1)--++(90:1)--++(0:1)--++(-90:1); 
\draw[fill=white] (0,0)++(180:1)++(90:0.5) circle (0.1);
\draw[fill=cyan] (0,0)++(180:0.5) circle (0.1);
\draw[fill=cyan] (0,0)++(90:0.5) circle (0.1);
\draw[fill=white] (0,0)++(90:1)++(180:0.5) circle (0.1);
\draw[cyan,thick] (0:0)++(180:0.5)--++(45:0.6);
\end{tikzpicture}\\ [2 em]
$1$ & \qquad $1$ & \qquad $\dfrac{1-x y^{-1}q^{-1}}{1-x y^{-1}}$ & \quad ${\dfrac{q\left(1-{x y^{-1}q^{-1}}\right)}{1-xy^{-1}}}$ &\quad $ \dfrac{q^{-1}(1-q)xy^{-1}}{1-xy^{-1}}$ & \quad $\dfrac{(1-q)}{1-x y^{-1}}$ 
\end{tabular}
\end{equation}

\begin{lemma}
The following equality holds:
\begin{equation}
 \label{partitionfucntion: regionA}   
 \mathcal{Z}_{\bl{A}}(x_{1},\dots,x_{n};y_{1},\dots,y_{P})|_{y_{i}=q^{-1}}= \mathrm{H}_{\overline{m}}(x_1,\dots,x_n)
\end{equation}
where $\overline{m}=(1-m_{1},1-m_{2},\dots,1-m_{P})$.
\end{lemma}
\begin{proof}
We redraw the vertices and the lattice of region $\bl{A}$ by complementing the particles on all edges and also substituting $y_{i}=q^{-1}$. We then have:
\begin{equation}
\mathcal{Z}_{\bl{A}}=
\begin{tikzpicture}[scale=0.6,rotate=180,baseline=(current bounding box.center)]
\draw (0:0)--++(180:4);
\draw[gray] (0:0)++(90:1)--++(180:4);
\draw[gray] (0:0)++(90:2)--++(180:4);
\draw (0:0)--++(180:4)--++(90:6)--++(0:4)--++(-90:6);
\foreach\x in {1,2,3,4}{
 \draw[gray] (0:0)++(90:2)++(90:\x)--++(180:4);
\draw[gray] (0:0)++(90:2)++(180:\x)--++(90:4);
 };
 \foreach\x in {1,2,3,4}{
\draw[gray] (0:0)++(180:\x)--++(90:6);
 };
\foreach\x in {1,2,3,4,5,6}{
 \draw[gray] (0:0)++(90:\x)--++(180:4);
 \draw[fill=white] (0:0)++(180:4)++(90:\x-0.5) circle (0.1);
 };
 \foreach\x in {1,2,3,4}{
 \draw[fill=blue] (0:0)++(180:\x-0.5) circle (0.1);
 \draw[fill=white] (0:0)++(90:6)++(180:\x-0.5) circle (0.1);
 };
 \path (0:0)++(0:0.5)++(90:2.5) node[rotate=90]{$\xleftarrow[\overline{m}]{}$};
\end{tikzpicture}
\end{equation}
with the following weights:
\begin{equation}
\label{weights:6v-regionA-complemented}
\begin{tabular}{cccccc}
\begin{tikzpicture}[baseline={(0,0)},scale=0.8]
\draw (0:0)--++(180:1)--++(90:1)--++(0:1)--++(-90:1); 
\draw[fill=blue] (0,0)++(180:1)++(90:0.5) circle (0.1);
\draw[fill=blue] (0,0)++(180:0.5) circle (0.1);
\draw[fill=blue] (0,0)++(90:0.5) circle (0.1);
\draw[fill=blue] (0,0)++(90:1)++(180:0.5) circle (0.1);
\draw[blue,thick] (0:0)++(180:0.5)--++(45:0.6);
\draw[blue,thick] (0:0)++(180:1)++(90:0.5)--++(45:0.6);
\end{tikzpicture}
&\qquad
\begin{tikzpicture}[baseline={(0,0)},scale=0.8]
\draw (0:0)--++(180:1)--++(90:1)--++(0:1)--++(-90:1); 
\draw[fill=white] (0,0)++(180:1)++(90:0.5) circle (0.1);
\draw[fill=white] (0,0)++(180:0.5) circle (0.1);
\draw[fill=white] (0,0)++(90:0.5) circle (0.1);
\draw[fill=white] (0,0)++(90:1)++(180:0.5) circle (0.1);
\end{tikzpicture}
&
\begin{tikzpicture}[baseline={(0,0)},scale=0.8]
\draw (0:0)--++(180:1)--++(90:1)--++(0:1)--++(-90:1); 
\draw[fill=white] (0,0)++(180:1)++(90:0.5) circle (0.1);
\draw[fill=blue] (0,0)++(180:0.5) circle (0.1);
\draw[fill=white] (0,0)++(90:0.5) circle (0.1);
\draw[fill=blue] (0,0)++(90:1)++(180:0.5) circle (0.1);
\draw[blue,thick] (0:0)++(180:0.5)--++(90:1);
\end{tikzpicture}
&
\begin{tikzpicture}[baseline={(0,0)},scale=0.8]
\draw (0:0)--++(180:1)--++(90:1)--++(0:1)--++(-90:1); 
\draw[fill=blue] (0,0)++(180:1)++(90:0.5) circle (0.1);
\draw[fill=white] (0,0)++(180:0.5) circle (0.1);
\draw[fill=blue] (0,0)++(90:0.5) circle (0.1);
\draw[fill=white] (0,0)++(90:1)++(180:0.5) circle (0.1);
\draw[blue,thick] (0:0)++(90:0.5)--++(180:1);
\end{tikzpicture}
&
\begin{tikzpicture}[baseline={(0,0)},scale=0.8]
\draw (0:0)--++(180:1)--++(90:1)--++(0:1)--++(-90:1); 
\draw[fill=white] (0,0)++(180:1)++(90:0.5) circle (0.1);
\draw[fill=blue] (0,0)++(180:0.5) circle (0.1);
\draw[fill=blue] (0,0)++(90:0.5) circle (0.1);
\draw[fill=white] (0,0)++(90:1)++(180:0.5) circle (0.1);
\draw[blue,thick] (0:0)++(180:0.5)--++(45:0.6);
\end{tikzpicture}
&
\begin{tikzpicture}[baseline={(0,0)},scale=0.8]
\draw (0:0)--++(180:1)--++(90:1)--++(0:1)--++(-90:1); 
\draw[fill=blue] (0,0)++(180:1)++(90:0.5) circle (0.1);
\draw[fill=white] (0,0)++(180:0.5) circle (0.1);
\draw[fill=white] (0,0)++(90:0.5) circle (0.1);
\draw[fill=blue] (0,0)++(90:1)++(180:0.5) circle (0.1);
\draw[blue,thick] (0:0)++(180:1)++(90:0.5)--++(45:0.6);
\end{tikzpicture}\\[2 em]
$1$ & \qquad $1$ & \qquad $\dfrac{1-x}{1-q x }$ & \quad ${\dfrac{q\left(1-{x}\right)}{1-q x}}$ &\quad $ \dfrac{(1-q)x}{1-q x}$ & \quad $\dfrac{(1-q)}{1-qx}$ 
\end{tabular}
\end{equation}

From ~\cref{def:6v-wavefunction}, we conclude that
\begin{equation}
\cz_{\bl{A}} (x_{1},\dots,x_{n};q^{-1},\dots,q^{-1})= \rh_{\overline{m}}(x_1,\dots,x_n).
\end{equation}
\end{proof}

\subsection{Region \texorpdfstring{$B$}{B}}

Since we are interested in the structure constants, we divide the right-hand side of~\ref{yangbaxterlatticemodel-sixvertex} by the partition function $\mathcal{Z}_{\bl{B}}$. This amounts to normalising the weights in region $\bl{B}$ so that the vertex
\begin{tikzpicture}[scale=0.5,baseline=(current bounding box.center)] \draw[fill=lightgray] (0:0)--++(0:1)--++(90:1)--++(180:1)--++(-90:1); \draw ++(90:0.5)--++(0:1); \draw[fill=cyan] (0:0)++(90:0.5) circle (0.1); \draw[fill=cyan] (0:0)++(0:1)++(90:0.5) circle (0.1); \end{tikzpicture} has weight 1. (This value is independent of the top and bottom labels.)

A blue particle enters region $\bl{B}$ in every row. Consider the blue particle entering from the bottom row: since no other blue particles enter from below, it must exit on the same row. This forces the entire bottom row to be fixed. Repeating this argument row by row, we find that the entire region 
$\bl{B}$ is frozen.

Therefore, we have:
\begin{equation}
\label{partitionfunction:regionB}
  \mathcal{Z}_{\bl{B}}=  1.
\end{equation}

 
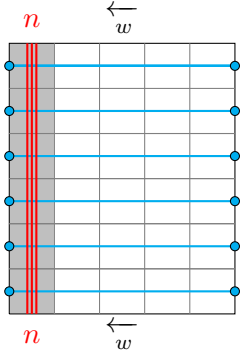
\begin{figure}[h]
\begin{tikzpicture}[scale=0.6]
\draw[white,fill=lightgray] (0:0)--++(0:1)--++(90:6)--++(180:1)--++(-90:6);
\draw (0:0)--++(90:2) coordinate (A)--++(90:4)--++(0:5)--++(-90:6)--++(180:5);
\foreach\x in {1,2,3,4}{
\draw[gray] (0:0)++(0:\x)--++(90:6);
\draw[gray] (0:0)++(90:\x)--++(0:5);
};
\foreach\x in {1,2,3,4,5}{
\draw[gray] (0:0)++(90:\x)--++(0:5);
\draw[cyan,thick] (0:0)++(90:\x-0.5)--++(0:5);
\draw[cyan,thick] (0:0)++(90:5.5)--++(0:5);
};
\foreach\x in {1,2,3,4,5,6}{
\draw[fill=cyan] (0,0)++(90:\x-0.5) circle (0.1);
\draw[fill=cyan] (0,0)++(0:5)++(90:\x-0.5) circle (0.1);};
\draw[red,thick] (0:0)++(0:0.4)--++(90:6);
\draw[red,thick] (0:0)++(0:0.5)--++(90:6);
\draw[red,thick] (0:0)++(0:0.6)--++(90:6);
\path (0:0)++(0:2.5)++(90:6.5) node {$\xleftarrow[w]{}$};
\path (0:0)++(0:2.5)++(-90:0.5) node {$\xleftarrow[w]{}$};
\path (0:0)++(0:0.5)++(-90:0.5) node {$\textcolor{red}{n}$};
\path (0:0)++(0:0.5)++(90:6.5) node {$\textcolor{red}{n}$};
\end{tikzpicture}
\end{figure}

\subsection{Region \texorpdfstring{$C$}{C}}
As in region $\bl{B}$, we normalise the weights in regions 
$\bl{C}$ and 
$\bl{D}$ so that the vertex
\begin{tikzpicture}[scale=0.7,baseline= (current bounding box.center)] 
\draw (0:0)--++(180:1)--++(90:1)--++(0:1)--++(-90:1);
\draw[fill=cyan] (0,0)++(180:1)++(90:0.5) circle (0.1);
\draw[fill=white] (0,0)++(180:0.5) circle (0.1);
\draw[fill=cyan] (0,0)++(90:0.5) circle (0.1); 
\draw[fill=white] (0,0)++(90:1)++(180:0.5) circle (0.1); \draw[cyan,thick] (0:0)++(90:0.5)--++(180:1); \end{tikzpicture} has weight $1$.

From our earlier analysis of regions 
$\bl{A}$ and 
$\bl{B}$, it follows that all the boundaries of 
$\bl{C}$ are fixed. Specifically, 
$n$ red particles enter from the bottom of the first column, and none exit from its top edge. Since there are $n$ rows, each red particle must turn right in every row, resulting in the completely frozen first column as depicted below.

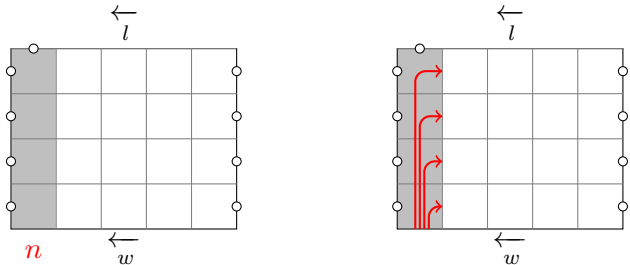
\begin{equation*}
\label{partitionfunction: regionC}
\begin{tikzpicture}[scale=0.6,baseline=(current bounding box.center)]
\draw[white,fill=lightgray] (0:0)--++(0:1)--++(90:4)--++(180:1)--++(-90:4);
\draw (0:0) coordinate (A)--++(0:5) coordinate (B)--++(90:4) coordinate (C)--++(180:5) coordinate (D)--cycle;
\foreach \x in {1,2,3,4}{
\draw[gray] (0:0)++(0:\x)--++(90:4);
\draw[gray] (0:0)++(90:\x)--++(0:5);
};
\foreach\x in {1,2,3,4}{
\draw[fill=white] (0,0)++(90:\x-0.5) circle (0.1);
\draw[fill=white] (0,0)++(0:5)++(90:\x-0.5) circle (0.1);};
\draw[fill=white] (0,0)++(0:0.5)++(90:4) circle (0.1);
\path (0:0)++(-90:0.5)++(0:2.5) node {$\xleftarrow[w]{}$};
\path (0:0)++(90:4.5)++(0:2.5) node {$\xleftarrow[l]{}$};
\path (0:0)++(-90:0.5)++(0:0.5) node {$\textcolor{red}{n}$};
\end{tikzpicture}
\hspace{2cm}
\begin{tikzpicture}[scale=0.6,baseline=(current bounding box.center)]
\draw[white,fill=lightgray] (0:0)--++(0:1)--++(90:4)--++(180:1)--++(-90:4);
\draw (0:0) coordinate (A)--++(0:5) coordinate (B)--++(90:4) coordinate (C)--++(180:5) coordinate (D)--cycle;
\foreach \x in {1,2,3,4}{
\draw[gray] (0:0)++(0:\x)--++(90:4);
\draw[gray] (0:0)++(90:\x)--++(0:5);
};
\draw[->,red,rounded corners,thick] (A)++(0:0.4)--++(90:3.5)--++(0:0.6);
\draw[->,red,rounded corners,thick] (A)++(0:0.5)--++(90:2.5)--++(0:0.5);
\draw[->,red,rounded corners,thick] (A)++(0:0.6)--++(90:1.5)--++(0:0.4);
\draw[->,red,rounded corners,thick] (A)++(0:0.7)--++(90:0.5)--++(0:0.3);
\foreach\x in {1,2,3,4}{
\draw[fill=white] (0,0)++(90:\x-0.5) circle (0.1);
\draw[fill=white] (0,0)++(0:5)++(90:\x-0.5) circle (0.1);};
\draw[fill=white] (0,0)++(0:0.5)++(90:4) circle (0.1);
\path (0:0)++(-90:0.5)++(0:2.5) node {$\xleftarrow[w]{}$};
\path (0:0)++(90:4.5)++(0:2.5) node {$\xleftarrow[l]{}$};
\end{tikzpicture}
\end{equation*}

Given that the weight of the vertex
\begin{center}
\begin{tikzpicture}[scale=1,baseline=(current bounding box.center)]
 \draw[gray] (0,0) rectangle (1,1); 
\node at (0.5,-0.3) {$\ss m$};
\node at (-0.3,0.5) {$\ss 0$};
\node at (0.5,1.3) {$\ss m-1$};
\node at (1.3,0.5) {$\ss 1$};
\draw[->] (0.5,-1) node[below] {$\ss z_{0}$}--(0.5,-0.7);
\draw[->] (-1,0.5) node[left] {$\ss x$}--(-0.7,0.5);
\end{tikzpicture}=\hspace{1.5mm}$\dfrac{(q;q)_{m}}{1-z_{0}x^{-1}}$
\end{center} the first column freezes with an overall weight of $\dfrac{(q;q)_{n}}{\prod^{n}_{i=1}(1-z_{0}x^{-1}_i)}$.

We then get that:
\begin{equation}
\mathcal{Z}_{\bl{C}}
=
    \dfrac{(q;q)_{n}}{\prod^{n}_{i=1}(1-z_{0}x^{-1}_i)} \times 
\begin{tikzpicture}[scale=0.6,baseline=(current bounding box.center)]
\draw (0:0) coordinate (A)--++(90:4) coordinate (B)--++(0:4) coordinate (C)-- ++(-90:4) coordinate (D)--cycle;
\path (0:0)++(-90:0.5)++(0:2.5) node {$\xleftarrow[w]{}$};
\path (0:0)++(90:4.5)++(0:2.5) node {$\xleftarrow[l]{}$};
\foreach\x in {1,2,3,4}{
\draw[gray] (0,0)++(0:\x)--++(90:4);
\draw[gray] (0,0)++(90:\x)--++(0:4);};
\foreach\x in {1,2,3,4}{
\draw[fill=red] (0,0)++(90:\x-0.5) circle (0.1);
\draw[fill=white] (0,0)++(0:4)++(90:\x-0.5) circle (0.1);};
\end{tikzpicture}
\end{equation}

We denote the partition function of the remaining lattice of $\bl{C}$ as $\widetilde{\rh}_{l/w}(x_1,\dots,x_n;z_1,\dots,z_{N})$. 

\

\begin{lemma} The following equation holds:
   \begin{equation}
 \label{eq:relationshipbetweenfandtildef}
 (-1)^{n}q^{-\sum^{N}\limits_{i=1}(i-1)(l_{i}-w_{i})}\widetilde{\rh}_{l/w}(x_1,\dots,x_n;z_1,\dots,z_{N})|_{z_{i}=q}=\rh_{l/w}(x_1,\dots,x_n).
\end{equation}

\end{lemma}

\begin{proof}
To compare $\rh$ and $\widetilde{\rh}$, it is convenient to give an alternate presentation for $\widetilde{\rh}$ by complementing the particles on horizontal edges.
\begin{equation}
\widetilde{\rh}_{l/w}(x_1,\dots,x_n;z_1,\dots,z_{N})=   \begin{tikzpicture}[scale=0.6,baseline=(current bounding box.center)]
\draw (0:0) coordinate (A)--++(90:4) coordinate (B)--++(0:4) coordinate (C)-- ++(-90:4) coordinate (D)--cycle;
\path (0:0)++(-90:0.5)++(0:2.5) node {$\xleftarrow[{w}]{}$};
\path (0:0)++(90:4.5)++(0:2.5) node {$\xleftarrow[{l}]{}$};
\foreach\x in {1,2,3,4}{
\draw[gray] (0,0)++(0:\x)--++(90:4);
\draw[gray] (0,0)++(90:\x)--++(0:4);};
\foreach\x in {1,2,3,4}{
\draw[fill=white] (0,0)++(90:\x-0.5) circle (0.1);
\draw[fill=red] (0,0)++(0:4)++(90:\x-0.5) circle (0.1);};
\end{tikzpicture} 
\end{equation}
with the following vertices:

\begin{equation}
\label{weight:6v-regionC}
\begin{tabular}{cccccc}
     \begin{tikzpicture}[baseline={(0,0)},scale=0.8]
      \draw (0:0)--++(90:1)--++(0:1)--++(-90:1)--++(180:1);
      \draw[fill=white] (0:0)++(90:0.5) circle (0.1);
      \draw[fill=white] (0,0)++(0:0.5) circle (0.1);
      \draw[fill=white] (0,0)++(0:1)++(90:0.5) circle (0.1);
      \draw[fill=white] (0,0)++(0:0.5)++(90:1) circle (0.1);
     \end{tikzpicture}&\qquad
     \begin{tikzpicture}[baseline={(0,0)},scale=0.8]
     \draw (0:0)--++(90:1)--++(0:1)--++(-90:1)--++(180:1);
      \draw[fill=red] (0:0)++(90:0.5) circle (0.1);
      \draw[fill=red] (0,0)++(0:0.5) circle (0.1);
      \draw[fill=red] (0,0)++(0:1)++(90:0.5) circle (0.1);
      \draw[fill=red] (0,0)++(0:0.5)++(90:1) circle (0.1);
      \draw[red,thick] (0:0.5)--++(135:0.6);
      \draw[red,thick] (0:0)++(90:1)++(0:0.5)--++(-45:0.7);
     \end{tikzpicture}&
     \begin{tikzpicture}[baseline={(0,0)},scale=0.8]
      \draw (0:0)--++(90:1)--++(0:1)--++(-90:1)--++(180:1);
      \draw[fill=white] (0:0)++(90:0.5) circle (0.1);
      \draw[fill=red] (0,0)++(0:0.5) circle (0.1);
      \draw[fill=white] (0,0)++(0:1)++(90:0.5) circle (0.1);
      \draw[fill=red] (0,0)++(0:0.5)++(90:1) circle (0.1);
      \draw[red,thick] (0:0.5)--++(90:1);
     \end{tikzpicture}&
     \begin{tikzpicture}[baseline={(0,0)},scale=0.8]
      \draw (0:0)--++(90:1)--++(0:1)--++(-90:1)--++(180:1);
      \draw[fill=red] (0:0)++(90:0.5) circle (0.1);
      \draw[fill=white] (0,0)++(0:0.5) circle (0.1);
      \draw[fill=red] (0,0)++(0:1)++(90:0.5) circle (0.1);
      \draw[fill=white] (0,0)++(0:0.5)++(90:1) circle (0.1);
      \draw[red,thick] (90:0.5)--++(0:1);
     \end{tikzpicture}&
     \begin{tikzpicture}[baseline={(0,0)},scale=0.8]
     \draw (0:0)--++(90:1)--++(0:1)--++(-90:1)--++(180:1);
      \draw[fill=white] (0:0)++(90:0.5) circle (0.1);
      \draw[fill=white] (0,0)++(0:0.5) circle (0.1);
      \draw[fill=red] (0,0)++(0:1)++(90:0.5) circle (0.1);
      \draw[fill=red] (0,0)++(0:0.5)++(90:1) circle (0.1);
      \draw[red,thick] (0:0)++(0:1)++(90:0.5)--++(135:0.8);
     \end{tikzpicture}&
     \begin{tikzpicture}[baseline={(0,0)},scale=0.8]
      \draw (0:0)--++(90:1)--++(0:1)--++(-90:1)--++(180:1);
      \draw[fill=red] (0:0)++(90:0.5) circle (0.1);
      \draw[fill=red] (0,0)++(0:0.5) circle (0.1);
      \draw[fill=white] (0,0)++(0:1)++(90:0.5) circle (0.1);
      \draw[fill=white] (0,0)++(0:0.5)++(90:1) circle (0.1);
      \draw[red,thick] (0:0.5)--++(135:0.7);
     \end{tikzpicture}\\[2em]
     \qquad $1$ \qquad&\qquad $q$\qquad &\qquad $\dfrac{1-q z x^{-1}}{1-z x^{-1}}$  \qquad&\qquad $\dfrac{1-q z x^{-1}}{1-z x^{-1}}$  \qquad&\qquad $\dfrac{(1-q)zx^{-1}}{1-zx^{-1}}$ \qquad &\qquad $\dfrac{1-q}{1-zx^{-1}}$\qquad\\ 
\end{tabular}
\end{equation}

We apply two operation to the weights in~\eqref{weight:6v-regionC}:
\begin{itemize}
    \item Multiplication by $q^{-1}$ for each vertex whose left edge is occupied by a red particle.  
    Consider a red particle entering from the right and exiting through the $i^\text{th}$ column from the right. This path necessarily passes through $i-1$ vertices where the red particle appears on the left edge. A similar argument applies to red particles entering from the bottom. Hence, this modification contributes a factor of  
    \[
    q^{-\sum_{i=1}^N (i-1)(l_i - w_i)}
    \]  
    to $\widetilde{H}_{l/w}$.

    \item Multiply a negative sign to the last two vertices. 
    For each red particle, the last two tile types occur in pairs in each column, except in the column through which the particle exits. Thus, multiply by $(-1)$ to these vertex weights is equal to multiplying $(-1)^n$ to $\widetilde{H}_{l/w}$, where $n$ is the number of red particles.

    \item 
    And then rotate them by $90^{\circ}$ anti-clockwise. For convenience, we reproduce the vertices after applying these operations.
\end{itemize}

\begin{equation}
\label{weight:6v-regionC-rotated}
\begin{tabular}{cccccc}
     \begin{tikzpicture}[baseline={(0,0)},scale=0.8]
      \draw (0:0)--++(90:1)--++(0:1)--++(-90:1)--++(180:1);
      \draw[fill=white] (0:0)++(90:0.5) circle (0.1);
      \draw[fill=white] (0,0)++(0:0.5) circle (0.1);
      \draw[fill=white] (0,0)++(0:1)++(90:0.5) circle (0.1);
      \draw[fill=white] (0,0)++(0:0.5)++(90:1) circle (0.1);
     \end{tikzpicture}&\qquad
     \begin{tikzpicture}[rotate=90,baseline={(0,0)},scale=0.8]
     \draw (0:0)--++(90:1)--++(0:1)--++(-90:1)--++(180:1);
      \draw[fill=red] (0:0)++(90:0.5) circle (0.1);
      \draw[fill=red] (0,0)++(0:0.5) circle (0.1);
      \draw[fill=red] (0,0)++(0:1)++(90:0.5) circle (0.1);
      \draw[fill=red] (0,0)++(0:0.5)++(90:1) circle (0.1);
      \draw[red,thick] (0:0.5)--++(135:0.6);
      \draw[red,thick] (0:0)++(90:1)++(0:0.5)--++(-45:0.7);
     \end{tikzpicture}&
     \begin{tikzpicture}[rotate=90,baseline={(0,0)},scale=0.8]
      \draw (0:0)--++(90:1)--++(0:1)--++(-90:1)--++(180:1);
      \draw[fill=white] (0:0)++(90:0.5) circle (0.1);
      \draw[fill=red] (0,0)++(0:0.5) circle (0.1);
      \draw[fill=white] (0,0)++(0:1)++(90:0.5) circle (0.1);
      \draw[fill=red] (0,0)++(0:0.5)++(90:1) circle (0.1);
      \draw[red,thick] (0:0.5)--++(90:1);
     \end{tikzpicture}&
     \begin{tikzpicture}[rotate=90,baseline={(0,0)},scale=0.8]
      \draw (0:0)--++(90:1)--++(0:1)--++(-90:1)--++(180:1);
      \draw[fill=red] (0:0)++(90:0.5) circle (0.1);
      \draw[fill=white] (0,0)++(0:0.5) circle (0.1);
      \draw[fill=red] (0,0)++(0:1)++(90:0.5) circle (0.1);
      \draw[fill=white] (0,0)++(0:0.5)++(90:1) circle (0.1);
      \draw[red,thick] (90:0.5)--++(0:1);
     \end{tikzpicture}&
     \begin{tikzpicture}[rotate=90,baseline={(0,0)},scale=0.8]
     \draw (0:0)--++(90:1)--++(0:1)--++(-90:1)--++(180:1);
      \draw[fill=white] (0:0)++(90:0.5) circle (0.1);
      \draw[fill=white] (0,0)++(0:0.5) circle (0.1);
      \draw[fill=red] (0,0)++(0:1)++(90:0.5) circle (0.1);
      \draw[fill=red] (0,0)++(0:0.5)++(90:1) circle (0.1);
      \draw[red,thick] (0:0)++(0:1)++(90:0.5)--++(135:0.8);
     \end{tikzpicture}&
     \begin{tikzpicture}[rotate=90,baseline={(0,0)},scale=0.8]
      \draw (0:0)--++(90:1)--++(0:1)--++(-90:1)--++(180:1);
      \draw[fill=red] (0:0)++(90:0.5) circle (0.1);
      \draw[fill=red] (0,0)++(0:0.5) circle (0.1);
      \draw[fill=white] (0,0)++(0:1)++(90:0.5) circle (0.1);
      \draw[fill=white] (0,0)++(0:0.5)++(90:1) circle (0.1);
      \draw[red,thick] (0:0.5)--++(135:0.7);
     \end{tikzpicture}\\[2em]
     \qquad $1$ &\qquad $1$ &\qquad $\dfrac{q(1-q^{-1} x z^{-1})}{1-x z^{-1}}$  &\qquad $\dfrac{(1-q^{-1} x z^{-1})}{1-x z^{-1}}$  &\qquad $\dfrac{(1-q)}{1-x z^{-1}}$  &\qquad $\dfrac{(1-q)xz^{-1}q^{-1}}{1-xz^{-1}}$ 
\end{tabular}
\end{equation}
 These weights are identical to the weights ~\eqref{weights:6v-regionA-complemented} (after specialisation $z_{i}=q^{-1}$) used in evaluating $\mathcal{Z}_{\bl{A}}$. We then deduce that:
\begin{equation}
 \label{eq:relationshipbetweenfandtildef_2}
 \rh_{l/w}(x_1,\dots,x_n;z_1,\dots,z_{N})=(-1)^{n}q^{-\sum^{N}\limits_{i=1}(i-1)(l_{i}-w_{i})}\widetilde{\rh}_{l/w}(x_1,\dots,x_n;z_1,\dots,z_{N}).
\end{equation}
\end{proof}

Finally, we get:
\begin{equation}
\label{weightofC}
 \cz_{\bl{C}}=\left({(-1)^{n} q^{\sum^{N}\limits_{i=1}(i-1)(l_{i}-w_{i})}}\right)\left(\dfrac{(q;q)_{n}}{\prod^{n}_{i=1}(1-z_{0} x^{-1}_{i})}\right) \rh_{l/w}(x_1,\dots,x_n;z_1,\dots,z_{N})
\end{equation}

\subsection{Region \texorpdfstring{$\bl{D}$}{D}}
\label{sub:spin1regD}

Observe that $n$ red particles enter $\bl{D}$ through its first column. These red particles cannot enter $\bl{F}$ since they would not be able to exit. Therefore, all red particles must travel into $\bl{E}$, ensuring that the right boundary of $\bl{D}$ consists solely of a binary string of white and blue particles.

\

From the left boundary of $\bl{D}$, $n$ blue particles enter. These particles can move into either $\bl{F}$ or $\bl{E}$, leading to many non-trivial configurations. However, we take the limit as $P \to \infty$ which requires certain convergence constraints. We argue that in this limit, all the blue particles entering $D$ from the left travel across into $\bl{F}$ and there by freezing $\bl{D}$.

\

\begin{tabular}{c c c}
&$
\begin{tikzpicture}[scale=0.7,baseline=
(current bounding box.center)]
\draw (0:0)--++(180:1)--++(90:1)--++(0:1)--++(-90:1); 
\draw[fill=cyan] (0,0)++(180:1)++(90:0.5) circle (0.1);
\draw[fill=white] (0,0)++(180:0.5) circle (0.1);
\draw[fill=cyan] (0,0)++(90:0.5) circle (0.1);
\draw[fill=white] (0,0)++(90:1)++(180:0.5) circle (0.1);
\draw[cyan,thick] (0:0)++(90:0.5)--++(180:1);
\end{tikzpicture}
=\left| \dfrac{1-x y^{-1}q^{-1}}{1-xy^{-1}} \right|<1
$
& \\[1em]
$
\begin{tikzpicture}[scale=0.7,baseline=(current bounding box.center)]
\draw (0:0)--++(180:1)--++(90:1)--++(0:1)--++(-90:1); 
\draw[fill=cyan] (0,0)++(180:1)++(90:0.5) circle (0.1);
\draw[fill=cyan] (0,0)++(180:0.5) circle (0.1);
\draw[fill=cyan] (0,0)++(90:0.5) circle (0.1);
\draw[fill=cyan] (0:0)++(90:1)++(180:0.5) circle (0.1);
\draw[cyan,thick] (0:0)++(90:0.5)--++(-135:0.7);
\draw[cyan,thick] (0:0)++(90:0.5)++(180:1)--++(45:0.7);
\end{tikzpicture}=\left| \dfrac{1-z_j y^{-1}_i}{1-q^{-1}z_j y_i} \right|<1 
$
& &
$
 \begin{tikzpicture}[scale=0.7,baseline=(current bounding box.center)]
\draw[gray] (0,0) rectangle (1,1); 
\draw[cyan,->,rounded corners] (0,0.5)--(0.4,0.5)--(0.4,1);
\draw[cyan,->,rounded corners] (0.5,0)--(0.5,1);
\draw[cyan,->,rounded corners] (0.6,0)--(0.6,0.5)--(1,0.5);
\draw[fill=cyan] (0,0)++(0:1)++(90:0.5) circle (0.1);
\draw[fill=cyan] (0,0)++(90:0.5) circle (0.1);
\end{tikzpicture}= \left| \dfrac{1-q^{\mathtt{M}-1} z_{0}y_{i}^{-1}}{1-q^{-1}z_{0}y^{-1}_{i}} \right|<1 
$
\end{tabular}

\

Suppose a white particle is located on the bottom boundary of $\bl{F}$. Given the boundary conditions of $\bl{F}$, this particle must eventually exit through the top boundary.

The maximum number $\begin{tikzpicture}[scale=0.5,baseline=
(current bounding box.center)]
\draw (0:0)--++(180:1)--++(90:1)--++(0:1)--++(-90:1); 
\draw[fill=cyan] (0,0)++(180:1)++(90:0.5) circle (0.1);
\draw[fill=white] (0,0)++(180:0.5) circle (0.1);
\draw[fill=cyan] (0,0)++(90:0.5) circle (0.1);
\draw[fill=white] (0,0)++(90:1)++(180:0.5) circle (0.1);
\draw[cyan,thick] (0:0)++(90:0.5)--++(180:1);
\end{tikzpicture}$ of vertices occurs when the particle moves vertically upward without any diagonal steps. In contrast, in the worst case, the particle may make up to 
$n$ rightward diagonal moves before proceeding vertically. Therefore, we conclude that there will be at least $P - n$ vertices of type $\begin{tikzpicture}[scale=0.5,baseline=
(current bounding box.center)]
\draw (0:0)--++(180:1)--++(90:1)--++(0:1)--++(-90:1); 
\draw[fill=cyan] (0,0)++(180:1)++(90:0.5) circle (0.1);
\draw[fill=white] (0,0)++(180:0.5) circle (0.1);
\draw[fill=cyan] (0,0)++(90:0.5) circle (0.1);
\draw[fill=white] (0,0)++(90:1)++(180:0.5) circle (0.1);
\draw[cyan,thick] (0:0)++(90:0.5)--++(180:1);
\end{tikzpicture}$ in $\bl{F}$.

\

Having a white particle on the bottom boundary of $\bl{F}$ implies that a blue particle from $\bl{D}$ has entered $\bl{E}$. Using a similar argument as above, we can conclude that there will be at least $(P - N)$ tiles of types \begin{tikzpicture}[scale=0.5,baseline=(current bounding box.center)] \draw[gray] (0,0) rectangle (1,1); 
particles from the left \draw[cyan,->,rounded corners] (0,0.5)--(0.4,0.5)--(0.4,1); \draw[cyan,->,rounded corners] (0.5,0)--(0.5,1); \draw[cyan,->,rounded corners] (0.6,0)--(0.6,0.5)--(1,0.5); 
\draw[fill=cyan] (0,0)++(0:1)++(90:0.5) circle (0.1); 
\draw[fill=cyan] (0,0)++(90:0.5) circle (0.1); \end{tikzpicture} 
and \begin{tikzpicture}[scale=0.5,baseline=(current bounding box.center)] \draw (0:0)--++(180:1)--++(90:1)--++(0:1)--++(-90:1);
\draw[fill=cyan] (0,0)++(180:1)++(90:0.5) circle (0.1);
\draw[fill=cyan] (0,0)++(180:0.5) circle (0.1); 
\draw[fill=cyan] (0,0)++(90:0.5) circle (0.1); 
\draw[fill=cyan] (0:0)++(90:1)++(180:0.5) circle (0.1); \draw[cyan,thick] (0:0)++(90:0.5)--++(-135:0.7); \draw[cyan,thick] (0:0)++(90:0.5)++(180:1)--++(45:0.7); \end{tikzpicture} within $\bl{E}$. The worst case scenario is illustrated in the figure below.

\

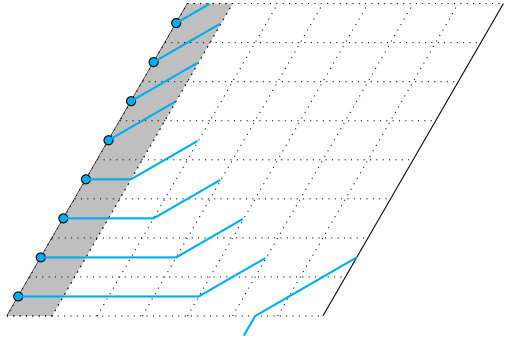
\begin{equation*}
 \begin{tikzpicture}[scale=0.6,baseline=(current bounding box.center)]
\draw (0:0)--++(60:8);
\draw (0:0)++(0:7)--++(60:8);
\draw[fill= lightgray,dotted] (0:0)--++(0:1)--++(60:8)--++(180:1)--++(-120:8);
\draw[dotted] (0:0)++(0:1)--++(60:8);
\draw[dotted] (0:0)++(0:2)--++(60:8);
\draw[dotted] (0:0)++(0:3)--++(60:8);
\draw[dotted] (0:0)++(0:4)--++(60:8);
\draw[dotted] (0:0)++(0:5)--++(60:8);
\draw[dotted] (0:0)++(0:6)--++(60:8);
\foreach \x in {0,1,2,3,4,5,6,7,8}{
\draw[dotted] (60:\x)--++(0:7);};

\foreach \x in {0,1,2,3,4,5,6,7}{
\draw[fill=cyan] (0:0)++(60:\x+0.5) circle (0.1);
}
\draw[cyan,thick] (0:0)++(60:0.5)--++(0:4)--++(30:1.7);
\draw[cyan,thick] (0:0)++(60:1.5)--++(0:3)--++(30:1.7);
\draw[cyan,thick] (0:0)++(60:2.5)--++(0:2)--++(30:1.7);
\draw[cyan,thick] (0:0)++(60:3.5)--++(0:1)--++(30:1.7);
\draw[cyan,thick] (0:0)++(60:4.5)--++(30:1.7);
\draw[cyan,thick] (0:0)++(60:5.5)--++(30:1.7);
\draw[cyan,thick] (0:0)++(60:6.5)--++(30:1.7);
\draw[cyan,thick] (0:0)++(60:7.5)--++(30:0.85);
\draw[cyan,thick] (0:0)++(0:5.5)++(-120:0.5)--++(60:0.5)--++(30:2.6);
\end{tikzpicture}   
\end{equation*}

\

By combining these arguments, we deduce that as $P \to \infty$, convergence constraints imply that the weight of configurations where a blue particle from $\bl{D}$ enters $ \bl{E}$ vanishes. Consequently, only those configurations in which the right boundary of $\bl{E}$ forms an infinitely long binary string ending in blue particles survive in the limit.

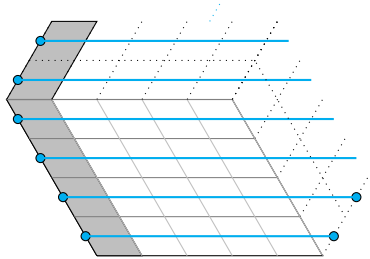
\begin{equation*}
 \begin{tikzpicture}[scale=0.6,baseline=(current bounding box.center)]
\draw (0:0)--++(120:4) coordinate (A)--++(0:5)  coordinate (B)--++(-60:4)  coordinate (C)--++(180:5);
\draw[fill= lightgray] (0:0)--++(0:1)--++(120:4)--++(180:1)--++(-60:4);
\draw[fill= lightgray] (0:0)++(120:4)--++(60:2)--++(0:1)--++(-120:2)--++(180:1);
\draw[dotted] (A)--++(60:2);
\draw[dotted] (B)--++(60:2);
\draw[dotted] (C)--++(60:2);
\foreach \x in {1,2,3,4}{
\draw[gray] (120:\x)--++(0:5);
\draw[dotted] (120:\x)++(0:5)--++(60:2);
\draw[fill=cyan] (0,0)++(120:\x-0.5) circle (0.1);
};
\foreach \x in {1,2,3,4,5}{
\draw[lightgray] (0:\x)--++(120:4);
\draw[dotted] (0:\x)++(120:4)--++(60:2);
};
\draw[fill=cyan] (0,0)++(120:4)++(60:0.5) circle (0.1);
\draw[fill=cyan] (0,0)++(120:4)++(60:1.5) circle (0.1);
\draw[fill=cyan] (0,0)++(0:5)++(60:0.5) circle (0.1);
\draw[fill=cyan] (0,0)++(0:5)++(60:1.5) circle (0.1);
\draw[cyan,thick] (0:0)++(120:4)++(60:0.5)--++(0:6.5);
\draw[cyan,thick] (0:0)++(120:4)++(60:1.5)--++(0:5.5);
\draw[cyan,thick] (0:0)++(120:3.5)--++(0:7);
\draw[cyan,thick] (0:0)++(120:2.5)--++(0:7);
\draw[cyan,thick] (0:0)++(120:1.5)--++(0:6.5);
\draw[cyan,thick] (0:0)++(120:0.5)--++(0:5.5);
\draw[dotted] (0:0)++(120:4)++(60:1)--++(0:5)--++(-60:4);
\draw[cyan,dotted] (0:0)++(120:2.5)++(0:2)++(60:2)++(60:1.5)--++(60:0.5);
\end{tikzpicture}   
\end{equation*}

\subsection{Region \texorpdfstring{$F$}{F}} Since all the blue particles from $\bl{D}$ enter $\bl{F}$, the bottom boundary (or south-west boundary of $\bl{F}$ in ~\eqref{yangbaxterlatticemodel-sixvertex}) is fixed. This forces the left boundary to be a binary string of blue particles and holes with exactly $n$ holes. For a fixed binary string $k$, we have:

\begin{equation}
 \cz_{\bl{F}}=
\begin{tikzpicture}[scale=0.6,baseline=(current bounding box.center)]
\draw (0:0)--++(180:4)--++(90:6)--++(0:4)--++(-90:6);
\foreach\x in {1,2,3,4,5,6}{
\draw[gray] (0:0)++(90:\x)--++(180:4);
\draw[fill=cyan] (0:0)++(90:\x-0.5) circle (0.1);
};
\foreach\x in {1,2,3,4}{
\draw[gray] (0:0)++(180:\x)--++(90:6);
 };
 \foreach\x in {1,2,3,4}{
 \draw[fill=cyan] (0:0)++(180:\x-0.5) circle (0.1);
 \draw[fill=white] (0:0)++(90:6)++(180:\x-0.5) circle (0.1);
 };
 \path (0:0)++(180:4.5)++(90:2.5) node[rotate=90] {$\xleftarrow[k]{}$};
\draw[->](0:0)++(180:2)++(-90:1.5)--++(90:1);
\draw[->](0:0)++(180:6.5)++(90:2.5)--++(0:1);
\end{tikzpicture}
\end{equation}

By applying similar reasoning/discussion as in~\eqref{partitionfucntion: regionA}, we conclude that:
\begin{equation}
   \label{weights:regionF}
   \cz_{\bl{F}}=\rh_{\overline{k}}(x_1,\dots,x_n;y_1,\dots,y_P).
\end{equation}

\subsection{Final equation}
Let $\mathbf{x}=(x_1,\dots,x_n)$, $\mathbf{y}=(y_1,y_{2},\dots)$ and $\mathbf{z}=(z_1,\dots,z_N)$. Putting all together, the original equation ~\eqref{yangbaxterlatticemodel-sixvertex} gives us the following:
\begin{align}
\label{productidentity:nospecialisations}
\displaystyle
 \left( (-1)^{n} q^{\sum\limits^{N}_{i=1}(i-1)(l_{i}-w_{i})}\right)\left(\dfrac{(q;q)_{n}}{\prod^{n}_{i=1}(1- z_{0}x^{-1}_i)}\right){\rh}_{l/w}(\mathbf{x};\mathbf{z}){\rh}_{\overline{m}}(\mathbf{x};\mathbf{y} )=\sum_{k}\mathcal{Z}^{k,w}_{l,m}(\mathbf{y};\mathbf{z})\rh_{\overline{k}}(\mathbf{x};\mathbf{y})
\end{align}
where 
\begin{equation}
\label{puzzles: wave function with no specialisation}
 \mathcal{Z}^{k,w}_{l,m}(\mathbf{y};\mathbf{z}) =
 \begin{tikzpicture}[scale=0.6,baseline=(current bounding box.center)]
 \draw[white,fill=lightgray] (0:0)--++(0:1)--++(90:6)--++(180:1)--++(-90:6);
 \draw (0,0)--++(0:5)--++(90:6)--++(180:5)--++(-90:6);
  \foreach\x in {1,2,3,4,5}{
 \draw[gray] (0:\x)--++(90:6);
  \draw[gray] (90:\x)--++(0:5);
 };
  \foreach \x in {1,2,3,4,5,6}{
 \draw[->] (0:0)++(180:2)++(90:\x-0.5)--++(0:0.5);
 };
 \foreach \x in {1,2,3,4,5}{
 \draw[->] (0:0)++(-90:1.5)++(0:\x-0.5)--++(90:0.5);
 };
 \path (0:0)++(180:3)++(90:0.5) node {$\ss (qy_{P},1)$};
 \path (0:0)++(180:3)++(90:2) node[rotate=90] {$\dots$};
 \path (0:0)++(180:3)++(90:3.5) node {$\ss (qy_{3},2)$};
 \path (0:0)++(180:3)++(90:4.5) node {$\ss (qy_{2},1)$};
  \path (0:0)++(180:3)++(90:5.5) node {$\ss (qy_{1},1)$};
  \path (0:0)++(-90:2)++(0:0.2) node {$\ss (z_{0},\mathtt{M})$};
  \path (0:0)++(-90:2)++(0:1.6) node {$\ss (z_{N},1)$};
  \path (0:0)++(-90:2)++(0:3) node {$\dots$};
  \path (0:0)++(-90:2)++(0:4.6) node {$\ss (z_{1},1)$};
 \path (0:0)++(180:0.5)++(90:0.5) node[cyan] {$\ss m_{P}$};
 \path (0:0)++(180:0.5)++(90:2) node[rotate=90] {$\dots$};
 \path (0:0)++(180:0.5)++(90:3.5) node[cyan] {$\ss m_{3}$};
 \path (0:0)++(180:0.5)++(90:4.5) node[cyan] {$\ss m_{2}$};
  \path (0:0)++(180:0.5)++(90:5.5) node[cyan] {$\ss m_1$};
 \path (0:0)++(0:5.5)++(90:0.5) node[cyan] {$\ss k_{P}$};
 \path (0:0)++(0:5.5)++(90:2) node[rotate=90] {$\dots$};
 \path (0:0)++(0:5.5)++(90:3.5) node[cyan] {$\ss k_{3}$};
 \path (0:0)++(0:5.5)++(90:4.5) node[cyan] {$\ss k_{2}$};
  \path (0:0)++(0:5.5)++(90:5.5) node[cyan] {$\ss k_1$};
  \path (0:0)++(-90:0.5)++(0:0.5) node {$\ss \textcolor{red}{{n}}$};
  \path (0:0)++(-90:0.5)++(0:1.6) node[red] {$\ss w_{N}$};
  \path (0:0)++(-90:0.5)++(0:2.6) node[red] {$\dots$};
  \path (0:0)++(-90:0.5)++(0:3.6) node[red] {$\ss w_2$};
  \path (0:0)++(-90:0.5)++(0:4.6) node[red] {$\ss w_1$};
  \draw[fill=white] (0:0)++(90:6)++(0:0.5) circle (0.1);
  \path (0:0)++(90:6.5)++(0:1.6) node[red] {$ \ss l_{N}$};
  \path (0:0)++(90:6.5)++(0:2.6) node[red] {$\dots$};
  \path (0:0)++(90:6.5)++(0:3.6) node[red] {$\ss l_2$};
  \path (0:0)++(90:6.5)++(0:4.6) node[red] {$\ss l_1$};
 \end{tikzpicture}
\end{equation}

\subsection{At \texorpdfstring{$z_{0}=0$}{}}
To relate ~(\ref{productidentity:nospecialisations}) to Theorem \ref{theorem:puzzlesforwavefunctions} we need to perform certain specialisations. We begin by setting $z_{0}=0$.

 Observe that we have a total of $P$ particles entering the first column, $P-n$ blue particles from the left and $n$ red particles from the bottom. Since the number of particles equals the number of rows and none can exit through the top edge, every particle must exit through a right edge.  
 
 Additionally, the weight of the vertex $\begin{tikzpicture}[baseline=(current bounding box.center),scale=0.6]
    \draw[fill=lightgray] (0:0)--++(0:1)--++(90:1)--++(180:1)--++(-90:1);
    \draw[fill=cyan] (0:0)++(90:0.5) circle (0.1);
    \draw[fill=red] (0:0)++(0:1)++(90:0.5) circle (0.1);
    \path (0:0)++(-90:0.3)++(0:0.5) node {$\ss (a,b)$};
\end{tikzpicture}$, which is equal to $\dfrac{z_{0}(1-q^{a})q^{b}}{y-q^{b}z_{0}}$, vanishes when $z_{0}=0$. This forces all blue particles to move across as shown in the picture below.

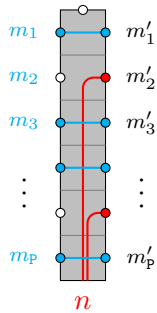
\begin{center}
\begin{tikzpicture}[scale=0.6,baseline=(current bounding box.center)]
      \draw[fill=lightgray] (0:0)--++(0:1)--++(90:6)--++(180:1)--++(-90:6);
      \foreach \x in {1,2,3,4,5}{
      \draw[gray] (90:\x)--++(0:1);};
      \draw[fill=white] (0:0)++(90:6)++(0:0.5) circle (0.1);
      \path[fill=white] (0:0)++(0:0.5)++(-90:0.5) node {$\textcolor{red}{n}$};
 \path (0:0)++(180:0.8)++(90:0.5) node[cyan] {$\ss m_{\tp}$};
 \path (0:0)++(180:0.8)++(90:2) node[rotate=90] {$\dots$};
 \path (0:0)++(180:0.8)++(90:3.5) node[cyan] {$\ss m_{3}$};
 \path (0:0)++(180:0.8)++(90:4.5) node[cyan] {$\ss m_{2}$};
  \path (0:0)++(180:0.8)++(90:5.5) node[cyan] {$\ss m_1$};
 \path (0:0)++(0:1.8)++(90:0.5) node {$\ss m'_{\tp}$};
 \path (0:0)++(0:1.8)++(90:2) node[rotate=90] {$\dots$};
 \path (0:0)++(0:1.8)++(90:3.5) node {$\ss m'_{3}$};
 \path (0:0)++(0:1.8)++(90:4.5) node {$\ss m'_{2}$};
  \path (0:0)++(0:1.8)++(90:5.5) node {$\ss m'_1$};
  \draw[red,thick,rounded corners](0:0.6)--++(90:1.5)--++(0:0.5);
  \draw[red,thick,rounded corners](0:0.5)--++(90:4.5)--++(0:0.5);
  \foreach \x in {0,2,3,5}{
  \draw[cyan, thick] (90:\x+0.5)--++(0:1);
  };
  \cbull{0}{0.5}{0.1};\cbull{1}{0.5}{0.1};
  \ebull{0}{1.5}{0.1};\rbull{1}{1.5}{0.1};
  \cbull{0}{2.5}{0.1};\cbull{1}{2.5}{0.1};
  \cbull{0}{3.5}{0.1};\cbull{1}{3.5}{0.1};
  \ebull{0}{4.5}{0.1};\rbull{1}{4.5}{0.1};
  \cbull{0}{5.5}{0.1};\cbull{1}{5.5}{0.1};
\end{tikzpicture}  
\end{center}

 Therefore, the first column in $\mathcal{Z}^{k,w}_{l,m}$ freezes at $z_{0}=0$. Moreover, the weight of this first column in $\mathcal{Z}^{k,w}_{l,m}$ is equal to $(q;q)_{n}$, which cancels out with the same on the left hand side of ~\eqref{productidentity:nospecialisations}.

\subsection{Main theorem} 
Our main Theorem ~\ref{theorem:puzzlesforwavefunctions} is obtained by specialising~\eqref{productidentity:nospecialisations}. We set $y_i= q^{-1}$, $z_{i}=q^{-1}$ for all $i \geq 1$. 

We list the tiles that appear in the puzzle region, along with their corresponding weights:
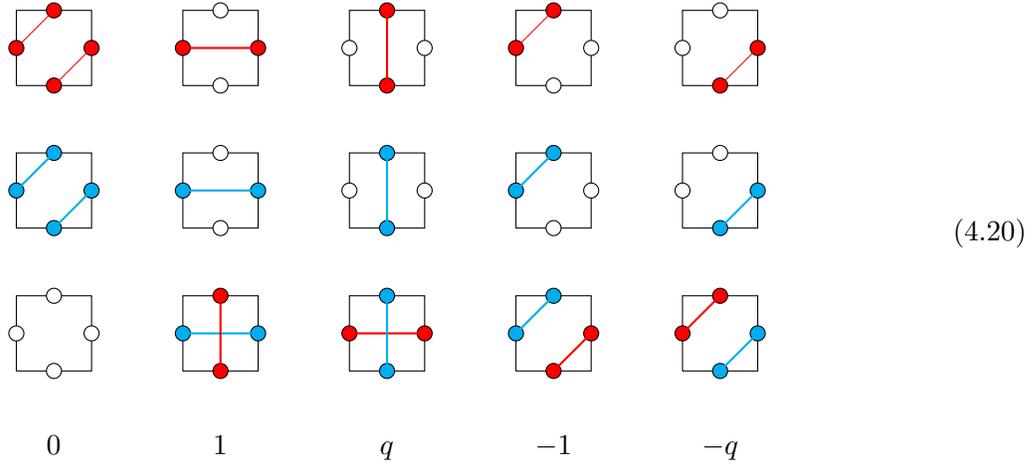
\begin{equation}
\label{puzzweights:1}
\begin{tabular}{c@{\hskip 1cm}c@{\hskip 1cm}c@{\hskip 1cm}c@{\hskip 1cm}c}
\begin{tikzpicture}[baseline=(current bounding box.center)]
    \draw (0:0)--++(0:1)--++(90:1)--++(180:1)--++(-90:1);
    \draw[fill=red] (0:0)++(90:0.5) circle (0.1);
    \draw[fill=red] (0:0)++(0:0.5) circle (0.1); 
    \draw[fill=red] (0:0)++(0:1)++(90:0.5) circle (0.1);
    \draw[fill=red] (0:0)++(0:0.5)++(90:1) circle (0.1);
    \draw[red] (0,0.5)--(0.5,1);
    \draw[red] (0.5,0)--(1,0.5);
\end{tikzpicture}
&
\begin{tikzpicture}[baseline=(current bounding box.center)]
    \draw (0:0)--++(0:1)--++(90:1)--++(180:1)--++(-90:1);
    \draw[fill=red] (0:0)++(90:0.5) circle (0.1);
    \draw[fill=white] (0:0)++(0:0.5) circle (0.1); 
    \draw[fill=red] (0:0)++(0:1)++(90:0.5) circle (0.1);
    \draw[fill=white] (0:0)++(0:0.5)++(90:1) circle (0.1);
    \draw[red,thick] (0,0.5)--(1,0.5);
\end{tikzpicture}
&
\begin{tikzpicture}[baseline=(current bounding box.center)]
    \draw (0:0)--++(0:1)--++(90:1)--++(180:1)--++(-90:1);
    \draw[fill=white] (0:0)++(90:0.5) circle (0.1);
    \draw[fill=red] (0:0)++(0:0.5) circle (0.1); 
    \draw[fill=white] (0:0)++(0:1)++(90:0.5) circle (0.1);
    \draw[fill=red] (0:0)++(0:0.5)++(90:1) circle (0.1);
    \draw[red,thick] (0.5,0)--(0.5,1);
\end{tikzpicture}
&
\begin{tikzpicture}[baseline=(current bounding box.center)]
    \draw (0:0)--++(0:1)--++(90:1)--++(180:1)--++(-90:1);
    \draw[fill=red] (0:0)++(90:0.5) circle (0.1);
    \draw[fill=white] (0:0)++(0:0.5) circle (0.1); 
    \draw[fill=white] (0:0)++(0:1)++(90:0.5) circle (0.1);
    \draw[fill=red] (0:0)++(0:0.5)++(90:1) circle (0.1);
    \draw[red] (0,0.5)--(0.5,1);
\end{tikzpicture}
&
\begin{tikzpicture}[baseline=(current bounding box.center)]
    \draw (0:0)--++(0:1)--++(90:1)--++(180:1)--++(-90:1);
    \draw[fill=white] (0:0)++(90:0.5) circle (0.1);
    \draw[fill=red] (0:0)++(0:0.5) circle (0.1); 
    \draw[fill=red] (0:0)++(0:1)++(90:0.5) circle (0.1);
    \draw[fill=white] (0:0)++(0:0.5)++(90:1) circle (0.1);
    \draw[red] (0.5,0)--(1,0.5);
\end{tikzpicture}
\\[3 em]
  \begin{tikzpicture}[baseline=(current bounding box.center)]
    \draw (0:0)--++(0:1)--++(90:1)--++(180:1)--++(-90:1);
    \draw[fill=cyan] (0:0)++(90:0.5) circle (0.1);
    \draw[fill=cyan] (0:0)++(0:0.5) circle (0.1); 
    \draw[fill=cyan] (0:0)++(0:1)++(90:0.5) circle (0.1);
    \draw[fill=cyan] (0:0)++(0:0.5)++(90:1) circle (0.1);
    \draw[cyan,thick] (0,0.5)--(0.5,1);
    \draw[cyan,thick] (0.5,0)--(1,0.5);
\end{tikzpicture}
&
\begin{tikzpicture}[baseline=(current bounding box.center)]
    \draw (0:0)--++(0:1)--++(90:1)--++(180:1)--++(-90:1);
    \draw[fill=cyan] (0:0)++(90:0.5) circle (0.1);
    \draw[fill=white] (0:0)++(0:0.5) circle (0.1); 
    \draw[fill=cyan] (0:0)++(0:1)++(90:0.5) circle (0.1);
    \draw[fill=white] (0:0)++(0:0.5)++(90:1) circle (0.1);
    \draw[cyan,thick] (0,0.5)--(1,0.5);
\end{tikzpicture}
&
\begin{tikzpicture}[baseline=(current bounding box.center)]
    \draw (0:0)--++(0:1)--++(90:1)--++(180:1)--++(-90:1);
    \draw[fill=white] (0:0)++(90:0.5) circle (0.1);
    \draw[fill=cyan] (0:0)++(0:0.5) circle (0.1); 
    \draw[fill=white] (0:0)++(0:1)++(90:0.5) circle (0.1);
    \draw[fill=cyan] (0:0)++(0:0.5)++(90:1) circle (0.1);
    \draw[cyan,thick] (0.5,0)--(0.5,1);
\end{tikzpicture}
&
\begin{tikzpicture}[baseline=(current bounding box.center)]
    \draw (0:0)--++(0:1)--++(90:1)--++(180:1)--++(-90:1);
    \draw[fill=cyan] (0:0)++(90:0.5) circle (0.1);
    \draw[fill=white] (0:0)++(0:0.5) circle (0.1); 
    \draw[fill=white] (0:0)++(0:1)++(90:0.5) circle (0.1);
    \draw[fill=cyan] (0:0)++(0:0.5)++(90:1) circle (0.1);
    \draw[cyan,thick] (0,0.5)--(0.5,1);
\end{tikzpicture}
&
\begin{tikzpicture}[baseline=(current bounding box.center)]
    \draw (0:0)--++(0:1)--++(90:1)--++(180:1)--++(-90:1);
    \draw[fill=white] (0:0)++(90:0.5) circle (0.1);
    \draw[fill=cyan] (0:0)++(0:0.5) circle (0.1); 
    \draw[fill=cyan] (0:0)++(0:1)++(90:0.5) circle (0.1);
    \draw[fill=white] (0:0)++(0:0.5)++(90:1) circle (0.1);
    \draw[cyan,thick] (0.5,0)--(1,0.5);
\end{tikzpicture}
\\[3 em] 
  \begin{tikzpicture}[baseline=(current bounding box.center)]
    \draw (0:0)--++(0:1)--++(90:1)--++(180:1)--++(-90:1);
    \draw[fill=white] (0:0)++(90:0.5) circle (0.1);
    \draw[fill=white] (0:0)++(0:0.5) circle (0.1); 
    \draw[fill=white] (0:0)++(0:1)++(90:0.5) circle (0.1);
    \draw[fill=white] (0:0)++(0:0.5)++(90:1) circle (0.1);
\end{tikzpicture}
&
\begin{tikzpicture}[baseline=(current bounding box.center)]
    \draw (0:0)--++(0:1)--++(90:1)--++(180:1)--++(-90:1);
    \draw[fill=cyan] (0:0)++(90:0.5) circle (0.1);
    \draw[fill=red] (0:0)++(0:0.5) circle (0.1); 
    \draw[fill=cyan] (0:0)++(0:1)++(90:0.5) circle (0.1);
    \draw[fill=red] (0:0)++(0:0.5)++(90:1) circle (0.1);
    \draw[cyan,thick] (0,0.5)--(1,0.5);
    \draw[red,thick] (0.5,0)--(0.5,1);
\end{tikzpicture}
&
\begin{tikzpicture}[baseline=(current bounding box.center)]
    \draw (0:0)--++(0:1)--++(90:1)--++(180:1)--++(-90:1);
    \draw[fill=red] (0:0)++(90:0.5) circle (0.1);
    \draw[fill=cyan] (0:0)++(0:0.5) circle (0.1); 
    \draw[fill=red] (0:0)++(0:1)++(90:0.5) circle (0.1);
    \draw[fill=cyan] (0:0)++(0:0.5)++(90:1) circle (0.1);
    \draw[red,thick] (0,0.5)--(1,0.5);
    \draw[cyan,thick] (0.5,0)--(0.5,1);
\end{tikzpicture}
&
\begin{tikzpicture}[baseline=(current bounding box.center)]
    \draw (0:0)--++(0:1)--++(90:1)--++(180:1)--++(-90:1);
    \draw[fill=cyan] (0:0)++(90:0.5) circle (0.1);
    \draw[fill=red] (0:0)++(0:0.5) circle (0.1); 
    \draw[fill=red] (0:0)++(0:1)++(90:0.5) circle (0.1);
    \draw[fill=cyan] (0:0)++(0:0.5)++(90:1) circle (0.1);
    \draw[cyan,thick] (0,0.5)--(0.5,1);
    \draw[red,thick] (0.5,0)--(1,0.5);
\end{tikzpicture}
&
\begin{tikzpicture}[baseline=(current bounding box.center)]
    \draw (0:0)--++(0:1)--++(90:1)--++(180:1)--++(-90:1);
    \draw[fill=red] (0:0)++(90:0.5) circle (0.1);
    \draw[fill=cyan] (0:0)++(0:0.5) circle (0.1); 
    \draw[fill=cyan] (0:0)++(0:1)++(90:0.5) circle (0.1);
    \draw[fill=red] (0:0)++(0:0.5)++(90:1) circle (0.1);
    \draw[red,thick] (0,0.5)--(0.5,1);
    \draw[cyan,thick] (0.5,0)--(1,0.5);
\end{tikzpicture}\\[3 em] 
$0$ & $1$ & $q$ & $-1$ & $-q$\\
\end{tabular}
\end{equation}

\

The tile weights above differ from those in Theorem~\ref{theorem:puzzlesforwavefunctions}. In order to obtain the weights in Theorem ~\ref{theorem:puzzlesforwavefunctions} we need to incorporate the factor $\displaystyle \left( (-1)^{n} q^{-\sum^{N}_{i=0}(i-1)(l_i - w_i)} \right)$ into ~(\ref{puzzweights:1}).

\

\begin{lemma}
\label{lemma:redvertices_in_puzzles}
    We claim that in a puzzle with boundary conditions as in ~\ref{puzzles: wave function with no specialisation} (after removing the first column) the number of vertices with a red particle on the right edge is equal to:
    \[
   \# \begin{tikzpicture}[scale=0.6,baseline= (current bounding box.center)]
        \draw (0:0)--++(0:1)--++(90:1)--++(180:1)--++(-90:1);
        \rbull{1}{0.5}{0.1};
    \end{tikzpicture}+n-Nn=-\sum_{i>0}(i-1)(l_{i}-w_{i})
    \]
\end{lemma}

\begin{proof}
 We recall the boundary conditions and notations. We have $n$ red particles entering from the left. The top boundary is labeled $l$ while the bottom boundary is labeled $w$. Additionally, $l$ has exactly $n$ more particles than $w$. Moreover, the weights of the vertices that have same label on all edges is $0$. This implies that two paths of same colors never cross each other. It is convenient to extend these puzzles so that all particles enter from the left, as shown in the picture below:

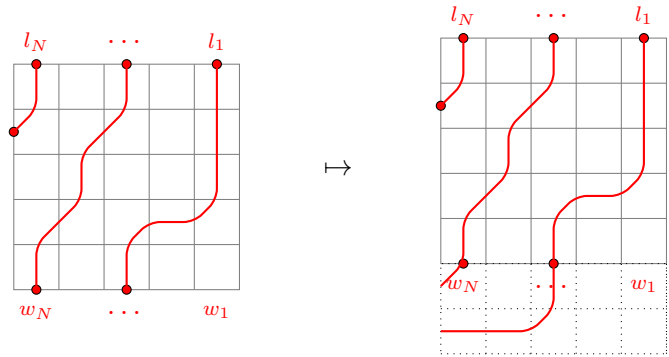
\[
 \begin{tikzpicture}[scale=0.6,baseline=(current bounding box.center)]
\draw[gray](0.5,0.5) rectangle (5.5,5.5);  
\foreach \x in {1,2,3,4}{
\draw[gray] (0.5,\x+0.5)--(5.5,\x+0.5);
};
\foreach \x in {1,2,3,4}{
\draw[gray] (\x+0.5,0.5)--(\x+0.5,5.5);
};
\rbull{1}{0.5}{0.1};
\rbull{3}{5.5}{0.1};
\rbull{3}{0.5}{0.1};
\rbull{5}{5.5}{0.1};
\rbull{0.5}{4}{0.1};
\rbull{1}{5.5}{0.1};
\draw[red,thick,rounded corners] (1,0.5)--(1,1.5)--(1.5,2)--(2,2.5)--(2,3.5)--(2.5,4)--(3,4.5)--(3,5.5);
\draw[red,thick,rounded corners] (3,0.5)--(3,1.5)--(3.5,2)--(4.5,2)--(5,2.5)--(5,5.5);
\draw[red,thick,rounded corners] (0.5,4)--(1,4.5)--(1,5.5);
\node[red] at (1,6) {$\ss l_{N}$};
\node[red] at (3,6) {$\dots$};
\node[red] at (5,6) {$\ss l_{1}$};
\node[red] at (1,0) {$\ss w_{N}$};
\node[red] at (3,0) {$\dots$};
\node[red] at (5,0) {$\ss w_{1}$};
\end{tikzpicture}
\hspace{1cm}\mapsto \hspace{1cm}
 \begin{tikzpicture}[scale=0.6,baseline=(current bounding box.center)]
\draw[gray](0.5,0.5) rectangle (5.5,5.5);  
\foreach \x in {1,2,3,4}{
\draw[gray] (0.5,\x+0.5)--(5.5,\x+0.5);
};
\foreach \x in {1,2,3,4}{
\draw[gray] (\x+0.5,0.5)--(\x+0.5,5.5);
};
\draw[dotted] (0.5,-1.5) rectangle (5.5,0.5);
\draw[dotted](0.5,-0.5) --(5.5,-0.5);
\draw[dotted](1.5,-1.5) --(1.5,0.5);\draw[dotted](2.5,-1.5) --(2.5,0.5);
\draw[dotted](3.5,-1.5) --(3.5,0.5);\draw[dotted](4.5,-1.5) --(4.5,0.5);
\draw[dotted](5.5,-1.5) --(5.5,0.5);
\rbull{1}{0.5}{0.1};
\rbull{3}{5.5}{0.1};
\rbull{3}{0.5}{0.1};
\rbull{5}{5.5}{0.1};
\rbull{0.5}{4}{0.1};
\rbull{1}{5.5}{0.1};
\draw[red,thick,rounded corners] (0.5,-0)--(1,0.5)--(1,1.5)--(1.5,2)--(2,2.5)--(2,3.5)--(2.5,4)--(3,4.5)--(3,5.5);
\draw[red,thick,rounded corners] (0.5,-1)--(2.5,-1)--(3,-0.5)--(3,0.5)--(3,1.5)--(3.5,2)--(4.5,2)--(5,2.5)--(5,5.5);
\draw[red,thick,rounded corners] (0.5,4)--(1,4.5)--(1,5.5);
\node[red] at (1,6) {$\ss l_{N}$};
\node[red] at (3,6) {$\dots$};
\node[red] at (5,6) {$\ss l_{1}$};
\node[red] at (1,0) {$\ss w_{N}$};
\node[red] at (3,0) {$\dots$};
\node[red] at (5,0) {$\ss w_{1}$};
\end{tikzpicture}
\]
 
Observe that when a red particle coming from the left exits through the $i^{th}$ column from the right, it implies that there are $N-i$ tiles with that red particle on their right edge. 

We can then conclude that the number of \begin{tikzpicture}[scale=0.5,baseline= (current bounding box.center)]
        \draw (0:0)--++(0:1)--++(90:1)--++(180:1)--++(-90:1);
        \rbull{1}{0.5}{0.1};
    \end{tikzpicture} vertices in the extended picture is $\sum_{i>0}(N-i) l_{i}$. 

    To compute the number of vertices \begin{tikzpicture}[scale=0.5,baseline= (current bounding box.center)]
        \draw (0:0)--++(0:1)--++(90:1)--++(180:1)--++(-90:1);
        \rbull{1}{0.5}{0.1};
    \end{tikzpicture} in the non-extended diagram, we need to subtract those coming from the extended part. This is computed as follows:

    \begin{equation*}
   \# \begin{tikzpicture}[scale=0.6,baseline= (current bounding box.center)]
        \draw (0:0)--++(0:1)--++(90:1)--++(180:1)--++(-90:1);
        \rbull{1}{0.5}{0.1};
    \end{tikzpicture}=\sum_{i}(N-i)l_{i}-\sum_{i}(N-i)w_{i}=
    N n-n-\sum_{i}(i-1) (l_{i}-w_{i})
    \end{equation*}
\end{proof}

Using ~\cref{lemma:redvertices_in_puzzles}, we can incorporate the factor  $\displaystyle \left( (-1)^{n} q^{-\sum^{N}_{i=0}(i-1)(l_i - w_i)} \right)$ into the puzzle weights ~(\ref{puzzweights:1}) with the following operations:
\begin{itemize}
    \item [(i)] If the right edge contains a red particle, multiply the weight by $q$.
    \item [(ii)] If the left edge does not contain a blue particle, multiply the weight by $q^{-1}$ which gives an overall factor of $q^{-n(N)}$, and 
    \item [(iii)] If the top edge contains a red particle, multiply the weight by $-q$; if the bottom edge contains a red particle, divide the weight by $-q$ which gives an overall factor of $(-q)^{n}$.
\end{itemize}
Finally, to obtain the tiles in~(\ref{weights:wavefunctions-puzzleweights}), multiply each weight by $(-1)$ whenever there is a blue particle on the top edge, and also whenever there is a blue particle on the bottom edge. Since $n$ blue particles enter from the right and exit through the left, this operation does not alter the partition function of a puzzle. We list the tiles with their weights after performing the above-mentioned operations:

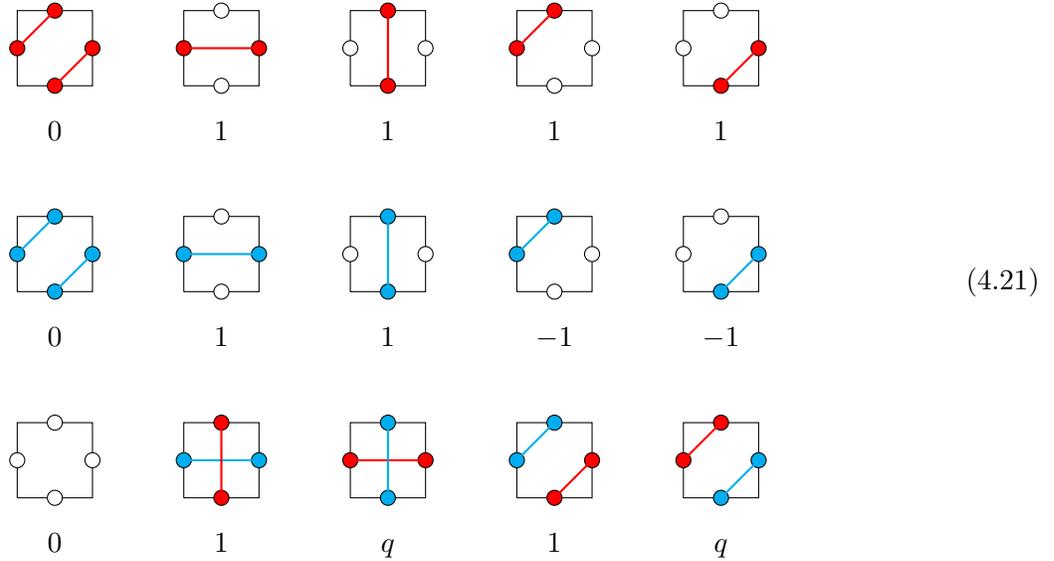
\begin{equation}
\begin{tabular}{c@{\hskip 5mm}@{\hskip 5mm}c@{\hskip 5mm}@{\hskip 5mm}c@{\hskip 5mm}@{\hskip 5mm}c@{\hskip 5mm}@{\hskip 5mm}c@{\hskip 5mm}}
\begin{tikzpicture}[baseline=(current bounding box.center)]
    \draw (0:0)--++(0:1)--++(90:1)--++(180:1)--++(-90:1);
    \draw[fill=red] (0:0)++(90:0.5) circle (0.1);
    \draw[fill=red] (0:0)++(0:0.5) circle (0.1); 
    \draw[fill=red] (0:0)++(0:1)++(90:0.5) circle (0.1);
    \draw[fill=red] (0:0)++(0:0.5)++(90:1) circle (0.1);
    \draw[red,thick] (0,0.5)--(0.5,1);
    \draw[red,thick] (0.5,0)--(1,0.5);
\end{tikzpicture}
&
\begin{tikzpicture}[baseline=(current bounding box.center)]
    \draw (0:0)--++(0:1)--++(90:1)--++(180:1)--++(-90:1);
    \draw[fill=red] (0:0)++(90:0.5) circle (0.1);
    \draw[fill=white] (0:0)++(0:0.5) circle (0.1); 
    \draw[fill=red] (0:0)++(0:1)++(90:0.5) circle (0.1);
    \draw[fill=white] (0:0)++(0:0.5)++(90:1) circle (0.1);
    \draw[red,thick] (0,0.5)--(1,0.5);
\end{tikzpicture}
&
\begin{tikzpicture}[baseline=(current bounding box.center)]
    \draw (0:0)--++(0:1)--++(90:1)--++(180:1)--++(-90:1);
    \draw[fill=white] (0:0)++(90:0.5) circle (0.1);
    \draw[fill=red] (0:0)++(0:0.5) circle (0.1); 
    \draw[fill=white] (0:0)++(0:1)++(90:0.5) circle (0.1);
    \draw[fill=red] (0:0)++(0:0.5)++(90:1) circle (0.1);
    \draw[red,thick] (0.5,0)--(0.5,1);
\end{tikzpicture}
&
\begin{tikzpicture}[baseline=(current bounding box.center)]
    \draw (0:0)--++(0:1)--++(90:1)--++(180:1)--++(-90:1);
    \draw[fill=red] (0:0)++(90:0.5) circle (0.1);
    \draw[fill=white] (0:0)++(0:0.5) circle (0.1); 
    \draw[fill=white] (0:0)++(0:1)++(90:0.5) circle (0.1);
    \draw[fill=red] (0:0)++(0:0.5)++(90:1) circle (0.1);
    \draw[red,thick] (0,0.5)--(0.5,1);
\end{tikzpicture}
&
\begin{tikzpicture}[baseline=(current bounding box.center)]
    \draw (0:0)--++(0:1)--++(90:1)--++(180:1)--++(-90:1);
    \draw[fill=white] (0:0)++(90:0.5) circle (0.1);
    \draw[fill=red] (0:0)++(0:0.5) circle (0.1); 
    \draw[fill=red] (0:0)++(0:1)++(90:0.5) circle (0.1);
    \draw[fill=white] (0:0)++(0:0.5)++(90:1) circle (0.1);
    \draw[red,thick] (0.5,0)--(1,0.5);
\end{tikzpicture}\\[2 em]
$0$ & $1$ & $1$ & $1$ & $1$\\[2 em]
  \begin{tikzpicture}[baseline=(current bounding box.center)]
    \draw (0:0)--++(0:1)--++(90:1)--++(180:1)--++(-90:1);
    \draw[fill=cyan] (0:0)++(90:0.5) circle (0.1);
    \draw[fill=cyan] (0:0)++(0:0.5) circle (0.1); 
    \draw[fill=cyan] (0:0)++(0:1)++(90:0.5) circle (0.1);
    \draw[fill=cyan] (0:0)++(0:0.5)++(90:1) circle (0.1);
    \draw[cyan,thick] (0,0.5)--(0.5,1);
    \draw[cyan,thick] (0.5,0)--(1,0.5);
\end{tikzpicture}
&
\begin{tikzpicture}[baseline=(current bounding box.center)]
    \draw[black] (0:0)--++(0:1)--++(90:1)--++(180:1)--++(-90:1);
    \draw[fill=cyan] (0:0)++(90:0.5) circle (0.1);
    \draw[fill=white] (0:0)++(0:0.5) circle (0.1); 
    \draw[fill=cyan] (0:0)++(0:1)++(90:0.5) circle (0.1);
    \draw[fill=white] (0:0)++(0:0.5)++(90:1) circle (0.1);
    \draw[cyan,thick] (0,0.5)--(1,0.5);
\end{tikzpicture}
&
\begin{tikzpicture}[baseline=(current bounding box.center)]
    \draw (0:0)--++(0:1)--++(90:1)--++(180:1)--++(-90:1);
    \draw[fill=white] (0:0)++(90:0.5) circle (0.1);
    \draw[fill=cyan] (0:0)++(0:0.5) circle (0.1); 
    \draw[fill=white] (0:0)++(0:1)++(90:0.5) circle (0.1);
    \draw[fill=cyan] (0:0)++(0:0.5)++(90:1) circle (0.1);
    \draw[cyan,thick] (0.5,0)--(0.5,1);
\end{tikzpicture}
&
\begin{tikzpicture}[baseline=(current bounding box.center)]
    \draw (0:0)--++(0:1)--++(90:1)--++(180:1)--++(-90:1);
    \draw[fill=cyan] (0:0)++(90:0.5) circle (0.1);
    \draw[fill=white] (0:0)++(0:0.5) circle (0.1); 
    \draw[fill=white] (0:0)++(0:1)++(90:0.5) circle (0.1);
    \draw[fill=cyan] (0:0)++(0:0.5)++(90:1) circle (0.1);
    \draw[cyan,thick] (0,0.5)--(0.5,1);
\end{tikzpicture}
&
\begin{tikzpicture}[baseline=(current bounding box.center)]
    \draw (0:0)--++(0:1)--++(90:1)--++(180:1)--++(-90:1);
    \draw[fill=white] (0:0)++(90:0.5) circle (0.1);
    \draw[fill=cyan] (0:0)++(0:0.5) circle (0.1); 
    \draw[fill=cyan] (0:0)++(0:1)++(90:0.5) circle (0.1);
    \draw[fill=white] (0:0)++(0:0.5)++(90:1) circle (0.1);
    \draw[cyan,thick] (0.5,0)--(1,0.5);
\end{tikzpicture}
\\[2 em] 
$0$ & $1$ & $1$ & $-1$ & $-1$\\[2 em]
  \begin{tikzpicture}[baseline=(current bounding box.center)]
    \draw (0:0)--++(0:1)--++(90:1)--++(180:1)--++(-90:1);
    \draw[fill=white] (0:0)++(90:0.5) circle (0.1);
    \draw[fill=white] (0:0)++(0:0.5) circle (0.1); 
    \draw[fill=white] (0:0)++(0:1)++(90:0.5) circle (0.1);
    \draw[fill=white] (0:0)++(0:0.5)++(90:1) circle (0.1);
\end{tikzpicture}
&
\begin{tikzpicture}[baseline=(current bounding box.center)]
    \draw (0:0)--++(0:1)--++(90:1)--++(180:1)--++(-90:1);
    \draw[fill=cyan] (0:0)++(90:0.5) circle (0.1);
    \draw[fill=red] (0:0)++(0:0.5) circle (0.1); 
    \draw[fill=cyan] (0:0)++(0:1)++(90:0.5) circle (0.1);
    \draw[fill=red] (0:0)++(0:0.5)++(90:1) circle (0.1);
    \draw[cyan,thick] (0,0.5)--(1,0.5);
    \draw[red,thick] (0.5,0)--(0.5,1);
\end{tikzpicture}
&
\begin{tikzpicture}[baseline=(current bounding box.center)]
    \draw (0:0)--++(0:1)--++(90:1)--++(180:1)--++(-90:1);
    \draw[fill=red] (0:0)++(90:0.5) circle (0.1);
    \draw[fill=cyan] (0:0)++(0:0.5) circle (0.1); 
    \draw[fill=red] (0:0)++(0:1)++(90:0.5) circle (0.1);
    \draw[fill=cyan] (0:0)++(0:0.5)++(90:1) circle (0.1);
    \draw[red,thick] (0,0.5)--(1,0.5);
    \draw[cyan,thick] (0.5,0)--(0.5,1);
\end{tikzpicture}
&
\begin{tikzpicture}[baseline=(current bounding box.center)]
    \draw (0:0)--++(0:1)--++(90:1)--++(180:1)--++(-90:1);
    \draw[fill=cyan] (0:0)++(90:0.5) circle (0.1);
    \draw[fill=red] (0:0)++(0:0.5) circle (0.1); 
    \draw[fill=red] (0:0)++(0:1)++(90:0.5) circle (0.1);
    \draw[fill=cyan] (0:0)++(0:0.5)++(90:1) circle (0.1);
    \draw[cyan,thick] (0,0.5)--(0.5,1);
    \draw[red,thick] (0.5,0)--(1,0.5);
\end{tikzpicture}
&
\begin{tikzpicture}[baseline=(current bounding box.center)]
    \draw (0:0)--++(0:1)--++(90:1)--++(180:1)--++(-90:1);
    \draw[fill=red] (0:0)++(90:0.5) circle (0.1);
    \draw[fill=cyan] (0:0)++(0:0.5) circle (0.1); 
    \draw[fill=cyan] (0:0)++(0:1)++(90:0.5) circle (0.1);
    \draw[fill=red] (0:0)++(0:0.5)++(90:1) circle (0.1);
    \draw[red,thick] (0,0.5)--(0.5,1);
    \draw[cyan,thick] (0.5,0)--(1,0.5);
\end{tikzpicture}\\[2 em] 
$0$ & $1$ & $q$ & $1$ & $q$\\
\end{tabular}
\end{equation}

It remains to show that these coefficients are polynomials in $q$ with positive coefficients, upto an overall multiplicative sign.

\begin{lemma}
\label{lem:positivity}
The coefficients $C^{k,w}_{l,m}$ are polynomials in $q$ with positive coefficients, up to an overall multiplicative sign.
\end{lemma}
\begin{proof}
To demonstrate positivity, we work with weights ~\eqref{puzzweights:1}. We argue that, for a fixed boundary all puzzles have the same overall sign. This follows from the fact that the weights of the tiles with identical labels on all edges are $0$ (i.e., tiles in the first column~\eqref{puzzweights:1}). We make the following observation:

For convenience, we draw holes as black particles. We refer to the tiles in the first column of ~(\ref{puzzweights:1}) as $\mathbf{a}$ vertices, those in the second and third columns as $\mathbf{b}$ vertices, and those in the fourth and fifth column as $\mathbf{c}$ vertices. Pick a blue particle entering from the left and a black particle entering from the bottom. These two particles exit from specific positions. Then in all the configurations, the number of $\mathbf{b}$ type vertices involving the strings corresponding to these two particles is always either odd or even, as shown in the picture below. As the weights of the $\mathbf{a}$ vertices are $0$, this implies that for a fixed boundary all the configurations have either an odd number or an even number of $\mathbf{c}$ tiles. And similar argument works when the black particles exit from the right.
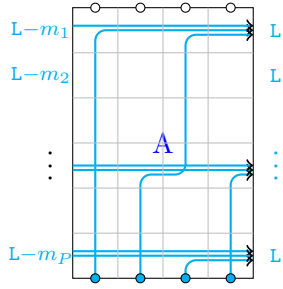
\begin{center}
\begin{tabular}{c@{\hskip 1cm}c}
\begin{tikzpicture}[scale=0.6,baseline=(current bounding box.center)]
\draw[gray](0.5,0.5) rectangle (5.5,5.5);  
\foreach \x in {1,2,3,4}{
\draw[gray] (0.5,\x+0.5)--(5.5,\x+0.5);
};
\foreach \x in {1,2,3,4}{
\draw[gray] (\x+0.5,0.5)--(\x+0.5,5.5);
};
\bbull{2}{0.5}{0.1};
\bbull{4}{5.5}{0.1};
\cbull{0.5}{2}{0.1};
\cbull{5.5}{4}{0.1};
\draw[cyan,thick,rounded corners] (0.5,2)--(2.5,4)--(5.5,4);
\draw[black,thick,rounded corners] (2,0.5)--(4,2.5)--(4,5.5);
\end{tikzpicture}
&
\begin{tikzpicture}[scale=0.6,baseline=(current bounding box.center)]
\draw[gray](0.5,0.5) rectangle (5.5,5.5);  
\foreach \x in {1,2,3,4}{
\draw[gray] (0.5,\x+0.5)--(5.5,\x+0.5);
};
\foreach \x in {1,2,3,4}{
\draw[gray] (\x+0.5,0.5)--(\x+0.5,5.5);
};
\bbull{2}{0.5}{0.1};
\bbull{4}{5.5}{0.1};
\cbull{0.5}{2}{0.1};
\cbull{5.5}{4}{0.1};
\draw[cyan,thick,rounded corners] (0.5,2)--(2.5,2)--(3,2.5)--(3,3.5)--(3.5,4)--(5.5,4);
\draw[black,thick,rounded corners] (2,0.5)--(2,2.5)--(2.5,3)--(3.5,3)--(4,3.5)--(4,5.5);
\end{tikzpicture}
\end{tabular}
\end{center}
\end{proof}
\section{Higher spin Puzzles}
\label{sec:proof_spinhall}
We follow the same argument in deriving the puzzles for the spin Hall--Littlewood functions as we did in the previous section. We shall only focus on the modifications and keep our explanations to a minimum to avoid repeating ourselves.

\subsection{Notation}
Recall that positive integers $\tl,\tm$ represent the spin of the lines attached to the variables $y,z$. We use a white bullet to denote $(0,0)$, for a positive integer $k$, we write $\textcolor{red}{k}$ to denote $(k,0)$, and  $\textcolor{cyan}{k}$  to denote $(0,k)$. For this section, we depict all vertices without any shading. 

\subsection{Lattice model} 
\label{latticemodel:higherspin}
Using the weights of Theorem~\ref{theorem:weightsoftheYBE}, in the case $r=2$, we built the following lattice model:
\begin{equation}
\label{eq:YBEhigerspin}
 \begin{tikzpicture}[scale=0.65,baseline=(current bounding box.center)]
\draw (0:0)++(-60:1) coordinate (A)--++(60:2) coordinate (AA)--++(60:4) coordinate (B)--++(0:5) coordinate (C)--++(-60:4) coordinate (D)--++(-120:6) coordinate (E);
\draw (A)--++(-60:4) coordinate (F)--++(0:5);
\draw (F)--++(60:6) coordinate (G)--++(120:4);
\draw (G)--(D);
\draw[gray] (AA)++(60:-1)--++(-60:4)--++(0:5);
\foreach \x in {1,2,3,4}{
\draw[gray] (F)++(0:\x)--++(60:6)--++(120:4);
\draw[gray] (A)++(-60:\x)--++(60:6)--++(0:5);
\draw[gray] (A)++(60:\x)--++(-60:4)--++(0:5);
};
\foreach \x in {1,2,3,4,5}{
\draw[gray] (A)++(60:\x)--++(-60:4)--++(0:5);
};
\foreach \x in {1,2,3,4,5,6}
{
\path (F)++(0:5.3)++(60:\x-0.5) node[cyan] {$\ss \tl$};
};
\foreach \x in {1,2,3,4}
{\draw[fill=cyan] (A)++(-60:\x-0.5) circle (0.1);
\draw[fill=white] (F)++(0:5)++(60:6)++(120:\x-0.5) circle (0.1);};
\path (F)++(0:0.5)++(-90:0.4) node {$\ss \textcolor{red}{n}$};
\draw[fill=white] (A)++(0:-0.5)++(60:6)++(0:1) circle (0.1);
\path (0,0)++(-60:2.5)++(60:3) node {$\fontsize{25pt}{0}{\bl{A}}$};
\path (F)++(0:2.5)++(60:3) node {$\fontsize{25pt}{0}{\bl{B}}$};
\path (F)++(0:2.5)++(60:6)++(120:2.5) node {$\fontsize{25pt}{0}{\bl{C}}$};
\path (F)++(0:2.5)++(60:6)++(120:2.5) node {$\fontsize{25pt}{0}{\bl{C}}$};
\path (F)++(60:6)++(0:4.5)++(120:4.5) node[red] {$\ss l_{1}$};
\path (F)++(60:6)++(0:3)++(120:4.5) node[red] {$\dots$};
\path (F)++(60:6)++(0:1.5)++(120:4.5) node[red] {$\ss l_{N}$};
\draw[fill=white] (F)++(0:1.5) circle (0.1);
\draw[fill=white] (F)++(0:2.5) circle (0.1);
\draw[fill=white] (F)++(0:3.5) circle (0.1);
\draw[fill=white] (F)++(0:4.5) circle (0.1);
\path (A)++(0:-0.7)++(60:6) node[cyan,rotate=-60] {$\ss {\tl-m_{1}}$};
\path (A)++(0:-0.7)++(60:5) node[cyan,rotate=-60] {$\ss {\tl-m_{2}}$};
\path (A)++(0:-0.7)++(60:4) node[cyan,rotate=-60] {$\ss {\tl-m_{3}}$};
\path (A)++(0:-0.7)++(60:2) node[cyan,rotate=60] {$\ldots $};
\path (A)++(0:-0.7)++(60:1) node[cyan,rotate=-60] {$\ss { \tl-m_{P}}$};
\foreach \x in {1,4}{
\draw[->] (A)++(-60:\x-0.5)++(-120:1)--++(60:0.5);};
\path (A)++(-60:1-0.5)++(-120:1.5) node {$\ss (x_1,1) $};
\path (A)++(-60:2)++(-120:1.5) node[rotate=-60] {$ \ldots$};
\path (A)++(-60:4-0.5)++(-120:1.5) node {$\ss (x_{n},1) $};
\path (F)++(0:0.4)++(-90:1.8) node {$\ss (z_0,\mathtt{M}_{0})$};
\draw[->] (F)++(0:0.5)++(-90:1.5)--++(90:0.5);
\draw[->] (F)++(0:1.5)++(-90:1.5)--++(90:0.5);
\path (F)++(0:1.5)++(-90:1.8)++(0:0.2) node {$\ss (z_{N},\tm) $};
\draw[->] (F)++(0:2.5)++(-90:1.5)--++(90:0.5);
\draw[->] (F)++(0:4.7)++(-90:1.5)--++(90:0.5);
\draw[->] (F)++(0:3.5)++(-90:1.5)--++(90:0.5);
\path (F)++(0:3.3)++(-90:1.8) node {$\dots $};
\path (F)++(0:4.7)++(-90:1.8) node {$\ss (z_1,\mathtt{M}) $};
\foreach\x in {1,6}{
\draw[->] (A)++(60:\x-0.5)++(120:2)--++(-60:0.5);};
\path (A)++(60:0.5)++(120:2.5)++(0:-0.4) node {$\ss ( q^{-\tl +1}y^{-1}_{P},\tl)$};
\path (A)++(60:2.5)++(120:2.5)++(90:0.5) node[rotate=60] {$ \dots$};
\path (A)++(60:5.5)++(120:2.5)++(0:-0.4) node {$\ss (q^{-\tl+1}y^{-1}_1,\tl)$};
\end{tikzpicture}=
 \begin{tikzpicture}[scale=0.65,baseline=(current bounding box.center)]
\draw (0:0)++(-60:1) coordinate (A)--++(60:2) coordinate (AA)--++(60:4) coordinate (B)--++(0:5) coordinate (C)--++(-60:4) coordinate (D)--++(-120:4) coordinate (EE)--++(-120:2) coordinate (E);
\draw (A)--++(-60:4) coordinate (F)--++(0:5);
\draw (A)--++(0:5) coordinate (G);
\draw (G)--++(60:6);
\draw (G)--++(-60:4);
\path (F)++(0:0.5)++(-90:0.4) node {$\ss \textcolor{red}{n}$};
\foreach \x in {1,2,3,4}{
\draw[gray] (A)++(-60:\x)--++(0:5)--++(60:6);
};
\foreach \x in {1,2,3,4,5}{
\draw[gray] (A)++(-60:4)++(0:\x)--++(120:4)--++(60:6);
};
\foreach \x in {1,2,3,4,5}{
\draw[gray] (A)++(60:\x)--++(0:5)--++(-60:4);
};
\foreach \x in {1,2,3,4,5,6}
{
\path (F)++(0:5.3)++(60:\x-0.5) node[cyan] {$\ss \tl$};
};
\foreach \x in {1,2,3,4}
{\draw[fill=cyan] (A)++(-60:\x-0.5) circle (0.1);
\draw[fill=white] (F)++(0:5)++(60:6)++(120:\x-0.5) circle (0.1);};
\draw[fill=white] (A)++(0:-0.5)++(60:6)++(0:1) circle (0.1);
\path (0,0)++(0:3)++(60:3) node {$\fontsize{25pt}{0}{\bl{E}}$};
\path (F)++(0:2.5)++(120:2.5) node {$\fontsize{25pt}{0}{\bl{D}}$};
\path (0,0)++(0:5)++(-60:2.5)++(60:3) node {$\fontsize{25pt}{0}{\bl{F}}$};
\path (F)++(60:6)++(0:5)++(120:4.5) node[red] {$\ss l_{1}$};
\path (F)++(60:6)++(0:3.5)++(120:4.5) node[red] {$\dots$};
\path (F)++(60:6)++(0:2)++(120:4.5) node[red] {$\ss l_{N}$};
\draw[fill=white] (F)++(0:1.5) circle (0.1);
\draw[fill=white] (F)++(0:2.5) circle (0.1);
\draw[fill=white] (F)++(0:3.5) circle (0.1);
\draw[fill=white] (F)++(0:4.5) circle (0.1);
\path (A)++(0:-0.7)++(60:5.5) node[cyan]{$\ss \tl-m_{1}$};
\path (A)++(0:-0.7)++(60:4.5) node[cyan]{$\ss \tl-m_{2}$};
\path (A)++(0:-0.7)++(60:3.5) node[cyan]{$\ss \tl-m_{3}$};
\path (A)++(0:-0.7)++(60:2) node[cyan,rotate=60]{$\ldots$};
\path (A)++(0:-0.7)++(60:0.5) node[cyan]{$\ss \tl-m_{P}$};
\path (F)++(0:0.3)++(-60:1.8) node {$\ss (z_0,\mathtt{M}_{0})$};
\draw[->] (F)++(0:0.5)++(-60:1.5)--++(120:0.5);
\draw[->] (F)++(0:1.5)++(-60:1.5)--++(120:0.5);
\draw[->] (F)++(0:2.5)++(-60:1.5)--++(120:0.5);
\draw[->] (F)++(0:3.5)++(-60:1.5)--++(120:0.5);
\draw[->] (F)++(0:4.5)++(-60:1.5)--++(120:0.5);
\path (F)++(0:3.3)++(-60:1.8) node {$ \dots $};
\path (F)++(0:4.5)++(-60:1.8) node {$\ss (z_1,\tm) $};
\path (F)++(0:1.5)++(-60:1.8)++(0:0.1) node {$\ss (z_{N},\tm) $};
\foreach \x in {1,4}{
\draw[->] (A)++(-60:\x-0.5)++(0:-0.3)++(0:-1)--++(0:0.5);};
\path (A)++(-60:1-0.5)++(0:-2) node {$\ss (x_1,1) $};
\path (A)++(-60:2)++(0:-2) node[rotate=120] {$ \ldots$};
\path (A)++(-60:4-0.5)++(0:-2.2) node {$\ss (x_{n},1) $};
\foreach \x in {1,6}{
\draw[->] (A)++(60:\x-0.5)++(0:-0.5)++(0:-1.5)--++(0:0.5);};
\path (A)++(60:1-0.5)++(0:-3.5) node {$\ss ( q^{-\tl+1}y^{-1}_{P},\tl)$};
\path (A)++(60:2.5)++(0:-3.5)++(60:0.5) node[rotate=60] {$ \dots$};
\path (A)++(60:6-0.5)++(0:-3.5) node {$\ss (q^{-\tl +1} y^{-1}_1,\tl)$};
\end{tikzpicture}
\end{equation}
where $\sum^{P}_{i=1}m_{i}=\sum^{N}_{i=1}l_{i}=n$.

\subsection{Region \texorpdfstring{$A$}{A}}
\label{subsec:spin_regionA}
Only blue particles enter $\bl{A}$ through the left and bottom boundaries. Given the boundary of $\bl{C}$, all the blue particles have to enter $\bl{B}$. In total, there are $(\tl) (P)$ blue particles entering the $\bl{A}$. Recall that the horizontal edges in $\bl{A}$ and $\bl{B}$ can carry atmost $\tl$ particles. This implies that $\tl$ number of particles exit from $\bl{A}$ through every row. A typical configuration in $\bl{A}$ is given below.

\

\begin{align*}
\begin{tikzpicture}[scale=0.6,baseline=2pt]
\draw (0:0) coordinate (A)--++(90:6) coordinate (B)--++(0:4) coordinate (C)-- ++(-90:6) coordinate (D)--cycle;
\draw[part2,->,thick,rounded corners] (A)++(0:0.5)--++(90:5.5)--++(0:3.5);
\draw[part2,->,thick,rounded corners] (A)++(0:2.5)--++(90:0.4)--++(0:1.5);
\draw[part2,->,thick,rounded corners] (A)++(0:3.5)--++(90:2.3)--++(0:0.5);
\draw[part2,->,thick,rounded corners] (A)++(0:1.5)--++(90:2.3)--++(0:1)--++(90:3.1)--++(0:1.5);
\draw[part2,->,thick,rounded corners] (A)++(90:2.4)--++(0:4);
\draw[part2,->,thick,rounded corners] (A)++(90:2.5)--++(0:4);
\draw[part2,->,thick,rounded corners] (A)++(90:0.6)--++(0:4);
\draw[part2,->,thick,rounded corners] (A)++(90:0.5)--++(0:4);
\draw[part2,->,thick,rounded corners] (A)++(90:5.6)--++(0:4);
\path (A)++(0:-0.7)++(90:5.5) node[cyan] {$\ss \tl-m_{1}$};
\path (A)++(0:-0.7)++(90:4.5) node[cyan] {$\ss \tl-m_{2}$};
\path (A)++(0:-0.5)++(90:2.5) node[rotate=90] {$\ldots$};
\path (A)++(0:-0.7)++(90:0.5) node[cyan] {$\ss \tl-m_{P}$};
\path (D) ++(0:0.5)++(90:0.5) node[rotate=0,cyan] {$\ss \tl$};
\path (D) ++(0:0.5)++(90:2.5) node[rotate=90,cyan] {$\ldots$};
\path (D) ++(0:0.5)++(90:4.5) node[cyan] {$\ss \tl$};
\path (D) ++(0:0.5)++(90:5.5) node[cyan] {$\ss \tl$};
\path (A) ++(0:2)++(90:3) node {$\bl{A} $};
\foreach\x in {1,2,3,4,5}{
\draw[lightgray] (0:0)++(90:\x)--++(0:4);};
\foreach\x in {1,2,3}{
\draw[lightgray] (0:0)++(0:\x)--++(90:6);};
\foreach\x in {1,2,3,4}{
\draw[fill=cyan] (0:0)++(0:\x-0.5) circle (0.1);
\draw[fill=white] (0:0)++(0:\x-0.5)++(90:6) circle (0.1);};
\end{tikzpicture}    
\end{align*}

\

By complementing the particles on the vertical edges with $\tl$, and swapping $1$ to $0$ and vice versa on the horizontal edges, we get the following:
\begin{equation}
\label{Spin:partitionfunctionAcomplemented}
\mathcal{Z}_{\bl{A}}(x_1,\dots,x_n;y_1,\dots,y_{P};q^{-\tl})=
\begin{tikzpicture}[scale=0.6,baseline=(current bounding box.center)]
\draw (0:0) coordinate (A)--++(90:6) coordinate (B)--++(0:4) coordinate (C)-- ++(-90:6) coordinate (D)--cycle;
\path (A)++(0:-0.7)++(90:5.5) node[black] {$\ss m_{1}$};
\path (A)++(0:-0.7)++(90:4.5) node[black] {$\ss m_{2}$};
\path (A)++(0:-0.5)++(90:2.5) node[rotate=90] {$\ldots$};
\path (A)++(0:-0.7)++(90:0.5) node[black] {$\ss m_{P}$};
\path (A) ++(0:2)++(90:3) node {$\bl{A} $};
\foreach\x in {1,2,3,4,5,6}{
\draw[fill=white] (0:0)++(0:4)++(90:\x-0.5) circle (0.1);};
\foreach\x in {1,2,3,4,5}{
\draw[lightgray] (0:0)++(90:\x)--++(0:4);};
\foreach\x in {1,2,3}{
\draw[lightgray] (0:0)++(0:\x)--++(90:6);};
\foreach\x in {1,2,3,4}{
\draw[fill=white] (0:0)++(0:\x-0.5) circle (0.1);
\draw[fill=black] (0:0)++(0:\x-0.5)++(90:6) circle (0.1);};
\end{tikzpicture}
\end{equation}
with weights:
\begin{equation}
\label{weightsforregionA}
\begin{tabular}{ccccc}
\begin{tikzpicture}[scale=1,baseline=-2pt]
 \draw[gray] (0,0) rectangle (1,1); 
\node at (0.5,-0.3) {$\ss A$};
\node at (-0.3,0.5) {$\ss B$};
\node at (0.5,1.3) {$\ss C$};
\node at (1.3,0.5) {$\ss D$};
\draw[->] (0.5,-1) node[below] {$\ss x$}--(0.5,-0.7);
\draw[->] (-1,0.5) node[left] {$\ss {\dfrac{q}{y q^{\tl}}}$}--(-0.7,0.5);
\end{tikzpicture}
&
\begin{tikzpicture}[scale=1,baseline=-2pt]
 \draw[gray] (0,0) rectangle (1,1); 
\draw[->] (0,0.4)--(1,0.4);
\draw[->] (0,0.5)--(1,0.5);
\draw[->] (0,0.6)--(1,0.6);
\node at (0.5,-0.3) {$\ss 1$};
\node at (-0.3,0.5) {$\ss m$};
\node at (0.5,1.3) {$\ss 1$};
\node at (1.5,0.5) {$\ss m$};
\end{tikzpicture}
&
\begin{tikzpicture}[scale=1,baseline=-2pt]
 \draw[gray] (0,0) rectangle (1,1); 
\draw[->] (0,0.4)--(1,0.4);
\draw[->] (0,0.5)--(1,0.5);
\draw[->,rounded corners] (0,0.6)--(0.5,0.6)--(0.5,1);
\node at (0.5,-0.3) {$\ss 0$};
\node at (-0.3,0.5) {$\ss m$};
\node at (0.5,1.3) {$\ss 1$};
\node at (1.5,0.5) {$\ss m-1$};
\end{tikzpicture}
&
\begin{tikzpicture}[scale=1,baseline=-2pt]
 \draw[gray] (0,0) rectangle (1,1); 
\draw[->] (0,0.5)--(1,0.5);
\draw[->] (0,0.6)--(1,0.6);
\draw[->,rounded corners] (0.5,0)--(0.5,0.4)--(1,0.4);
\node at (0.5,-0.3) {$\ss 1$};
\node at (-0.3,0.5) {$\ss m$};
\node at (0.5,1.3) {$\ss 0$};
\node at (1.5,0.5) {$\ss m+1$};
\end{tikzpicture}
&
\begin{tikzpicture}[scale=1,baseline=-2pt]
\draw[gray] (0,0) rectangle (1,1); 
\draw[->,rounded corners] (0,0.5)--(1,0.5);
\draw[->,rounded corners] (0,0.6)--(0.5,0.6)--(0.5,1);
\draw[->,rounded corners] (0.5,0)--(0.5,0.4)--(1,0.4);
\node at (0.5,-0.3) {$\ss 0$};
\node at (-0.3,0.5) {$\ss m$};
\node at (0.5,1.3) {$\ss 0$};
\node at (1.3,0.5) {$\ss m$};
\end{tikzpicture}\\
& $\ss \dfrac{1-q^{m}xy}{(-y q^{\tl})(x-y^{-1}q^{-\tl})}$
& $\ss \dfrac{1-q^{m}}{(-y q^{\tl})(x-y^{-1}q^{-\tl})}$
& $\ss \dfrac{(1-q^{m-\tl})x}{x-y^{-1}q^{-\tl}}$
& $\ss \dfrac{x-q^{m} y^{-1}q^{-\tl} }{x-y^{-1} q^{-\tl}} $
\end{tabular}
\end{equation}

Then $\mathcal{Z}_{\bl{A}}$ is a function in $x$ variables, $y$ variables and also in $q^{-\tl}$. After substituting $y_{i}=s$ and $q^{-\tl}=s^{2}$, from ~\cref{def:spinHl}, we get that:
\begin{equation}
\mathcal{Z}_{\bl{A}}=(-s)^{\sum^{P}_{i=1} i m_{i}} \hspace{1mm} \mathrm{F}_{\bm}(x_1,\dots,x_n;s)
\end{equation}

\subsection{Region \texorpdfstring{$B$}{B}}
 From the discussion concerning $\bl{A}$, both left and right boundaries of $\bl{B}$ are fixed and are identical. We show that the entire region is frozen. To see this, we begin with the last row. In the last row, we have $\tl$ blue particles entering from the left and the same exiting through the right. As no other blue particles are entering from the bottom, those blue particles that entered from the left are the ones that exit from that row.
 
 As each of the vertical edges can carry atmost $\tl$ particles, we get that the entire last row is frozen. Using the same reasoning for every row, we conclude that the entire region is frozen. The weights of the vertex where the vertical labels are both $(0,\tl)$ are independent of the top and bottom labels. We chose the normalisation so that the weight of $\bl{B}$ is trivial and equal to $1$.

\begin{align*}
\begin{tikzpicture}[scale=0.6]
\draw[white] (0:0)--++(0:1)--++(90:6)--++(180:1)--++(-90:6);
\draw (0:0)--++(90:2) coordinate (A)--++(90:4)--++(0:5)--++(-90:6)--++(180:5);
\foreach\x in {1,2,3,4}{
\draw[lightgray] (0:0)++(0:\x)--++(90:6);
\draw[lightgray] (0:0)++(90:\x)--++(0:5);
};
\foreach\x in {1,2,3,4,5}{
\draw[lightgray] (0:0)++(90:\x)--++(0:5);
};
\foreach\x in {1,2,3,4,5,6}{
\draw[part2,->,thick] (0:0)++(90:\x-0.5)--++(0:5);
\draw[part2,->,thick] (0:0)++(90:\x-0.4)--++(0:5);
\draw[part2,->,thick] (0:0)++(90:\x-0.6)--++(0:5);
};
\draw[red,thick] (0:0)++(0:0.4)--++(90:6);
\draw[red,thick] (0:0)++(0:0.5)--++(90:6);
\draw[red,thick] (0:0)++(0:0.6)--++(90:6);
\path (0:0)++(0:0.5)++(-90:0.5) node {$\textcolor{red}{n}$};
\path (0:0)++(0:0.5)++(90:6.5) node {$\textcolor{red}{n}$};
\end{tikzpicture}
\end{align*}

\subsection{Region \texorpdfstring{$C$}{C}}
\label{subsec:spin1regionC}
Similar to region $\bl{B}$, we normalise the weights of region $\bl{C}$ such that the weight of the vertex $\begin{tikzpicture}[scale=0.5,baseline=-2pt]
\draw[gray] (0,0) rectangle (1,1); 
\draw[->,rounded corners] (0,0.5)--(0.4,0.5)--(0.4,1);
\draw[->,rounded corners] (0.5,0)--(0.5,1);
\draw[->,rounded corners] (0.6,0)--(0.6,0.5)--(1,0.5);
\node at (0.5,-0.3) {$\ss m$};
\node at (-0.3,0.5) {$\ss 1$};
\node at (0.5,1.3) {$\ss m$};
\node at (1.3,0.5) {$\ss 1$};
\end{tikzpicture}$

As $\bl{B}$ is frozen, the bottom boundary of $\bl{C}$ is the same as the bottom boundary of $\bl{B}$. Recall that particles from $\bl{A}$ do not enter $\bl{C}$, this fixes the labels on the vertical edges of the left boundary of $\bl{C}$ to be all holes. This implies that in $\bl{C}$ we only have red particles. We have $n$ red particles enter $\bl{C}$ from the bottom boundary of the first column. Then we compute $\mathcal{Z}_{\bl{C}}$ as partition function:

\begin{equation}
\label{Spin:partitionfunctionC}
\mathcal{Z}_{\bl{C}}(x_1,\dots,x_n;z_{0},\dots,z_{N};q^{-\tm})=
\begin{tikzpicture}[scale=0.6,baseline=(current bounding box.center)]
\draw (0:0) coordinate (A)--++(90:4) coordinate (B)--++(0:6) coordinate (C)-- ++(-90:4) coordinate (D)--cycle;
\path (A)++(90:4.7)++(0:5.5) node[black] {$\ss l_{1}$};
\path (A)++(90:4.7)++(0:4.5) node[black] {$\ss l_{2}$};
\path (A)++(90:4.5)++(0:3) node[rotate=0] {$\ldots$};
\path (A)++(90:4.7)++(0:1.5) node[black] {$\ss l_{N}$};
\draw[fill=white] (0:0)++(0:0.5)++(90:4) circle (0.1);
\foreach\x in {2,3,4,5,6}{
\draw[fill=white] (0:0)++(0:\x-0.5) circle (0.1);};
\path (0:0.5)++(-90:0.5) node {$\ss n$};
\foreach\x in {1,2,3,4,5}{
\draw[lightgray] (0:0)++(0:\x)--++(90:4);};
\foreach\x in {1,2,3}{
\draw[lightgray] (0:0)++(90:\x)--++(0:6);};
\foreach\x in {1,2,3,4}{
\draw[fill=white] (0:0)++(90:\x-0.5) circle (0.1);
\draw[fill=white] (0:0)++(90:\x-0.5)++(0:6) circle (0.1);};
\draw[->] (-1,0.5) node[left] {$\ss x_n$}--(-0.5,0.5);
\node[rotate=90] at (-0.75,2) {$\dots$} ;
\draw[->] (-1,3.5) node[left] {$\ss x_1$}--(-0.5,3.5);
\draw[->] (0.5,-1.5) node[below] {$\ss z_0$}--(0.5,-1);
\draw[->] (1.5,-1.5) node[below] {$\ss z_{N}$}--(1.5,-1);
\node[rotate=0] at (3.5,-1.25) {$\dots$} ;
\draw[->] (5.5,-1.5) node[below] {$\ss z_1$}--(5.5,-1);
\end{tikzpicture}
\end{equation}

with weights:
\begin{equation}
\label{weightsforregionc}
\begin{tabular}{ccccc}
\begin{tikzpicture}[scale=1,baseline=-2pt]
 \draw[gray] (0,0) rectangle (1,1); 
\node at (0.5,-0.3) {$\ss A$};
\node at (-0.3,0.5) {$\ss B$};
\node at (0.5,1.3) {$\ss C$};
\node at (1.3,0.5) {$\ss D$};
\draw[->] (0.5,-1) node[below] {$\ss z$}--(0.5,-0.7);
\draw[->] (-1,0.5) node[left] {$\ss x$}--(-0.7,0.5);
\end{tikzpicture}
&
\begin{tikzpicture}[scale=1,baseline=-2pt]
\draw[gray] (0,0) rectangle (1,1); 
\draw[->,rounded corners] (0.4,0)--(0.4,1);
\draw[->,rounded corners] (0.5,0)--(0.5,1);
\draw[->,rounded corners] (0.6,0)--(0.6,1);
\node at (0.5,-0.3) {$\ss m$};
\node at (-0.3,0.5) {$\ss 0$};
\node at (0.5,1.3) {$\ss m$};
\node at (1.3,0.5) {$\ss 0$};
\end{tikzpicture}
&
\begin{tikzpicture}[scale=1,baseline=-2pt]
\draw[gray] (0,0) rectangle (1,1); 
\draw[->,rounded corners] (0.4,0)--(0.4,1);
\draw[->,rounded corners] (0.5,0)--(0.5,1);
\draw[->,rounded corners] (0.6,0)--(0.6,0.5)--(1,0.5);
\node at (0.5,-0.3) {$\ss m$};
\node at (-0.3,0.5) {$\ss 0$};
\node at (0.5,1.3) {$\ss m-1$};
\node at (1.3,0.5) {$\ss 1$};
\end{tikzpicture}
&
\begin{tikzpicture}[scale=1,baseline=-2pt]
\draw[gray] (0,0) rectangle (1,1); 
\draw[->,rounded corners] (0,0.5)--(0.4,0.5)--(0.4,1);
\draw[->,rounded corners] (0.5,0)--(0.5,1);
\draw[->,rounded corners] (0.6,0)--(0.6,1);
\node at (0.5,-0.3) {$\ss m$};
\node at (-0.3,0.5) {$\ss 1$};
\node at (0.5,1.3) {$\ss m+1$};
\node at (1.3,0.5) {$\ss 0$};
\end{tikzpicture}
&
\begin{tikzpicture}[scale=1,baseline=-2pt]
\draw[gray] (0,0) rectangle (1,1); 
\draw[->,rounded corners] (0,0.5)--(0.4,0.5)--(0.4,1);
\draw[->,rounded corners] (0.5,0)--(0.5,1);
\draw[->,rounded corners] (0.6,0)--(0.6,0.5)--(1,0.5);
\node at (0.5,-0.3) {$\ss m$};
\node at (-0.3,0.5) {$\ss 1$};
\node at (0.5,1.3) {$\ss m$};
\node at (1.3,0.5) {$\ss 1$};
\end{tikzpicture}\\
& $\ss \dfrac{(- z q^{\tm})(1-q^{m}x z^{-1} q^{-\tm})}{x-z}$
& $\ss{ \dfrac{(1-q^{m})x}{x-z}}$
& $\ss \dfrac{(-z q^{\tm})(1-q^{m-\tm})}{x-z} $
& $\ss \dfrac{x-zq^{m}}{x-z}$\\[2 em]
\end{tabular}
\end{equation}
\begin{lemma}
\label{lemma:regionC_partition function}
   At $z_{0}=0$, $z_{i}=s$ and $q^{-\tm}=s^{2}$ we have:
   \begin{equation}
     \mathcal{Z}_{\bl{C}}=(-s)^{-\sum_{i}i l_{i}}(q;q)_{n}\left(\prod^{N}_{i=1}\dfrac{(s^{2};q)_{l_{i}}}{(q;q)_{l_{i}}}\right)\mathrm{F}_{\bll}(x_{1},\dots,x_{n};s)
   \end{equation}
\end{lemma}
\begin{proof}
Observe that $n$ particles enter through the first column and none exit through it. As there are $n$ rows, this enforces that a particle turns right in every row of the first column. Therefore, as we set $z_{0}=0$, the overall weight of the first column is $(q;q)_{n}$.
\begin{equation}
\label{Spin:gfunction-z0=0}
\mathcal{Z}_{\bl{C}}(x_1,\dots,x_n;0,z_{1},\dots,z_{N};q^{-\tm})=(q;q)_{n} \hspace{3mm}\times
\begin{tikzpicture}[scale=0.6,baseline=(current bounding box.center)]
\draw (0:0) coordinate (A)--++(90:4) coordinate (B)--++(0:5) coordinate (C)-- ++(-90:4) coordinate (D)--cycle;
\path (A)++(90:4.7)++(0:4.5) node[black] {$\ss l_{1}$};
\path (A)++(90:4.7)++(0:3.5) node[black] {$\ss l_{2}$};
\path (A)++(90:4.5)++(0:2) node[rotate=0] {$\ldots$};
\path (A)++(90:4.7)++(0:0.5) node[black] {$\ss l_{N}$};
\foreach\x in {0,1,2,3,4}{
\draw[lightgray] (0:0)++(0:\x)--++(90:4);};
\foreach\x in {1,2,3}{
\draw[lightgray] (0:0)++(90:\x)--++(0:5);};
\foreach\x in {1,2,3,4}{
\draw[fill=black] (0:0)++(90:\x-0.5) circle (0.1);
\draw[fill=white] (0:0)++(90:\x-0.5)++(0:5) circle (0.1);};
\draw[->] (-1,0.5) node[left] {$\ss x_n$}--(-0.5,0.5);
\node[rotate=90] at (-0.75,2) {$\dots$};
\draw[->] (-1,3.5) node[left] {$\ss x_1$}--(-0.5,3.5);
\draw[->] (0.5,-1.5) node[below] {$\ss z_{N}$}--(0.5,-1);
\node[rotate=0] at (2.5,-1.25) {$\dots$} ;
\draw[->] (4.5,-1.5) node[below] {$\ss z_1$}--(4.5,-1);
\end{tikzpicture}
\end{equation}
Let us redraw the above lattice by complementing the particles on the horizontal edges, and also performing the substitutions $z_{i}=s$ and $q^{-\tm}=s^{2}$. Then
\begin{equation}
\label{Spin:gfunction-z0=0-complemented}
\mathcal{Z}_{\bl{C}}(x_1,\dots,x_n;0,s,\dots,s;s^{2})=(q;q)_{n} \hspace{3mm}\times
\begin{tikzpicture}[scale=0.6,baseline=(current bounding box.center)]
\draw (0:0) coordinate (A)--++(90:4) coordinate (B)--++(0:5) coordinate (C)-- ++(-90:4) coordinate (D)--cycle;
\path (A)++(90:4.7)++(0:4.5) node[black] {$\ss l_{1}$};
\path (A)++(90:4.7)++(0:3.5) node[black] {$\ss l_{2}$};
\path (A)++(90:4.5)++(0:2) node[rotate=0] {$\ldots$};
\path (A)++(90:4.7)++(0:0.5) node[black] {$\ss l_{N}$};
\foreach\x in {0,1,2,3,4}{
\draw[lightgray] (0:0)++(0:\x)--++(90:4);};
\foreach\x in {1,2,3}{
\draw[lightgray] (0:0)++(90:\x)--++(0:5);};
\foreach\x in {1,2,3,4}{
\draw[fill=white] (0:0)++(90:\x-0.5) circle (0.1);
\draw[fill=black] (0:0)++(90:\x-0.5)++(0:5) circle (0.1);};
\draw[->] (-1,0.5) node[left] {$\ss x_n$}--(-0.5,0.5);
\node[rotate=90] at (-0.75,2) {$\dots$};
\draw[->] (-1,3.5) node[left] {$\ss x_1$}--(-0.5,3.5);
\draw[->] (0.5,-1.5) node[below] {$\ss s$}--(0.5,-1);
\node[rotate=0] at (2.5,-1.25) {$\dots$} ;
\draw[->] (4.5,-1.5) node[below] {$\ss s$}--(4.5,-1);
\end{tikzpicture}
\end{equation}
with weights:
\begin{equation}
\label{weights:gtof}
\begin{tabular}{ccccc}
\begin{tikzpicture}[scale=1,baseline=-2pt]
 \draw[gray] (0,0) rectangle (1,1); 
\node at (0.5,-0.3) {$\ss A$};
\node at (-0.3,0.5) {$\ss B$};
\node at (0.5,1.3) {$\ss C$};
\node at (1.3,0.5) {$\ss D$};
\draw[->] (0.5,-1) node[below] {$\ss s$}--(0.5,-0.7);
\draw[->] (-1,0.5) node[left] {$\ss x$}--(-0.7,0.5);
\end{tikzpicture}
&
\begin{tikzpicture}[scale=1,baseline=-2pt]
\draw[gray] (0,0) rectangle (1,1); 
\draw[<-,rounded corners] (0,0.5)--(0.4,0.5)--(0.4,0);
\draw[->,rounded corners] (0.5,0)--(0.5,1);
\draw[<-,rounded corners] (0.6,1)--(0.6,0.5)--(1,0.5);
\node at (0.5,-0.3) {$\ss m$};
\node at (-0.3,0.5) {$\ss 1$};
\node at (0.5,1.3) {$\ss m$};
\node at (1.3,0.5) {$\ss 1$};
\end{tikzpicture}
&
\begin{tikzpicture}[scale=1,baseline=-2pt]
\draw[gray] (0,0) rectangle (1,1); 
\draw[->,rounded corners] (0.5,0)--(0.5,1);
\draw[->,rounded corners] (0.6,0)--(0.6,1);
\draw[->,rounded corners] (0.4,0)--(0.4,0.5)--(0,0.5);
\node at (0.5,-0.3) {$\ss m$};
\node at (-0.3,0.5) {$\ss 1$};
\node at (0.5,1.3) {$\ss m-1$};
\node at (1.3,0.5) {$\ss 0$};
\end{tikzpicture}
&
\begin{tikzpicture}[scale=1,baseline=-2pt]
\draw[gray] (0,0) rectangle (1,1); 
\draw[->,rounded corners] (1,0.5)--(0.6,0.5)--(0.6,1);
\draw[->,rounded corners] (0.4,0)--(0.4,1);
\draw[->,rounded corners] (0.5,0)--(0.5,1);
\node at (0.5,-0.3) {$\ss m$};
\node at (-0.3,0.5) {$\ss 0$};
\node at (0.5,1.3) {$\ss m+1$};
\node at (1.3,0.5) {$\ss 1$};
\end{tikzpicture}
&
\begin{tikzpicture}[scale=1,baseline=-2pt]
\draw[gray] (0,0) rectangle (1,1); 
\draw[->,rounded corners] (0.4,0)--(0.4,1);
\draw[->,rounded corners] (0.5,0)--(0.5,1);
\draw[->,rounded corners] (0.6,0)--(0.6,1);
\node at (0.5,-0.3) {$\ss m$};
\node at (-0.3,0.5) {$\ss 0$};
\node at (0.5,1.3) {$\ss m$};
\node at (1.3,0.5) {$\ss 0$};
\end{tikzpicture}\\
& $\ss \dfrac{(-s^{-1})(1-q^{m}xs)}{x-s}$
& $\ss{ \dfrac{(1-q^{m})x}{x-s}}$
& $\ss \dfrac{(-s^{-1})(1-q^{m}s^{2})}{x-s} $
& $\ss \dfrac{x-sq^{m}}{x-s}$\\[2 em]
\end{tabular}
\end{equation}

The weights in the above line differs from those of ~\cref{def:spinHl} in the weights of the second and third vertices. We define conjugated vertices as follows:
\begin{equation}
\label{weights:conjguated_region_C}
\begin{tabular}{ccccc}
\begin{tikzpicture}[scale=1,baseline=(current bounding box.center)]
 \draw[gray] (0,0) rectangle (1,1); 
\node at (0.5,-0.3) {$\ss A$};
\node at (-0.3,0.5) {$\ss B$};
\node at (0.5,1.3) {$\ss C$};
\node at (1.3,0.5) {$\ss D$};
\draw[fill=black] (0.4,0.4) rectangle (0.6,0.6);
\draw[->] (-1,0.5) node[left] {$\ss x$}--(-0.7,0.5);
\end{tikzpicture}
=$\dfrac{(q;q)_{C} (s^{2};q)_{A}}{(q;q)_{A} (s^{2};q)_{C}}$ \hspace{
3mm}
\begin{tikzpicture}[scale=1,baseline=(current bounding box.center)]
 \draw[gray] (0,0) rectangle (1,1); 
\node at (0.5,-0.3) {$\ss A$};
\node at (-0.3,0.5) {$\ss B$};
\node at (0.5,1.3) {$\ss C$};
\node at (1.3,0.5) {$\ss D$};
\draw[->] (-1,0.5) node[left] {$\ss x$}--(-0.7,0.5);
\end{tikzpicture}
\end{tabular}
\end{equation}
Then the partition function of the lattice in equation~\eqref{Spin:gfunction-z0=0-complemented} with the conjugated vertices is equal to
\[
(-s^{-1})^{\sum^{N}_{i=1}i l_{i}}\mathrm{F}_{l}(x_{1},\dots,x_{n};s).
\]
We then conclude that:
\[
(-s)^{\sum^{N}_{i=1}i l_{i}}\mathcal{Z}_{\bl{C}}(x_{1},\dots,x_{N};0,s,\dots,s;s^{2})=\prod^{N}_{i=1}\dfrac{(s^{2};q)_{l_{i}}}{(q;q)_{l_{i}}}\mathrm{F}_{\bll}(x_{1},\dots,x_{n};s)
\]

\end{proof}

\subsection{Region \texorpdfstring{$D$}{D}}
We assume that $\left|\dfrac{1-q^{\tl}yz}{1-yz} \cdot \dfrac{(1-x y)}{1-x y q^{\tl}} \right|<1$. With these assumptions and applying the same argument as in~\cref{sub:spin1regD}, we conclude that $\bl{D}$ is entirely frozen, where all the blue particles travel across into $\bl{F}$, with a trivial weight. 
\subsection{Region \texorpdfstring{$F$}{F}}
Let us begin by analysing the boundaries of $\bl{F}$. As a consequence of $\bl{D}$ being frozen, the bottom boundary of $\bl{F}$ is fixed i.e., a blue particle enters from the bottom boundary in every column. Observe that $\tl$ particles exit from every row through the right boundaries. This implies that the labels of the edges of the left boundary are of the form $(0,\tl-k_{i})$ where $\sum_{i}k_{i}=n$. Observe that these boundaries as precisely the same as the boundaries of $\bl{A}$. As the weights used in $\bl{A}$ and $\bl{F}$ are same, $\mathcal{Z}_{\bl{F}}(x_1,\dots,x_n;y_{1},\dots,y_{p};q^{-\tl})$ is a function in $x$ variables, $y$ variables and also in $q^{-\tl}$. After substituting $y_{i}=s$ and $q^{-\tl}=s^{2}$, from ~\cref{def:spinHl} and the argument used in ~\cref{subsec:spin_regionA}, we get that:
\begin{equation}
(-s^{-1})^{\sum^{P}_{i=1} i k_{i}} \mathcal{Z}_{\bl{F}}(x_{1},\dots,x_{n};s,\dots,s;s^{2})=\mathrm{F}_{\bk}(x_1,\dots,x_n;s)
\end{equation}

\subsection{Region \texorpdfstring{$E$}{E}} Finally, we are left with $\bl{E}$. All boundaries are fixed except for the right boundary. However, this boundary is in the same form as the left boundary of $\bl{F}$. For an arbitrary boundary $\mathbf{k}$, $\bl{E}$ is of the following:
\begin{align}
\label{puzzle:higherspin-uncomplemented-blueparticles}
 \begin{tikzpicture}[scale=0.6,baseline=(current bounding box.center)]
 \draw (0,0)--++(0:5)--++(90:6)--++(180:5)--++(-90:6);
  \foreach\x in {1,2,3,4,5}{
 \draw[gray] (0:\x)--++(90:6);
  \draw[gray] (90:\x)--++(0:5);
 };
 \path (0:0)++(180:1)++(90:0.5) node[cyan] {$\ss \tl- m_{P}$};
 \path (0:0)++(180:1)++(90:2) node[rotate=90] {$\dots$};
 \path (0:0)++(180:1)++(90:3.5) node[cyan] {$\ss \tl-m_{3}$};
 \path (0:0)++(180:1)++(90:4.5) node[cyan] {$\ss \tl-m_{2}$};
  \path (0:0)++(180:1)++(90:5.5) node[cyan] {$\ss \tl-m_1$};
 \path (0:0)++(0:6)++(90:0.5) node[cyan] {$\ss \tl-k_{P}$};
 \path (0:0)++(0:6)++(90:2) node[rotate=90] {$\dots$};
 \path (0:0)++(0:6)++(90:3.5) node[cyan] {$\ss \tl-k_{3}$};
 \path (0:0)++(0:6)++(90:4.5) node[cyan] {$\ss \tl-k_{2}$};
  \path (0:0)++(0:6)++(90:5.5) node[cyan] {$\ss \tl-k_1$};
  \path (0:0)++(-90:0.5)++(0:0.5) node {$\ss \textcolor{red}{{n}}$};
\draw[fill=white] (0:0)++(0:1.5) circle (0.1);
\draw[fill=white] (0:0)++(0:2.5) circle (0.1);
\draw[fill=white] (0:0)++(0:3.5) circle (0.1);
\draw[fill=white] (0:0)++(0:4.5) circle (0.1);
  \draw[fill=white] (0:0)++(90:6)++(0:0.5) circle (0.1);
  \path (0:0)++(90:6.5)++(0:1.6) node[red] {$\ss l_{N}$};
  \path (0:0)++(90:6.5)++(0:2.6) node[red] {$\dots$};
  \path (0:0)++(90:6.5)++(0:3.6) node[red] {$\ss l_2$};
  \path (0:0)++(90:6.5)++(0:4.6) node[red] {$\ss l_1$};
 \end{tikzpicture}
\end{align}
where $\sum_{i}k_{i}=n$. As usual, it is convenient to redraw the lattice by complementing particles on vertical edges with their corresponding spins, as shown in the picture below. 
\begin{align}
\label{lattice:regionE_complemented}
 \begin{tikzpicture}[scale=0.6,baseline=(current bounding box.center)]
 \draw (0,0)--++(0:5)--++(90:6)--++(180:5)--++(-90:6);
  \foreach\x in {1,2,3,4,5}{
 \draw[gray] (0:\x)--++(90:6);
  \draw[gray] (90:\x)--++(0:5);
 };
  \foreach \x in {1,2,3,4,5,6}{
 \draw[->] (0:0)++(180:2)++(90:\x-0.5)--++(0:0.5);
 };
 \foreach \x in {1,2,3,4,5}{
 \draw[->] (0:0)++(-90:1.5)++(0:\x-0.5)--++(90:0.5);
 };
 \path (0:0)++(180:3.5)++(90:0.5) node {$\ss (q^{1-\tl}y^{-1}_{P},\tl)$};
 \path (0:0)++(180:3.5)++(90:2) node[rotate=90] {$\dots$};
 \path (0:0)++(180:3.5)++(90:3.5) node {$\ss (q^{1-\tl}y^{-1}_{3},\tl)$};
 \path (0:0)++(180:3.5)++(90:4.5) node {$\ss (q^{1-\tl}y^{-1}_{2},\tl)$};
  \path (0:0)++(180:3.5)++(90:5.5) node {$\ss (q^{1-\tl}y^{-1}_{1},\tl)$};
  \path (0:0)++(-90:2)++(0:0.5) node {$\ss (z_{0},\mathtt{M_{0}})$};
  \path (0:0)++(-90:2)++(0:3) node {$\dots$};
  \path (0:0)++(-90:2)++(0:4.6) node {$\ss (z_{1},\tm)$};
 \path (0:0)++(180:0.5)++(90:0.5) node[cyan] {$\ss  m_{P}$};
 \path (0:0)++(180:0.5)++(90:2) node[rotate=90] {$\dots$};
 \path (0:0)++(180:0.5)++(90:3.5) node[cyan] {$\ss m_{3}$};
 \path (0:0)++(180:0.5)++(90:4.5) node[cyan] {$\ss m_{2}$};
  \path (0:0)++(180:0.5)++(90:5.5) node[cyan] {$\ss m_1$};
 \path (0:0)++(0:5.5)++(90:0.5) node[cyan] {$\ss k_{P}$};
 \path (0:0)++(0:5.5)++(90:2) node[rotate=90] {$\dots$};
 \path (0:0)++(0:5.5)++(90:3.5) node[cyan] {$\ss k_{3}$};
 \path (0:0)++(0:5.5)++(90:4.5) node[cyan] {$\ss k_{2}$};
  \path (0:0)++(0:5.5)++(90:5.5) node[cyan] {$\ss k_1$};
  \path (0:0)++(-90:0.5)++(0:0.5) node {$\ss \textcolor{red}{{n}}$};
\draw[fill=white] (0:0)++(0:1.5) circle (0.1);
\draw[fill=white] (0:0)++(0:2.5) circle (0.1);
\draw[fill=white] (0:0)++(0:3.5) circle (0.1);
\draw[fill=white] (0:0)++(0:4.5) circle (0.1);
  \draw[fill=white] (0:0)++(90:6)++(0:0.5) circle (0.1);
  \path (0:0)++(90:6.5)++(0:1.6) node[red] {$\ss l_{N}$};
  \path (0:0)++(90:6.5)++(0:2.6) node[red] {$\dots$};
  \path (0:0)++(90:6.5)++(0:3.6) node[red] {$\ss l_2$};
  \path (0:0)++(90:6.5)++(0:4.6) node[red] {$\ss l_1$};
 \end{tikzpicture}
\end{align}
We now study the consequence of taking $z_{0}=0$. We prove in the ~\cref{spin:z0_limitinpuzzles} that at $z_{0}=0, q^{-\tm_{0}}=0$, the first column freezes with an overall weight of $(q;q)_{n}$. Then at $z_{0}=0$, $y_{i}=s$ and $z_{i}=s$, $q^{-\tl}=s^{2}$, $q^{-\tm_{0}}=0$ and $q^{-\tm}=s^{2}$, we have:
\begin{multline}
\mathcal{Z}_{\bl{E}}(s,\dots,s;{0},s,\dots,s;s^{2};0,s^{2})\hspace{2mm} 
\mathcal{Z}_{\bl{F}}(x_1,\dots,x_n;s,\dots,s;s^{2})\\
=(q;q)_{n}\sum_{\bk}
 \left(
 \begin{tikzpicture}[scale=0.6,baseline=(current bounding box.center)]
 \draw (0:1)--++(0:4)--++(90:6)--++(180:4)--++(-90:6);
  \foreach\x in {1,2,3,4,5}{
 \draw[gray] (0:\x)--++(90:6);
  \draw[gray] (90:\x)++(0:1)--++(0:4);
 };
 \path (0:0)++(180:0.5)++(90:0.5) node {$\ss \textcolor{cyan}{m_{P}}+ \textcolor{red}{m_{P}}$};
 \path (0:0)++(180:0.5)++(90:2) node[rotate=90] {$\dots$};
 \path (0:0)++(180:0.5)++(90:3.5) node {$\ss \textcolor{cyan}{m_{3}} + \textcolor{red}{m_{3}}$};
 \path (0:0)++(180:0.5)++(90:4.5) node {$\ss \textcolor{cyan}{m_{2}}+ \textcolor{red}{m_{2}}$};
  \path (0:0)++(180:0.5)++(90:5.5) node {$\ss \textcolor{cyan}{m_{1}}+ \textcolor{red}{m_{1}}$};
 \path (0:0)++(0:5.5)++(90:0.5) node[cyan] {$\ss k_{P}$};
 \path (0:0)++(0:5.5)++(90:2) node[rotate=90] {$\dots$};
 \path (0:0)++(0:5.5)++(90:3.5) node[cyan] {$\ss k_{3}$};
 \path (0:0)++(0:5.5)++(90:4.5) node[cyan] {$\ss k_{2}$};
  \path (0:0)++(0:5.5)++(90:5.5) node[cyan] {$\ss k_1$};
\draw[fill=white] (0:0)++(0:1.5) circle (0.1);
\draw[fill=white] (0:0)++(0:2.5) circle (0.1);
\draw[fill=white] (0:0)++(0:3.5) circle (0.1);
\draw[fill=white] (0:0)++(0:4.5) circle (0.1);
  \path (0:0)++(90:6.5)++(0:1.6) node[red] {$\ss l_{N}$};
  \path (0:0)++(90:6.5)++(0:2.6) node[red] {$\dots$};
  \path (0:0)++(90:6.5)++(0:3.6) node[red] {$\ss l_2$};
  \path (0:0)++(90:6.5)++(0:4.6) node[red] {$\ss l_1$};
 \end{tikzpicture}
 \right)(-s^{-1})^{\sum^{P}_{i=1}i k_{i}}\mathrm{F}_{\bk}(x_{1},\dots,x_{n};s)
\end{multline}

\subsection{Final equation}
The Yang--Baxter equation~\ref{eq:YBEhigerspin}, gives the following relation among the partition functions of various regions:
\begin{equation}
    \mathcal{Z}_{\bl{A}} \dfrac{\mathcal{Z}_{\bl{C}}}{\mathcal{Z}_{\bl{D}}}=\dfrac{\mathcal{Z}_{\bl{E}}}{\mathcal{Z}_{\bl{B}}}\mathcal{Z}_{\bl{F}}
\end{equation}
By plugging in the actual functions in the above relation, we get the following equation which governs the product of two functions:
\begin{multline}
\label{eq:hs-generalequation}
\mathrm{F}_{\bm}(x_1,\dots,x_n;s) \hspace{1mm}\mathrm{F}_{\bll}(x_1,\dots,x_{n};s)
=\\
\left( \prod^{N}_{i=1}\dfrac{(q;q)_{l_{i}}}{(s^2;q)_{l_{i}}}\right) \hspace{1mm}\sum_{\bk} (-s)^{\sum^{P}_{i=1} i(k_{i}-m_{i})} (-s)^{\sum^{N}_{i=1} i l_{i}}  \hspace{1mm}
\mathcal{C}^{\bk}_{\bll,\bm}(q,s) \hspace{1mm}\mathrm{F}_{\bk}(x_1,\dots,x_{n};s)  
\end{multline}
where 
\begin{equation}
\label{puzzles:hs-general}
\mathcal{C}^{\bk}_{\bll,\bm}(q,s) =
 \begin{tikzpicture}[scale=0.6,baseline=(current bounding box.center)]
 \draw (0:1)--++(0:4)--++(90:6)--++(180:4)--++(-90:6);
  \foreach\x in {1,2,3,4,5}{
 \draw[gray] (0:\x)--++(90:6);
  \draw[gray] (90:\x)++(0:1)--++(0:4);
 };
 \path (0:0)++(180:0.5)++(90:0.5) node {$\ss \textcolor{cyan}{m_{P}}+ \textcolor{red}{m_{P}}$};
 \path (0:0)++(180:0.5)++(90:2) node[rotate=90] {$\dots$};
 \path (0:0)++(180:0.5)++(90:3.5) node {$\ss \textcolor{cyan}{m_{3}} + \textcolor{red}{m_{3}}$};
 \path (0:0)++(180:0.5)++(90:4.5) node {$\ss \textcolor{cyan}{m_{2}}+ \textcolor{red}{m_{2}}$};
  \path (0:0)++(180:0.5)++(90:5.5) node {$\ss \textcolor{cyan}{m_{1}}+ \textcolor{red}{m_{1}}$};
 \path (0:0)++(0:5.5)++(90:0.5) node[cyan] {$\ss k_{P}$};
 \path (0:0)++(0:5.5)++(90:2) node[rotate=90] {$\dots$};
 \path (0:0)++(0:5.5)++(90:3.5) node[cyan] {$\ss k_{3}$};
 \path (0:0)++(0:5.5)++(90:4.5) node[cyan] {$\ss k_{2}$};
  \path (0:0)++(0:5.5)++(90:5.5) node[cyan] {$\ss k_1$};
\draw[fill=white] (0:0)++(0:1.5) circle (0.1);
\draw[fill=white] (0:0)++(0:2.5) circle (0.1);
\draw[fill=white] (0:0)++(0:3.5) circle (0.1);
\draw[fill=white] (0:0)++(0:4.5) circle (0.1);
  \path (0:0)++(90:6.5)++(0:1.6) node[red] {$\ss l_{N}$};
  \path (0:0)++(90:6.5)++(0:2.6) node[red] {$\dots$};
  \path (0:0)++(90:6.5)++(0:3.6) node[red] {$\ss l_2$};
  \path (0:0)++(90:6.5)++(0:4.6) node[red] {$\ss l_1$};
 \end{tikzpicture}
\end{equation}
By slightly adjusting the weights in region $\bl{E}$ of the above equation, we arrive at our Theorem ~\ref{theorem:Puzzles-Spin-Hall}.

\

\appendix
\section{\texorpdfstring{$z_{0}=0$ limit for the puzzles}{Text}}
\label{spin:z0_limitinpuzzles}
\begin{wrapfigure}{r}{0.25\textwidth}
\begin{center}
\begin{tikzpicture}[scale=0.6]
\draw[gray] (0:0)--++(0:1)--++(90:6)--++(180:1)--++(-90:6);
\foreach\x in {1,2,3,4,5,6}{
\draw[gray] (0,0)++(90:\x)--++(0:1);};
\draw[cyan,thick,stealth-](0,5.6)--(1,5.6);
\draw[cyan,thick,stealth-](0,4.6)--(1,4.6);
\draw[cyan,thick,stealth-](0,4.7)--(1,4.7);
\draw[cyan,thick,stealth-](0,2.6)--(1,2.6);
\draw[red,thick,-stealth,rounded corners](0.7,0)--(0.7,2.5)--(1,2.5);
\draw[red,thick,-stealth,rounded corners](0.6,0)--(0.6,4.4)--(1,4.4);
\draw[red,thick,-stealth,rounded corners](0.5,0)--(0.5,4.5)--(1,4.5);
\draw[red,thick,-stealth,rounded corners](0.4,0)--(0.4,5.5)--(1,5.5);
\end{tikzpicture}
\end{center}
\end{wrapfigure}
We argue that setting $z_{0}=q^{-\tm_{0}}=0$ in ~\ref{lattice:regionE_complemented} results in the first column freezing with weight $(q;q)_{n}$. To show this, we explicitly compute the weights for the type of vertices in the initial column of the puzzles. We prove that when a blue particle entering from the right turns upwards, the weight of such a vertex has the factor $z^{\alpha}_{0}$ for some $\alpha>0$. Subsequently, by setting $z_{0}=0$, the weight of such vertices vanishes. Therefore, the first column freezes, with all blue particles traversing across, see the picture on the right.

Consider the general vertex \begin{tikzpicture}[scale=0.8,baseline=(current bounding box.center))]
 \draw[gray] (0,0) rectangle (1,1); 
\node at (0.5,-0.3) {$\ss (a_1 ,a_2)$};
\node at (-0.65,0.5) {$\ss (b_1 ,b_2)$};
\node at (0.5,1.3) {$\ss (c_1 , c_2)$};
\node at (1.65,0.5) {$\ss (d_1 , d_2)$};
\end{tikzpicture}. For a vertex appearing in the first column, we have $b_{1}=0$. In ~(\ref{puzzle:higherspin-uncomplemented-blueparticles}), we see that the right edge of every vertex in the first column is entirely filled. We can then conclude that $d_1 =d_2$ for all the vertices in the first column.

For the vertex at the bottom of the first column, we have $a_{2}=0$. We show that the weight of the bottom vertex vanishes as $z_{0} = 0$ unless $c_{2}=0$. Using this fact, we can assume that $a_{2}=0$ for every other vertex in the first column. In conclusion, we need to show that the weight of the \begin{tikzpicture}[scale=0.8,baseline=-2pt]
 \draw[gray] (0,0) rectangle (1,1); 
\node at (0.5,-0.3) {$\ss (a ,0)$};
\node at (-0.5,0.5) {$\ss (0 ,b)$};
\node at (0.5,1.3) {$\ss (c_1 , c_2)$};
\node at (1.5,0.5) {$\ss (d , d)$};
\end{tikzpicture} vanishes as $z_{0}=0$ unless $c_2=0$. To further simplify the computations, we can take the limit $q^{-\tm_{0}}\mapsto 0$. This does not affect any other regions: $\bl{B}$ and $\bl{D}$ are normalised so that the weights of the vertices in the first column are $1$ and $\bl{C}$ remains unaffected which can be seen from the discussion in ~(\cref{lemma:regionC_partition function}). We take the weights from ~(\ref{generalweights}) which are then written in such way that their $q^{-\tl}$ and $q^{-\tm}$ dependence is explicit. We then take $q^{-\tm}\mapsto 0$ which gives us the following:

\begin{multline}
\mathcal{W}_{\tl,\tm}\left(
\begin{tikzpicture}[scale=0.8,baseline=-2pt]
 \draw[gray] (0,0) rectangle (1,1); 
\node at (0.5,-0.3) {$\ss (a,0)$};
\node at (-0.5,0.5) {$\ss (0 ,b)$};
\node at (0.5,1.3) {$\ss (c_1 , c_2)$};
\node at (1.5,0.5) {$\ss (d,d)$};
\draw[->] (0.5,-1) node[below] {$\ss z$}--(0.5,-0.7);
\draw[->] (-1.4,0.5) node[left] {$\ss {y^{-1}q^{-\tl+1}}$}--(-1.1,0.5);
\end{tikzpicture}\right)
=\mathbf{1}_{a=c_1+d} \mathbf{1}_{d=b+c_2}
q^{b}(-y)^{-c_2}
\dfrac{(q^{-1};q^{-1})_{b}} {(q;q)_{c_2}}\prod^{c_2}_{i=1}(q^{-\tl}-q^{i-d}) \\[1 em]
\binom{c_1 +d}{c_1 }_{q} q^{\frac{2 c_2-d+d^{2}}{2}}(-1)^{d+c_2} \sum^{c_2}_{ p_{2}=0 } (-1)^{-p_2} q^{\frac{(c_2 -p_2)(c_2 -p_2-1)}{2}}  
\frac{\prod^{p_2}_{i=1} (1-q^{-1}z y q^{i-1})}{\prod^{p_2 +d-c_2}_{i=1} (1-y zq^{i-1})}
\left(
 \dfrac{(q;q)_{c_2}}{(q;q)_{c_2-p_2}(q;q)_{p_2}}\right)\\[1em]
 \end{multline}
By applying the binomial theorem, we get:
\begin{multline}
=\mathbf{1}_{a=c_1+d} \mathbf{1}_{d=b+c_2}
q^{b}(-y)^{-c_2}
\dfrac{(q^{-1};q^{-1})_{b}} {(q;q)_{c_2}}\binom{c_1 +d}{c_1 }_{q} \prod^{c_2}_{i=1}(q^{-\tl}-q^{i-d}) \\[1 em]
q^{\frac{2 c_2-d+d^{2}}{2}}(-1)^{d+c_2} 
\dfrac{\left(\prod^{c2}_{i=1}q^{i-2}y z\right)}{(yz;q)_{d}}\dfrac{(q;q)_{d}}{(q;q)_{d-c_2}}
\end{multline}

\begin{multline}
=\mathbf{1}_{a=c_1+d} \mathbf{1}_{d=b+c_2}
\dfrac{(q^{-1};q^{-1})_{b}} {(q;q)_{c_2}}\binom{c_1 +d}{c_1 }_{q} \prod^{c_2}_{i=1}(q^{-\tl}-q^{i-d}) \\[1 em]
q^{\frac{2 c_2-d+d^{2}}{2} +b}(-1)^{d} 
z^{c_2}\dfrac{\left(\prod^{c2}_{i=1}q^{i-2}\right)}{(yz;q)_{d}}\dfrac{(q;q)_{d}}{(q;q)_{d-c_2}}
\end{multline}

Therefore, the weight of the vertices in the first column vanishes as $z_{0}=0$ unless $c_{2}=0$. When $c_{2}=0$, we get that $b=d$ and $c_1=a-d$ and the weights reduce to:

\begin{align}
=\mathbf{1}_{a_1=c_1+d} \mathbf{1}_{d=b}
\dfrac{(q;q)_{a}}{(q;q)_{a-b}}  
\dfrac{1}{(yz;q)_{b}}
\end{align}

As we set $z_{0}$ the weight of the first column, through telescoping, is equal to $(q;q)_{n}$.

\bibliographystyle{alpha}
\bibliography{spinlr}{}
\end{document}